\documentclass[12pt]{article}

\usepackage{amsfonts}
\usepackage{amssymb}
\usepackage{amsmath}
\usepackage{mathrsfs}
\usepackage[colorlinks, linkcolor=blue, anchorcolor=blue, citecolor=blue]{hyperref}
\usepackage[normalem]{ulem}
\usepackage{color}
\usepackage{leftidx}
\usepackage{verbatim}
\usepackage{enumerate}
\usepackage{amsthm}
\usepackage{geometry}
\usepackage{MnSymbol}
\usepackage{authblk}

\newcommand{\R}{\mathbb{R}}
\newcommand{\N}{\mathbb{N}}

\newcommand{\X}{\mathbb{X}}
\newcommand{\bX}{\mathbf{X}}
\newcommand{\tX}{\widetilde{X}}
\newcommand{\bx}{\mathbf{x}}
\newcommand{\by}{\mathbf{y}}
\newcommand{\bY}{\mathbf{Y}}
\newcommand{\cY}{\mathcal{Y}}
\newcommand{\tY}{\widetilde{Y}}
\newcommand{\dtY}{\dot{\widetilde{Y}}}
\newcommand{\dY}{\dot{Y}}

\newcommand{\dbY}{\dot{\mathbf{Y}}}

\newcommand{\tbY}{\widetilde{\mathbf{Y}}}

\newcommand{\hx}{\widehat{x}}

\newcommand{\hY}{\widehat{Y}}

\newcommand{\Y}{\mathbb{Y}}

\newcommand{\dZ}{\dot{Z}}
\newcommand{\tZ}{\widetilde{Z}}

\newcommand{\bz}{\mathbf{z}}
\newcommand{\W}{\mathbb{W}}
\newcommand{\bW}{\mathbf{W}}
\newcommand{\cW}{\mathcal{W}}
\newcommand{\cC}{\mathcal{C}}
\newcommand{\cL}{\mathcal{L}}
\newcommand{\kL}{\mathfrak{L}}
\newcommand{\kB}{\mathfrak{B}}

\newcommand{\C}{\mathscr{C}}
\newcommand{\tC}{\widetilde{\mathscr{C}}}

\newcommand{\cJ}{\mathcal{J}}

\newcommand{\tW}{\widetilde{W}}
\newcommand{\tbW}{\widetilde{\bW}}
\newcommand{\tdW}{\widetilde{\W}}
\newcommand{\tXi}{\widetilde{\Xi}}
\newcommand{\bxi}{\pmb{\xi}}
\newcommand{\cD}{\mathcal{D}}
\newcommand{\sD}{\mathscr{D}}
\newcommand{\sE}{\mathscr{E}}

\newcommand{\sG}{\mathscr{G}}
\newcommand{\tsD}{\widetilde{\mathscr{D}}}
\newcommand{\tsG}{\widetilde{\mathscr{G}}}

\newcommand{\tsE}{\widetilde{\mathscr{E}}}

\newcommand{\sR}{\mathscr{R}}

\newcommand{\bbeta}{\pmb{\beta}}
\newcommand{\tbbeta}{\widetilde{\pmb{\beta}}}

\newcommand{\tbeta}{\widetilde{\beta}}

\newcommand{\txi}{\widetilde{\xi}}
\newcommand{\vrho}{\varrho}

\newcommand{\bV}{\overline{V}}
\newcommand{\tV}{\widetilde{V}}

\newcommand{\bU}{\overline{U}}

\newcommand{\lgl}{\langle}
\newcommand{\rgl}{\rangle}
\newcommand{\llg}{\llangle}
\newcommand{\rrg}{\rrangle}
\newcommand{\cI}{\mathcal{I}}
\newcommand{\tc}{\widetilde{c}}

\newtheorem{theorem}{Theorem}[section]
\newtheorem{proposition}[theorem]{Proposition}
\newtheorem{lemma}[theorem]{Lemma}
\newtheorem{corollary}[theorem]{Corollary}
\newtheorem{definition}[theorem]{Definition}
\newtheorem{remark}[theorem]{Remark}

\newtheorem{innerhyp}{Hypothesis}
\newenvironment{hyp}[1]
 {\renewcommand\theinnerhyp{#1}\innerhyp}
 {\endinnerhyp}

\begin{document}
\numberwithin{equation}{section}

\title{On nonlinear rough paths}

\author{David Nualart\thanks{David Nualart was supported by the NSF grant DMS 1811181} \footnote{Email: nualart@ku.edu}}
\author{Panqiu Xia \footnote{Email: pqxia@ku.edu}}

\affil{\small Department of Mathematics, University of Kansas, Lawrence, KS, 66044, USA.}
\date{}
\maketitle

\begin{abstract}	
In this paper, we establish the theory of nonlinear rough paths. We give the definition of nonlinear rough paths, and develop the integrals. Then, we study differential equations driven by nonlinear rough paths. Afterwards, we compare the nonlinear rough paths and classic theory of linear rough paths. Finally, we apply this theory to rough partial differential equations.
\end{abstract}

\section{Introduction}\label{s.1}
Nonlinear integrals in the sense of Young were introduced by
 Hu and L\^{e} in \cite{tams-17-hu-le}. In this paper, the authors consider the following nonlinear integral
\begin{equation} \label{a1}
\cI_{s,t}=\int_s^t W(dr, Y_r),
\end{equation}
where $W$ is a function on $[0,T]\times \R^d$ with values in $\R^d$, that is $\tau$-H\"{o}lder continuous in time and $\lambda$-H\"{o}lder continuous in space, and $Y: [0,T]\to \R^d$ is $\gamma$-H\"{o}lder continuous. Under the assumption $\tau+\lambda\gamma>1$, the nonlinear integral (\ref{a1}) is well-defined in the sense of Young (see Young \cite{am-36-young}). That is, $\cI_{s,t}$ is the limit of the following linear approximations as $|\pi|\to 0$
\[
\sum_{k=1}^nW_{t_{k-1},t_k}( Y_{k-1}):=\sum_{k=1}^n\big[W(t_k, Y_{k-1})-W(t_{k-1}, Y_{k-1})\big],
\]
where $\pi=(s=t_0<t_1<\dots<t_n=t)$ is a partition of the interval $[s,t]$. As an example, one can define a pathwise nonlinear integral of the form (\ref{a1}), where $W$ is a fractional Brownian sheet with Hurst parameters $H_0\in (\frac{1}{2}, 1)$ in time and $H_1=\cdots =H_d=1$ in space, and $Y$ is a $d$-dimensional standard Brownian motion that is independent of $W$. By applying this theory of nonlinear Young's integrals, Hu and L\^{e} studied the following transport equation with distributional vector field:
\begin{align}\label{tpe0}
\frac{\partial}{\partial t}u(t,x)+D u(t,x)\frac{\partial}{\partial t}W(t,x)=0,
\end{align}
where $D$ denotes the spatial derivative operator. The existence and uniqueness of the solution to (\ref{tpe0}) with $\cC^{1+\lambda_0}_{loc}(\R^d;\R^d)$-valued initial condition were proved in this paper assuming that $(1+\lambda_0)\tau>1$. They also provided a formula for the solution:
\begin{align}\label{tpf}
u(t,x)=h(Z_t(x)),
\end{align}
 where $h$ is the initial condition, and $Z_t$ is the inverse of $Y_t$, where $Y$ is the solution to the following nonlinear differential equation:
\begin{align}\label{nde0}
Y_t(x)=x+\int_0^t W(ds, Y_s(x)).
\end{align}
On the other hand, applying the theory of nonlinear integrals to the stochastic heat equation, Hu and L\^{e} also gave a pathwise proof of the Feynman-Kac formula,
which provides an alternative method to study this topic (see e.g. \cite{ap-12-hu-lu-nualart,ap-11-hu-nualart-song} for a probabilistic approach).

The purpose of this paper is to extend the theory of nonlinear integrals to the case when the functions $W$ and $Y$ are rougher, that is $\tau+\lambda\gamma<1$.
In this situation, Young's approach fails.
 The following example, inspired by the lecture notes from Zanco (see Example 3.6 of \cite{ln-16-zanco}), provides a non-standard nonlinear rough path behavior in $\R$. For any $n\in \N$, $t\in[0,T]$ and $x,y\in\R^d$, we define
\begin{align*}
F(x,y)=e^{xy},\ X^{(n)}_t=\frac{1}{n}\cos(2\pi n^2 t)\ \mathrm{and}\ Y^{(n)}_t=\frac{1}{n}\sin(2\pi n^2 t).
\end{align*}
Then $F(X^{(n)}_t, y)$ converges to $1$ and $Y^{(n)}_t$ converges to $0$ uniformly on compact sets as $n\to \infty$. On the other hand, however, the following integral
\begin{align*}
\int_0^1 F(dX^{(n)}_t, Y^{(n)}_t)=&-\frac{1}{4}\int_0^{4\pi}\exp\Big(\frac{1}{2n^2}\sin(s)\Big)ds\to -\pi,
\end{align*}
by dominated convergence theorem, as $n\to\infty$. 

In the linear situation, a useful tool to deal with the integration of rough functions is the theory of rough paths. This theory has been developed by the pioneering work of Lyons since the early nineties (see e.g. Lyons \cite{mrl-94-lyons,rmi-98-lyons}) to study $d$-dimensional dynamic systems of the form 
\[
d Y_t= f(Y_t)dX_t,\ t\in [0,T],
\]
 where the driven signal $X_t$ is $\alpha$-H\"{o}lder continuous and $\alpha\in (0,\frac{1}{2}]$. The main idea of the rough path analysis is as follows. Let $p=\lfloor \frac{1}{\alpha}\rfloor$, and let $T^{(p)}$ be a $p$-step truncated tensor algebra given by the expression
 \[
 T^{(p)}:=\R\oplus (\R^d)\oplus (\R^d)^{\otimes 2}\oplus \cdots \oplus (\R^d)^{\otimes p}.
 \]
 The rough path associated to $X$ is a lifting of $X$ to a $T^{(p)}$-valued function on $[0,T]^2$, denoted by $S^{(p)}(X)$, in such a way that  when $X$ is piecewise differentiable, $S^{(p)}_{s,t}=(1, X^1_{s,t}, X^2_{s,t},\dots, X^p_{s,t})$, and each component $X_{s,t}^i$ is the $i$th iterated integral of $X$ on the time interval $[s,t]\subset [0,T]$. Suppose that $f$ is a smooth function, then the integral of $f(X)$ against $X$ on $[s,t]$ can be approximated by
\begin{align}\label{1fm}
\int_s^tf(X_r)dX_r\approx f(X_s)X^1_{s,t}+f'(X_s)X^2_{s,t}+\cdots +f^{(p-1)}(Z_s)X^{p}_{s,t},
\end{align}
with an error of order $O(|t-s|^{(p+1)\alpha})$. This allows us to define the integral by passing the limit as $|\pi|\to 0$ of the following expression
\[
\sum_{k=1}^n\sum_{i=0}^pf^{(i-1)}(X_{t_{k-1}})X^i_{t_{k-1},t_k},
\]
where $\pi=(s=t_1<\dots<t_n=t)$.

Suppose that $\alpha\in (\frac{1}{3},\frac{1}{2}]$. Gubinelli (see \cite{jfa-04-gubinelli}) generalized the integration of ``1-forms'', which means the integrand is a function
$f(X_t)$ of the driving signal, to a class of rough functions called ``controlled rough paths''. A controlled rough path (by $X$), is a function $Y: [0,T]\to \R^d$ whose increment on an interval $[s,t]$ can be written in the following way: $Y_{s,t}=Y'_s X_{s,t}+R^Y_{s,t}$, for some $\R^d\otimes \R^d$-valued $\alpha$-H\"{o}dler continuous function $Y'$ and some $\R^d$-valued $2\alpha$-H\"{o}lder continuous function $R^Y$. In this case, the approximation of the integral is the following
\[
\int_s^t Y_r dX_r\approx Y_sX^1_{s,t}+ Y_s'X^2_{s,t}.
\]
For a more detailed account of this topic, we refer the readers to the books of Friz and Hairer \cite{springer-14-friz-hairer} and Lyons and Qian \cite{oxford-02-lyons-qian}.
An alternative approach to deal with the integration of ``non-1-forms" based on fractional calculus was developed in \cite{sd-11-besalu-nualart,tams-09-hu-nualart}.

In the present paper, we will extend the nonlinear Young's integral to the rough case by using Gubinelli's approach, and assuming a H\"older regularity of order $\alpha\in (\frac{1}{3},\frac{1}{2}]$. The paper is organized in the following way.
 In Section \ref{s.2} we give brief review of the preliminaries about (linear) rough paths.
In Section \ref{s.3} we introduce a nonlinear variant of rough paths. By definition a nonlinear rough path is a pair $(W ,\W)$ such that $W(t,x)$ is a function of two variables, $(t,x) \in [0,T] \times V$, where $V$ is Banach space and $\W_{s,t}(x,y)$ should be interpreted as the double integral 
\[
\int_s^t DW(dr,y)(W(t,x)-W(r,x)),
\]
for any $0\le s\le t\le T$ and $(W ,\W)$ 
 satisfies certain properties, including $\alpha$-H\"older continuity and a version of Chen's relation. Then, a nonlinear rough integral can be approximated in the following way:
\[
\int_s^t W(dr, Y_r)\approx W_{s,t}(Y)+\W_{s,t}(\dY_s,Y_s),
\]
where $\dY$ is the Gubinelli derivative of $Y$ in the context of nonlinear rough paths. We prove that the nonlinear rough integral is a nonlinear controlled rough path and
we establish some properties of nonlinear rough integrals.

In Section \ref{s.4}, we consider the following rough differential equation (RDE):
\begin{align}\label{rde0}
Y_t=\xi+\int_0^t W(dr, Y_r),
\end{align}
where $(W,\W)$ is an $\alpha$-H\"{o}lder nonlinear rough path. Local and global existence and uniqueness of the solution to the RDE (\ref{rde0}) is proved in this section. We also obtain some estimates of the solution to this equation.

In Section \ref{s.5}, we compare the linear and nonlinear rough paths from two points of view. In Section \ref{s.5.1}, we consider a special class of nonlinear rough paths, that is a composition of a ``nice'' function and a linear rough path. In Section \ref{s.5.2}, a nonlinear rough path is treated as a function space-valued linear rough path. In Section \ref{s.5.3}, we provide a generalized It\^{o} type formula for (nonlinear) controlled rough paths.

As an application to the theory of nonlinear rough paths, in Section \ref{s.6} we analyze the gradient flow of the following equation with spatial parameter,
\[
Y_t(x)=x+\int_0^tW(dr,Y_r(x)),
\]
where $x\in \R^d$ and $W:[0,T]\times \R^d\to \R^d$ is a nonlinear rough path. We will prove that under some assumptions, $Y_t(x)$ is differentiable in $x$. In addition for every $(t,x)\in [0,T]\times \R^d$, the gradient $D Y_t(x)$ is an invertible matrix. Thus, there exists $Z:[0,T]\times \R^d\to \R^d$ such that $Z_t(Y_t(x))=Y_t(Z_t(x))=x$ for all $(t,x)\in [0,T]\times \R^d$. Assume that $h\in \cC^4_{loc}(\R^d;\R^d)$. Because of the rougher structure of $W$ than Young's case, it turns out that $h(Z_t(x))$ doesn't satisfies the transport equation (\ref{tpe0}). We will prove that $h(Z_t(x))$ is indeed the solution to the following rough partial differential equation (RPDE):
\begin{align}\label{rpde}
\frac{\partial}{\partial t}u(t,x)&+D u(t,x)\frac{\partial W(t,x)}{\partial t}=\frac{1}{2} D u(t,x)\frac{\partial \llg DW(x), W(x) \rrg_{0,t}}{\partial t}\nonumber\\
&+\frac{1}{2}D u(t,x)\frac{\partial \llg W(x), DW(x) \rrg_{0,t}}{\partial t}+\frac{1}{2} D^2 u(t,x)\frac{\partial \lgl W(x)\rgl_{0,t}}{\partial t}.
\end{align}
Furthermore, the solution is unique in the space $\cC^{\alpha,3}_{loc}([0,T]\times \R^d;\R)$.

\section{Preliminaries}\label{s.2}
 Fix a time interval $[0,T]$. Assume that $\alpha\in (\frac{1}{3}, \frac{1}{2}]$. Let $V$ and $K$ be separable Banach spaces. We follow the construction of Friz and Hairer \cite{springer-14-friz-hairer} to introduce the basic framework of the theory of (linear) rough paths.
\begin{definition}
\begin{enumerate}[(i)]
		\item $\cC^{\alpha}([0,T];V)$ is the space of functions on $[0,T]$ taking values in $V$ such that the following $\alpha$-H\"{o}lder seminorm is finite
		\begin{align}
	\|\Phi\|_{\alpha}:=\sup_{s\neq t\in [0,T]}\frac{\|\Phi_{s,t}\|_V}{|t-s|^{\alpha}},
		\end{align}
		where $\Phi_{s,t}:=\Phi_t-\Phi_s$.

	\item $\cC^{\alpha}_2([0,T]^2; V\otimes V)$ is the space of functions on $[0,T]^2$ taking values in the tensor product space $V\otimes V$ and such that the following $\alpha$-H\"{o}lder seminorm is finite
	 \begin{align}
	\|\Psi\|_{\alpha}:=\sup_{s\neq t\in [0,T]}\frac{\|\Psi_{s,t}\|_{V\otimes V}}{|t-s|^{\alpha}},
	\end{align} 
	where the norm $\|\cdot\|_{V\otimes V}$ is the projective norm on tensor product spaces.
	\end{enumerate}
\end{definition}

A $V$-valued rough path, introduced below, is defined as a pair of a rough function and a double integral term. 

\begin{definition}\label{dlrp}
The space of rough paths $\C^{\alpha}([0,T];V)$ is the collection of pairs ${\bf X}=(X,\X)$ satisfying the following properties:
\begin{enumerate}[(i)]
	\item $X\in \cC^{\alpha}([0,T];V)$.
	\item $\X\in \cC^{2\alpha}_2([0,T]^2; V\otimes V)$.
\item $(X,\X)$ satisfies Chen's relation: for all $(s,u,t)\in[0,T]^3$,
\begin{align}\label{lchen}
\X_{s,t}-\X_{s,u}-\X_{u,t}=X_{s,u}\otimes X_{u,t}.
\end{align}
\end{enumerate}
\end{definition}
Here $\X$ has to be interpreted as a version of the following double integral:
\[
\int_s^tX_{s,r} \otimes dX_r:=\X_{s,t}.
\]

Let $X\in \cC^{\alpha}([0,T];V)$. We define controlled rough paths by $X$ as follows:
\begin{definition}\label{lcrp}
Let $X\in \cC^{\alpha}([0,T];V)$.
An element $Y\in \cC^{\alpha}([0,T]; K)$ is said to be controlled by $X$, if there exist functions $Y'\in \cC^{\alpha}([0,T];\cL(V;K))$ and $R^Y\in \cC^{2\alpha}_2([0,T]^2; K)$, such that
\begin{align*}
Y_{s,t}=Y'_s(X_{s,t})+R^Y_{s,t}
\end{align*}
for any $s,t\in[0,T]$. Here $\cL(V;K)$ denotes the space of continuous linear operators from $V$ to $K$ equipped with the operator norm. The function $Y'$ is called the Gubinelli derivative of $Y$.
\end{definition}
Denote by $\sD_{X}^{2\alpha}(K)$ the space of such pairs $(Y,Y')$. With an abuse of notations, we sometimes write $Y\in \sD_X^{\alpha}(K)$ instead of $(Y,Y')\in \sD_X^{\alpha}(K)$. 

Let $K_1$, $K_2$ be separable Banach spaces. For any positive integer $k$, denote by $\cC^k_{loc}(K_1;K_2)$ the space of continuous functions on $K_1$ with values in $K_2$ that are locally bounded and have locally bounded Fr\'{e}chet derivatives up to order $k$. The next lemma shows that composition of a function in $\cC^2_{loc}(K_1;K_2)$ and a $K_1$-valued controlled rough path is still a controlled rough path.
\begin{lemma}[Lemma 7.3 of Friz and Hairer \cite{springer-14-friz-hairer}]\label{cmpcr}
Let $X\in \cC^{\alpha}([0,T]; V)$, $(Y,Y')\in \sD_X^{2\alpha}(K_1)$, and $\phi\in \cC^2_{loc}(K_1;K_2)$. Then, $\phi(Y)$ is controlled by $X$. More precisely, we have $(\phi(Y), D\phi(Y)Y')\in \sD_X^{2\alpha}(K_2)$.
\end{lemma}

 Suppose that $X\in \cC^{\alpha}([0,T];V)$ and $(Y,Y')\in \sD_X^{2\alpha}(\cL(V;K))$. Then, $Y'$ takes values in $\mathcal{L}(V; \mathcal{L}(V;K))$, which can be identified with $\mathcal{L} (V\otimes V;K)$. The next theorem defines a version of  the (linear) rough integral.

\begin{theorem}[Theorem 4.10 (a) of Friz and Hairer \cite{springer-14-friz-hairer}]\label{tlri}
Let $\bX=(X,\X)\in \C^{\alpha}([0,T];V)$. Suppose that $(Y,Y')\in \sD_X^{2\alpha}(\cL(V;K))$. Then the following ``compensated Riemann-Stieltjes sum''
\begin{align}\label{lri}
\sum_{k=1}^n\Xi_{t_k,t_{k-1}}:=\sum_{k=1}^n\big[Y_{t_{k-1}}(X_{t_{k-1},t_k})+Y_{t_{k-1}}'(\X_{t_{k-1},t_k})\big],
\end{align} 
converges as $|\pi|\to 0$, where $\pi=(s=t_1<t_2<\dots t_n=t)$.  Denote by $\cJ_{s,t}(\Xi)$ the limit of \eqref{lri}. Then, $\cJ_{s,t}(\Xi)$ is additive, that is $\cJ_{s,t}(\Xi)=\cJ_{s,u}(\Xi)+\cJ_{u,t}(\Xi)$ for any $(s,u,t)\in [0,T]^3$.  Moreover, the following estimate is satisfied for all $0\leq s\leq t\leq T$:
\begin{align}\label{blri}
\big\|\cJ_{s,t}(\Xi)-\Xi_{s,t}\big\|_K\leq k_{\alpha}(\|X\|_{\alpha}\|R^{Y}\|_{2\alpha}+\|\X\|_{2\alpha}\|Y'\|_{\alpha})|t-s|^{3\alpha},
\end{align}
where 
\begin{align}\label{kalpha}
k_{\alpha}=(1-2^{1-3\alpha})^{-1}.
\end{align}
  By definition, the rough integral of $Y$ against $\bX=(X,\X)$ is defined as follows,
\begin{align}
\int_s^tY_rd\bX_r:=\cJ_{s,t}(\Xi),
\end{align}
for all $(s,t)\in [0,T]^2$.
\end{theorem}

Theorem \ref{tlri} can be proved by using the following sewing lemma. In this case, $\gamma=3\alpha>1$ and $k_{\alpha}$ comes from inequality (\ref{sewct}) below. The sewing lemma will also be used later in the theory of nonlinear rough paths.
\begin{lemma}[Lemma 2.1 of Feyel and De la Pradelle \cite{ejp-06-feyel-delapradelle}]\label{sew}
Let $\beta\in (0,1]$, and let $\Xi\in \cC^{\beta}_2([0,T]^2; K)$. Suppose there exist $C>0$ and $\gamma>1$ such that the following inequality holds:
\[
\|\delta\Xi(s,u,t)\|_K:=\|\Xi_{s,t}-\Xi_{s,u}-\Xi_{u,t}\|_K\leq C|t-s|^{\gamma},
\] 
for any $0\leq s\leq u\leq t\leq T$. Then there exists a unique (up to an additive constant) function $\mathcal{J}(\Xi)\in \cC^{\beta}([0,T];V)$,
 such that the following inequality holds
\begin{align}\label{sewct}
\|\mathcal{J}_{s,t}(\Xi)-\Xi_{s,t}\|_K=\|\mathcal{J}_t(\Xi)-\mathcal{J}_s(\Xi)-\Xi_{s,t}\|_K\leq (1-2^{1-\gamma})^{-1}C|t-s|^{\gamma}.
\end{align}
Moreover, $\mathcal{J}_{s,t}(\Xi)$ can be represented as follows,
\begin{align}
\mathcal{J}_{s,t}(\Xi)=\lim_{|\pi|\to 0}\sum_{k=1}^n\Xi_{t_{k-1},t_k},
\end{align}
where $\pi=(s=t_0<t_1<\cdots<t_n=t)$ and the limit is independent of the choice of $\pi$.
\end{lemma}

The next proposition shows that the rough integral is controlled by $X$.
\begin{proposition}[Theorem 4.10 (b) of Friz and Hairer \cite{springer-14-friz-hairer}]\label{lricrp}
Suppose that $(X,\X)\in \C^{\alpha}([0,T]; V)$ and $(Y,Y')\in \sD_X^{2\alpha}(\cL(V;K))$. Let
\[
Z_t=\int_0^t Y_r d\bX_r.
\]
  Then, $Z$ is an $\alpha$-H\"{o}lder continuous function taking values in $K$. Moreover $Z$ is controlled by $X$ with $Y$ as a Gubinelli derivative.
\end{proposition}
\begin{remark}
Controlled rough paths play a role similar to that of adapted (to the natural filtration) semimartingles in the It\^{o} calculus. The corresponding Doob-Meyer's decomposition theorem still holds in the context of rough paths, if $X$ is ``truly'' rough. In this case, the Gubinelli derivative of a controlled rough path of $X$ is unique (see Chapter 6 of Friz and Hairer \cite{springer-14-friz-hairer}).
\end{remark}

In the next proposition, we define the integration of two controlled rough paths.
\begin{proposition}\label{lritcrp}
 Let $V$, $K_1$ and $K_2$ be separable Banach spaces. Suppose that $\bX=(X, \mathbb{X})\in \C^{\alpha}([0,T]; V)$ and $(Y, Y')\in \sD_X^{2\alpha}(K_1)$.
\begin{enumerate}[(i)]
\item {[Remark 4.11 of Friz and Hairer \cite{springer-14-friz-hairer}]} Suppose that $(Z,Z')\in \sD_X^{2\alpha}(K_2)$. The following limit exists
\begin{align}\label{itgcrp}
 \lim_{|\pi|\to 0}\sum_{k=1}^n \big[Z_{t_{k-1}}\otimes Y_{t_{k-1},t_k}+( Z'_{t_{k-1}}\otimes Y'_{t_{k-1}} )(\X_{t_{k-1},t_k})\big],
\end{align}
where $\pi=(s=t_0<t_1<\cdots<t_n=t)$ and defines the integral $\int_s^t Z_r\otimes dY_r$.
\item {[Proposition 7.1 of Friz and Hairer \cite{springer-14-friz-hairer}]} Let $\Y:[0,T]^2\to K_1\otimes K_1$ be given by
\begin{align}\label{dyst}
\Y_{s,t}=\int_s^tY_r\otimes dY_r-Y_s\otimes Y_{s,t},
\end{align}
and the integral in \eqref{dyst} is defined by \eqref{itgcrp}. Then, $\bY:=(Y,\Y)$ is a rough path. Suppose that $(Z,\tZ')\in \sD_Y^{2\alpha}(K_2)$. Let $Z'_t=\tZ'_tY'_t$ for all $t\in [0,T]$.
Then, $(Z,Z')\in \sD_X^{2\alpha}(K_2)$. In addition, the following equality holds
\begin{align}
\int_s^t Z_r\otimes d\bY_r=\int_s^t Z_r\otimes dY_r,
\end{align}
where the integral on the left-hand side is in the sense of Theorem \ref{tlri}, and the integral on the right -hand side is in the sense of \eqref{itgcrp}.
\end{enumerate}
\end{proposition}

\begin{remark}
Assume the conditions of Proposition \ref{lritcrp} (i) where $K_2=\cL(K_1; K)$. Then,
\begin{align}\label{itgcrpl}
\int_s^t Z_r dY_r:= \lim_{|\pi|\to 0}\sum_{k=1}^n \big[Z_{t_{k-1}}(Y_{t_{k-1},t_k})+ (Z'_{t_{k-1}}Y'_{t_{k-1}}) \X_{t_{k-1},t_k}\big],
\end{align}
and
\begin{align}\label{itgcrpr}
\int_s^t dZ_r (Y_r):= \lim_{|\pi|\to 0}\sum_{k=1}^n \big[Z_{t_{k-1}, t_k}(Y_{t_{k-1}})+ (Z'_{t_{k-1}}Y'_{t_{k-1}}) \X_{t_{k-1},t_k}^*\big],
\end{align}
are well-defined, where $\pi=(s=t_0<t_1<\cdots<t_n=t)$, $(Z_t'Y_t'): V\otimes V\to K$ is given by
\[
(Z_t'Y_t')(x, y)=Z_t'(x)\big[Y_t'(y)\big].
\]
and $*$ denoted the transpose operator on the tensor product space $V\otimes V$.
\end{remark}

In order to deduce It\^{o}'s lemma for controlled rough paths of $X$, we need to introduce the following quadratic compensator. It plays a similar role as the quadratic variation in the It\^{o} calculus.
\begin{definition}\label{qdcp}
Let $\bX=(X,\X)\in\C^{\alpha}([0,T];V)$. Suppose that $(Y, Y')\in \sD_X^{2\alpha}(K_1)$ and $(Z, Z')\in \sD_X^{2\alpha}(K_2)$ respectively.
\begin{enumerate}[(i)]
\item The quadratic compensator $\lgl X\rgl$ is a function on $[0,T]^2$ with values in $V\otimes V$ given by
\begin{align}
\lgl X \rgl_{s,t}:=X_{s,t}\otimes X_{s,t}-2\X_{s,t}.
\end{align}
\item The quadratic compensator $\lgl Z, Y\rgl: [0,T]^2\to K_2\otimes K_1$ is given by
\begin{align}\label{qdcp1}
\lgl Z, Y \rgl_{s,t}:=Z_{s,t}\otimes Y_{s,t}-2\int_s^t Z_{s,r}\otimes dY_r.
\end{align}
\end{enumerate}
\end{definition}
\begin{remark} \label{rqdcp}
\begin{enumerate}[(i)]
 \item The name ``quadratic compensator'' comes from Keller and Zhang (see (2.7) of \cite{spa-16-keller-zhang}). Let $V=\R^d$, then,
\[
\frac{1}{2}\big(\lgl X\rgl_{s,t}+\lgl X \rgl^*_{s,t}\big)=\lgl \tX\rgl_{s,t},
\] 
where $\lgl \tX\rgl_{s,t}$ denotes the quadratic compensator of $X$ in the sense of Keller and Zhang. The transpose term in our setting is involved in the derivative when applying It\^{o}'s lemma. For example, let $f(x)=x\otimes x$ . Then $D^2f(x)(\bz)=\bz+\bz^*$ for all $\bz\in V\otimes V$.

\item Similar as the quadratic variation of It\^{o} processes, the following equality holds:
\begin{align}\label{itoiso}
\lgl Y,Z \rgl_{s,t}=\int_s^t Y'_r\otimes Z'_r d\lgl X\rgl_r.
\end{align}

\item Suppose that $K_2=\cL(K_1; K)$, we write
\begin{align}\label{qdcp2}
\llg Z, Y\rrg_{s,t}:=Z_{s,t}Y_{s,t}-2\int_s^t Z_{s,r} dY_r
\end{align}
and
\begin{align}\label{qdcp3}
\llg Y, Z\rrg_{s,t}:=Z_{s,t}Y_{s,t}-2\int_s^t dZ_r (Y_{s,r}).
\end{align}

\item It is easy to verify that $\lgl X\rgl\in \cC^{2\alpha}_2([0,T]; V\otimes V)$. Similarly, $\lgl Y, Z\rgl$, $\lgl Z, Y\rgl$, $\llg Y, Z\rrg$ and $\llg Z, Y\rrg$ are also $2\alpha$-H\"{o}lder continuous in corresponding spaces.
\end{enumerate}
\end{remark}
The next lemma is It\^{o}'s formula for (linear) rough paths. The proof is quite elementary (see e.g. Theorem 3.4 of Keller and Zhang \cite{spa-16-keller-zhang} for finite-dimensional cases), we omit it here.
\begin{lemma}\label{itolrp}
Let $\bX=(X,\X)\in\C^{\alpha}([0,T];V)$. Suppose that $(Y, Y')\in \sD_X^{2\alpha}(K_1)$ and $(Z, Z')\in \sD_X^{2\alpha}(K_2)$ respectively. Let $f\in \cC^3_{loc}(K_1\times K_2;K)$. Then, the following equality holds:
\begin{align}\label{ito}
f(Y_t, Z_t)-f(Y_s, Z_s)=&\int_s^tD_1f(Y_r, Z_r)d Y_r+\int_s^tD_2f(Y_r, Z_r)d Z_r\\
&+\frac{1}{2}\Big[\int_s^tD_{11}f(Y_r, Z_r) d\lgl Y\rgl_r+\int_s^tD_{12}f(Y_r, Z_r) d\lgl Y, Z\rgl_r\Big]\nonumber\\
&+\frac{1}{2}\Big[\int_s^tD_{21}f(Y_r, Z_r) d\lgl Z, Y\rgl_r+\int_s^tD_{22}f(Y_r, Z_r) d\lgl Z\rgl_r\Big]\nonumber,
\end{align}
where the first two integrals are defined in \eqref{itgcrpl} and last four integrals are Young's integrals.
\end{lemma}

Let $\bX=(X,\X)\in \C^{\alpha}([0,T]; V)$, and let $f:K\to \cL(V;K)$. Consider the following RDE:
\begin{align}\label{lrde}
Y_t=y+\int_0^t f(Y_r)d\bX_r.
\end{align}

\begin{definition}
An $\alpha$-H\"{o}lder continuous function $Y$ is said to be a solution to (\ref{lrde}), if the following properties are satisfied:
\begin{enumerate}[(i)]
\item $(Y, f(Y))\in \sD_X^{2\alpha}(K)$ and $f(Y)\in \sD_X^{2\alpha}(\cL(V;K))$.
\item Equality (\ref{lrde}) holds for all $t\in [0,T]$, where the integral on the right-hand side is a rough integral in the sense of Theorem \ref{tlri}.
\end{enumerate}
\end{definition}
This equation has been intensively studied in the literatures (see e.g. \cite{amre-08-davie,jde-08-friz-victoir,ejp-09-lejay,jde-06-lejay-victoir,rmi-98-lyons}). Some local and global existence and uniqueness results are given in these papers under certain conditions. Unlike regular ordinary differential equations, the linear growth of the vector field $f$ is not enough to guarantee the global existence. Counterexamples can be seen in Section 1 of Lejay \cite{sdp-12-lejay}.

Assume that $f$ is a linear function. The next theorem provides the existence and uniqueness of the RDE (\ref{lrde}) and also gives an estimate of the solution.
\begin{theorem}[Theorem 2 of Lejay \cite{ejp-09-lejay}]\label{tlrde}
Suppose that $f(Y)=AY$ for some bounded linear operator $A\in \cL(K;\cL(V;K))$. Then, there exists a unique solution to (\ref{lrde}) on any time interval $[0,T]$. In addition, the following estimate holds:
\begin{align*}
\sup_{t\in[0,T]}|Y_t-y|\leq& |y|\exp\Big\{\big(CT\|A\|_{\cL(K;\cL(V;K))}^{\frac{1}{\alpha}}\big) \max\big\{1,(\|X\|_{\alpha}+\|\X\|_{2\alpha})^{\frac{1}{\alpha}}\big\}\Big\},
\end{align*}
for some universal constant $C>0$.
\end{theorem}

	\section{Nonlinear rough integrals}\label{s.3}
\subsection{Definitions}\label{s.3.1}
Fix a time inteval $[0,T]$. Suppose that $\alpha\in (\frac{1}{3},\frac{1}{2}]$. In this section, we aim to define the following nonlinear integral:
\[
\int_s^t W(dr, Y_r).
\]
Here $W$ is $\alpha$-H\"{o}lder continuous in time, and differentiable in space, and $Y$ is $\alpha$-H\"{o}lder continuous. The idea is as follows. Assume that $Y$ is controlled by $W$, that is $Y_{s,t}=W_{s,t}(\dY_s)+O(|t-s|^{2\alpha})$. Then, we approximate the nonlinear integral by the following expression: 
\begin{align*}
\int_s^t W(dr, Y_r)\approx &\int_s^t W(dr, Y_s)+\int_s^t DW(dr, Y_s)Y_{r,s}\\
\approx &\int_s^t W(dr, Y_s)+\int_s^t DW(dr, Y_s)W_{s,r}(\dY_s)\\
=&W_{s,t}(Y_s)+\int_s^t DW(dr, y)W_{s,r}(x)\Big|_{(x,y)=(\dY_s, Y_s)},
\end{align*}
with the error of order $O(|t-s|^{3\alpha})$. This allows us to pass to the limit as $|\pi|\to 0$ in the following expression
\begin{align*}
\sum_{k=1}^{n}\Big[W_{t_{k-1},t_k}(Y_{t_{k-1}})+\int_{t_{k-1}}^{t_k} DW(dr, y)W(r, x)\Big|_{(x,y)=(\dY_{t_{k-1}}, Y_{t_{k-1}})}\Big],
\end{align*}
where $\pi=(s=t_0<t_1<\cdots<t_n=t)$. The limit is a desired version of the nonlinear integral.

To this end, we need to introduce the following definitions. Let $n$ be any nonnegative integer. We denote by $\mathcal{I}_n$ the set of all multi-indexes $\bbeta_n$ of length $n+1$. That is, $\bbeta_n=(\beta_0,\dots,\beta_n)$, where $\beta_0,\dots,\beta_n$ are nonnegative real numbers. These multi-indexes will be used to characterize the growth of a function and its spatial derivatives.
\begin{definition}\label{dnrf}
	\begin{enumerate}[(i)]	
		\item $\cC^{\alpha,\bbeta_n}([0,T]\times V;K)$ is the space of functions such that the following seminorm is finite:
		\begin{align}\label{ccv}
		\|\Phi\|_{\alpha,\bbeta_n}:=\sum_{k=0}^n\sup_{\substack{s\neq t\in [0,T]\\x\in V}}\frac{\|D^k\Phi_{s,t}(x)\|_{\kL_k(V;K)}}{|t-s|^{\alpha}(1+\|x\|_V)^{\beta_k}},
		\end{align}
		where $D^k$ is the $k$-th Fr\'{e}chet derivative operator, and $\kL_k(V;K)$ is the corresponding linear space of derivatives. That is, $\kL_0(V;K)=K$ and $\kL_k(V;K)=\mathcal{L}(V;\kL_{k-1}(V;K))$ for all $k=1,2,\dots, n$.
		
	\item $\cC^{\alpha,\bbeta^1_n,\bbeta^2_n}_2([0,T]^2\times V^2;K)$ is the space of functions such that the following seminorm is finite:
	\begin{align}
	\|\Psi\|_{\alpha,\bbeta^1_n,\bbeta^2_n}:=\sum_{k=0}^n\sup_{\substack{s\neq t\in [0,T]\\ \bx=(x_1,x_2)\in V^2}}\frac{\|D^k\Psi_{s,t}(\bx)\|_{\kL_k(V^2;K)}}{|t-s|^{\alpha}(1+\|x_1\|_{V})^{\beta^1_k}(1+\|x_2\|_{V})^{\beta^2_k}},
	\end{align}
	where $\kL_k(V^2;K)$ are the corresponding linear spaces of derivatives and the product space $V^2$ is treated as a Banach space equipped with the norm $\|\bx\|_{V^2}=\|x_1\|_V+\|x_2\|_V$.
	\end{enumerate}
	\end{definition}
	For any positive integer $m\leq n$, we write $\bbeta_n-m=(\beta_0,\dots,\beta_{n-m})$. Then, by definition, it is easy to verify that $\cC^{\alpha,\bbeta_n}([0,T]\times V;K)\subset \cC^{\alpha,\bbeta_n-m}([0,T]\times V;K)$. Let $\bbeta_n, \tbbeta_n\in\mathcal{I}_n$, we write $\bbeta_n\leq \tbbeta_n$ if $\beta_k\leq \tbeta_k$ for all $k=0,\dots, n$. Then, $\cC^{\alpha,\bbeta_n}([0,T]\times V;K)\subset \cC^{\alpha,\tbbeta_n}([0,T]\times V;K)$ if $\bbeta_n\leq \tbbeta_n$. The space $\cC^{\alpha,\bbeta_n^1,\bbeta_n^2}_2([0,T]^2\times V^2;K)$ also has a similar property. Given a multi-index $\bbeta_n$ where $n\geq 1$, we make use of the following notations: 
	\begin{align}\label{beta12}
	\bbeta^*_{n-1}=(\beta^*_0,\dots,\beta^*_{n-1})\ \mathrm{and} \ \bbeta^{**}_{n-1}=(\beta^{**}_0,\dots,\beta^{**}_{n-1}),
	\end{align}
	where $\beta^*_k:=\max\{\beta_0,\dots, \beta_{k}\}$ and $\beta^{**}_k:=\max\{\beta_1,\dots,\beta_{k+1}\}$ for all $0\leq k\leq n-1$.

Given multi-indexes $\bbeta_2$, $\bbeta^*_1$ and $\bbeta^{**}_1$, let $\Phi\in \cC^{\alpha,\bbeta_2}([0,T]\times V;K)$ and $\Psi \in\cC^{\alpha,\bbeta_1^1,\bbeta_1^2}_2([0,T]^2\times V^2;K)$. We make use of the following notations: $\sR^{\Phi}:[0,T]\times V^2\to K$ and $\sD^{\Psi}:[0,T]^2\times V^4\to K$ given by
\begin{align}\label{crphi}
\sR^{\Phi}_t(x,y):=\Phi_t(y)-\Phi_t(x)-D\Phi_t(x)(y-x),\ x, y\in V
\end{align}
and
	\begin{align}\label{sdpsi} 
	 \sD^{\Psi}_{s,t}(\bx,\by)=\Psi_{s,t}(\by)-\Psi_{s,t}(\bx),\ \bx,\by\in V^2.
	 \end{align}
	The following lemma provides the estimates for $\sR^{\Phi}$, $\sD^{\Psi}$ and their derivatives. It will be used in the proof of the stability of nonlinear rough integrals.
	\begin{lemma}\label{crsd}
	Suppose that $\sR^{\Phi}$ and $\sD^{\Psi}$ be given in (\ref{crphi}) and (\ref{sdpsi}), respectively. Then the following inequalities are satisfied:
 \begin{align}
 \|\sR^{\Phi}_{s,t}(x, y)\|_K\leq &\frac{1}{2}\|\Phi\|_{\alpha,\bbeta_2}(1+\|x\|_V+\|y\|_V )^{\beta_2}\|y-x\|_V^2|t-s|^{\alpha},\label{crsd1}\\
 \|\sD^{\Psi}_{s,t}(\bx,\by)\|_K\leq &\|\Psi\|_{\alpha,\bbeta_1^1,\bbeta_1^2}(1+\|x_1\|_V+\|y_1\|_V)^{\beta^1_1}(1+\|x_2\|_V+\|y_2\|_V)^{\beta^2_1}\nonumber\\
 &\times\|\by-\bx\|_{V^2}|t-s|^{\alpha}.\label{crsd2}
 \end{align}
 If furthermore $\Phi\in \cC^{\alpha,\bbeta_3}([0,T]\times V;K)$ and $\Psi \in\cC^{\alpha,\bbeta^*_2,\bbeta^{**}_2}_2([0,T]^2\times V^2;K)$. Then, for all $\bz^1,\bz^2\in V^2$, the following inequalities are satisfied:
 \begin{align}
 \|D\sR^{\Phi}_{s,t}(x, y)(z_1, z_2)\|_K\leq& \|\Phi\|_{\alpha,\bbeta_3}(1+\|x\|_V+\|y\|_V)^{\beta_2\vee\beta_3}\nonumber\\
 &\times\big[\|y-x\|_V^2\|z_2\|_V+\|y-x\|_V\|z_1-z_2\|_V\big]|t-s|^{\alpha},\label{crsd3}\\
	 \|D\sD^{\Psi}_{s,t}(\bx,\by)(\bz_1,\bz_2)\|_K\leq & \|\Psi\|_{\alpha,\bbeta^*_2,\bbeta^{**}_2}(1+\|x_1\|_V+\|y_1\|_V)^{\beta^1_1\vee\beta^1_2}(1+\|x_2\|_V+\|y_2\|_V)^{\beta^2_1\vee\beta^2_2}\label{crsd4}\nonumber\\
	 &\times\big[\|\by-\bx\|_{V^2}\|\bz_2\|_{V^2}+\|\bz_1-\bz_2\|_{V^2}\big]|t-s|^{\alpha}.
	 \end{align}
	\end{lemma}
	\begin{proof}
The inequality (\ref{crsd1}) is a consequence of Taylor's theorem:
\begin{align*}
\sR^{\Phi}_{s,t}(x, y)=\frac{1}{2}D^2\Phi_{s,t}(\xi)(y-x,y-x),
\end{align*}
 where $\xi=cx+(1-c)y$ for some $c\in [0,1]$.

For the inequality (\ref{crsd3}), we assume that $\Phi\in \cC^{\alpha,\bbeta_3}([0,T]\times V;K)$. Then, by differentiating $\sR^{\Phi}_{s,t}$ on the spatial argument, for any $(x,y), (z_1,z_2)\in V^2$, we have
\begin{align*}
D\sR^{\Phi}_{s,t}(x,y)(z_1,z_2)=&D\Phi_{s,t}(y)(z_2)-D\Phi_{s,t}(x)(z_2)-D^2\Phi_{s,t}(x)(z_1, y-x)\\
=&D^2\Phi_{s,t}(x)(z_2-z_1, y-x)+\frac{1}{2}D^3\Phi_{s,t}(\widetilde{\xi})(z_2, y-x,y-x)
\end{align*}
where $\widehat{\xi}$ is between $x$ and $y$. This implies the inequality (\ref{crsd3}). The inequality (\ref{crsd2}) and (\ref{crsd4}) can be proved similarly.
\end{proof}

In the rest of this paper, we focus on the case when $K=V$. A nonlinear rough path is defined as follows.
\begin{definition}\label{nrp}
Assume that $n\geq 1$. The space $\C^{\alpha,\bbeta_n}([0,T]\times V;V)$ is defined as the collection of $\alpha$-H\"{o}lder nonlinear rough paths $\bW=(W,\W)$ that satisfies the following properties:
\begin{enumerate}[(i)]
\item $W\in \cC^{\alpha,\bbeta_n}([0,T]\times V;V)$.

\item $\W\in \cC^{2\alpha,\bbeta^*_{n-1},\bbeta^{**}_{n-1}}_2([0,T]^2\times V^2;V)$, where $\bbeta^{*}_{n-1}$ and $\bbeta^{**}_{n-1}$ are defined in (\ref{beta12}).

\item $(W,\W)$ satisfies Chen's relation: 
\begin{align}\label{chen}
\W_{s,t}(x,y)-\W_{s,u}(x,y)-\W_{u,t}(x,y)=DW_{u,t}(y)(W_{s,u}(x)),
\end{align}
for all $(x,y)\in V^2$ and $s,u,t \in [0,T]$.
\end{enumerate}
\end{definition}

\begin{remark}
\begin{enumerate}[(i)] 
\item   In the smooth case, $\W$ can be interpreted as a version of the following integral
\begin{align*}
\int_s^t DW(dr,y)(W_{s,r}(x)):=\W_{s,t}(x,y)
\end{align*}
and this explains the choice of the multi-indexes $\bbeta^*_{n-1}$ and $\bbeta^{**}_{n-1}$ in point (ii) of Definition  \ref{nrp}.
\item By definition, we can deduce that $\C^{\alpha,\bbeta^*_n}([0,T]\times V;V)\subset \C^{\alpha,\bbeta^{**}_n-m}([0,T]\times V;V)$ for all $m\in \{0,\dots,n\}$ and $\bbeta^*_n\leq \bbeta^{**}_n$. 
\item Assume that $W(t,x)=W_t(x)$ where $W_t\in \cL(V;V)$. Then the nonlinear rough path degenerate to the linear rough path. In this case,
\[
\W_{s,t}(x,y)=\int_s^tW_{s,r}dW_r(x).
\]
\end{enumerate}
\end{remark}

Let $\bW=(W,\W)\in \cC^{\alpha,\bbeta_n}([0,T]\times V;V)$. We make use of the notation
\[
\|\bW\|_{\C_n}:=\|W\|_{\alpha, \bbeta_n}+\|\W\|_{\alpha, \bbeta^*_{n-1},\bbeta^{**}_{n-1}}.
\]
Notice that $\C^{\alpha,\bbeta_n}([0,T]\times V;V)$ is not a linear space with the usual addition and scalar product. Thus $\|\cdot\|_{\C_n}$ is not a seminorm in the usual sense. We introduce the pseudometric on $\C^{\alpha,\bbeta_n}([0,T]\times V;V)$ given by
\begin{align}
\vrho_{\alpha, \bbeta_n}(\bW,\tbW)=\|W-\tW\|_{\alpha,\bbeta_n}+\|\W-\tdW\|_{2\alpha, \bbeta^*_{n-1},\bbeta^{**}_{n-1}}.
\end{align}
Consider the following equivalent relation: $\bW\sim\tbW$ if and only if   there exists $f\in \cC^{\bbeta_n}(V;V)$ such that $W(t,x)-\tW(t,x)= f(x)$ for all $(t,x)\in[0,T]\times V$.  Then, $\vrho_{\alpha, \bbeta_n}$ is really a metric on the quotient space $\cC^{\alpha, \bbeta_n}([0,T]\times V;V)/\sim$.

Let $W\in \cC^{\alpha, \bbeta_n}([0,T]\times V;V)$. Like in the linear case, we also define the space of nonlinear controlled rough paths by $W$.

\begin{definition}\label{crp}
 The space of basic nonlinear controlled rough paths by $W$, denoted by $\sE_W^{2\alpha}$, is the collection of pairs $(Y, \dY)\in \cC^{\alpha}([0,T];V)\times \cC^{\alpha}([0,T];V)$ such that, for all $s,t \in [0,T]$,
\begin{align}\label{crpd}
	Y_{s,t}=W_{s,t}(\dY)+R^Y_{s,t},
	\end{align}
	where $R^Y\in \cC^{2\alpha}_2([0,T]^2;V)$. The function $\dY$ above is called the Gubinelli derivative of $Y$ with respect to $W$.
\end{definition}
 \begin{remark}
 \begin{enumerate}[(i)]
\item Unlike the linear case, $\sE^{2\alpha}_W$ does not need to be a linear space with the usual addition and scalar product, because it may be not closed under these operations.

 \item Assume that $V=\R$ and $W(t,x)=xW_t$, then the controlled rough path satisfies the following equality
\[
Y_{s,t}=\dY_sW_{s,t}+R^Y_{s,t},
\]
which coincides with the classic definition in the linear case.

\item With an abuse of notatios, we sometimes write $Y\in \sE_W^{2\alpha}$ instead of $(Y,\dY)\in \sE_W^{2\alpha}$.
\end{enumerate}
\end{remark}

Suppose that $W,\tW\in \cC^{\alpha,\bbeta_n}([0,T]\times V;V)$. Let $(Y,\dY)\in \sE^{2\alpha}_W$ and $(\tY,\dtY)\in \sE^{2\alpha}_{\tW}$ respectively. A ``distance" between $(Y,\dY)$ and $(\tY,\dtY)$ is defined as follows:
\begin{align}
 d_{\alpha, W,\tW}\big((Y,\dY),(\tY,\dtY)\big)=\|\dY-\dtY\|_{\alpha}+\|R^Y-R^{\tY}\|_{2\alpha}.
\end{align}
Notice that the definition of $d_{\alpha, W,\tW}$ does not include the term $ \|Y-\tY\|_{\alpha}$. Indeed, this term can be estimated in terms of
 $d_{\alpha, W,\tW}\big((Y,\dY),(\tY,\dtY)\big)$ as it is shown in the next lemma.

\begin{lemma}\label{tdtvphin}
Let $W,\tW\in \cC^{\alpha,\bbeta_1}([0,T]\times V;V)$. Suppose that $(Y,\dY)\in \sE_{W}^{2\alpha}$ and $(\tY,\dtY)\in \sE_{\tW}^{2\alpha}$ respectively. Then the following estimate holds:
\begin{align}\label{dtvphin}
 \|Y-\tY\|_{\alpha}\leq &(1+\|\dY\|_{\infty})^{\beta_0}\|W-\tW\|_{\alpha,\bbeta_1}\\
 &+\|\tW\|_{\alpha,\bbeta_1}(1+\|\dY\|_{\infty}+\|\dtY\|_{\infty})^{\beta_1}\|\dY_0-\dtY_0\|_V\nonumber\\
 &+T^{\alpha}(1+\|\tW\|_{\alpha,\bbeta_1})(1+\|\dY\|_{\infty}+ \|\dtY\|_{\infty})^{\beta_1}d_{\alpha,W,\tW}\big((Y,\dY),(\tY ,\dtY)\big).\nonumber
\end{align}
\end{lemma}
\begin{proof}
Since $Y$ and $\tY$ are controlled by $W$ and $\tW$ respectively, then we have
	 \begin{align*}
\|Y_{s,t}-\tY_{s,t}\|_V\leq &\big\|W_{s,t}(\dY_s)-\tW_{s,t}(\dY_s)\big\|_V+\big\|\tW_{s,t}(\dY_s)-\tW_{s,t}(\dtY_s)\big\|_V+\big\|R^{Y}_{s,t}-R^{\tY}_{s,t}\big\|_V\nonumber\\
	 \leq &\|W-\tW\|_{\alpha,\bbeta_1}(1+\|\dY\|_{\infty})^{\beta_0}|t-s|^{\alpha}\nonumber\\
	 &+\|\tW\|_{\alpha,\bbeta_1}(1+\|\dY\|_{\infty}+\|\dtY\|_{\infty})^{\beta_1}\big(\|\dY_0-\dtY_0\|_V+s^{\alpha}\|\dY-\dtY\|_{\alpha}\big)|t-s|^{\alpha}\nonumber\\
	 &+\big\|R^{Y}-R^{\tY}\big\|_{2\alpha}|t-s|^{2\alpha}.
	 \end{align*}
This proves the inequality (\ref{dtvphin}). 
\end{proof}

Applying Lemma \ref{tdtvphin}, the supremum norm of $Y-\tY$ can be estimated as follows:
	 \begin{align}\label{dtvphi1}
	 \|Y-\tY\|_{\infty}\leq &\|Y_0-\tY_0\|_V+T^{\alpha}\|Y-\tY\|_{\alpha}\\
	 \leq &T^{\alpha}(1+\|\dY\|_{\infty})^{\beta_0}\|W-\tW\|_{\alpha,\bbeta_3}\nonumber\\
	 &+(1+T^{\alpha})(1+\|\tW\|_{\alpha,\bbeta_3})(1+\|\dY\|_{\infty}+\|\dtY\|_{\infty})^{\beta_1}\big(\|Y_0-\tY_0\|_V+\|\dY_0-\dtY_0\|_V\big)\nonumber\\
	 &+T^{2\alpha}(1+\|\tW\|_{\alpha,\bbeta_3})(1+\|\dY\|_{\infty}+\|\dtY\|_{\infty})^{\beta_1}d_{\alpha, W,\tW}\big((Y,\dY), (\tY,\dtY)\big).\nonumber
	 \end{align}
Both inequalities (\ref{dtvphin}) and (\ref{dtvphi1}) will be used frequently throughout the rest of this paper.

\begin{remark}
$d_{\alpha,W,\tW}$ is not a metric, because $(Y,\dY)$ and $(\tY,\dtY)$ may belong to different spaces. For any $(y_1,y_2)\in V^2$, let 
\[
\sE_{W,y_1,y_2}^{2\alpha}=\{(Y,\dY)\in \sE^{2\alpha}_W, (Y_0,\dY_0)=(y_1,y_2)\}.
\]
 Then $d_{\alpha,W}=d_{\alpha,W,W}$ is really a metric on $\sE_{W,y_1,y_2}^{2\alpha}$. 
\end{remark}
 
 The next lemma shows that $\sE_{W,y_1,y_2}^{2\alpha}$ is complete under the metric $d_{\alpha, W}$.
\begin{lemma}\label{cmplt}
Suppose that $W\in \cC^{\alpha,\bbeta_1}([0,T]\times V;V)$. Let $(y_1,y_2)\in V^2$. Then $(\sE_{W,y_1,y_2}^{2\alpha}, d_{\alpha,W})$ is a complete metric space .
\end{lemma}
\begin{proof}
Suppose that $\{(Y^n,\dY^n)\}_{n\geq 1}\subset \sE_{W,y_1,y_2}^{2\alpha}$ is a Cauchy sequence under the metric $d_{\alpha,W}$.
We first show that $\{(Y^n,\dY^n, R^{Y^n})\}_{n\geq 1}$ converges to $(Y,\dY, R^Y)$ in the product space $\cC^{\alpha}([0,T];V)\times \cC^{\alpha}([0,T];V)\times \cC^{2\alpha}_2([0,T]^2;V)$ with the H\"{o}lder seminorms. 

Notice that the space $\cC^{\alpha}([0,T]; V)$ is complete with the norm
\begin{align*}
\|Y\|_{\cC^{\alpha}([0,T]; V)}:=\|Y_0\|_V+\|Y\|_{\alpha}.
\end{align*}
Thus there exists $\dY\in \cC^{\alpha}([0,T]; V)$, such that $\dY^n\to \dY$ as $n\to\infty$ pointwise and in $\cC^{\alpha}([0,T];V)$. In order to show the convergence of $\{R^{Y^n}\}_{n\geq 1}$, we fix $s\in [0,T]$ and consider the sequence of functions $\{R^{Y^n}_{s,\cdot}\}_{n\geq 1}$ on $[0,T]$. Then, for any $t\in[0,T]$ and $n,m\geq 1$, we have
\[
\|R^{Y^n}_{s,t}-R^{Y^m}_{s,t}\|_V\leq |t-s|^{2\alpha}\|R^{Y^n}-R^{Y^m}\|_{2\alpha}.
\]
Therefore, $\{R^{Y^n}_{s,t}\}_{n\geq 1}$ is a Cauchy sequence in $V$, and thus has a limit denoted by $R^Y_{s,t}$. On the other hand, we can show that
\[
\lim_{n\to\infty}\sup_{s\neq t\in[0,T]}\frac{\|R^Y_{s,t}-R^{Y^n}_{s,t}\|_V}{|t-s|^{2\alpha}}\leq\lim_{n\to\infty}\lim_{m \to\infty}\sup_{s\neq t\in [0,T]}\frac{\|R^{Y^m}_{s,t}-R^{Y^n}_{s,t}\|_V}{|t-s|^{2\alpha}}=0.
\]
This implies that the convergence is also in $\cC^{2\alpha}_2([0,T]^2;V)$. To prove the convergence of $\{Y^n\}_{n\geq 1}$, it suffices to show that $\{Y^n\}_{n\geq 1}$ is Cauchy in $\cC^{\alpha}([0,T]; V)$ with the $\alpha$-H\"{o}lder seminorm. Notice that for any $n,m\geq 1$, $Y^n$ and $Y^m$ are both controlled by $W$, then as a consequence of Lemma \ref{tdtvphin}, we have 
\begin{align}\label{dphi}
 \|Y^n-Y^m\|_{\alpha}\leq &T^{\alpha}(1+\|W\|_{\alpha,\bbeta_1})(1+\|\dY^n\|_{\infty}+ \|\dY^m\|_{\infty})^{\beta_1}\nonumber\\
 &\times d_{\alpha,W,\tW}\big((Y^n,\dY^n),(Y^m ,\dY^m)\big).
\end{align}
Observe that
\[
\sup_{n\geq 1}\|\dY^n\|_{\infty}\leq y_2+T^{\alpha}\sup_{n\geq 1}\|\dY^n\|_{\alpha}=C<\infty.
\]
Therefore, $\{Y^n\}_{n\geq 1}$ converges to a function $Y$ in $\cC^{\alpha}([0,T]; V)$.

Finally, notice that for any $s,t \in[0,T]$,
\begin{align}
Y_{s,t}=\lim_{n\to\infty}Y^n_{s,t}=\lim_{n\to\infty}\big[W_{s,t}(\dY^n_s)+R^{Y^n}_{s,t}\big]=W_{s,t}(\dY_s)+R^Y_{s,t}.
\end{align}
Thus $(Y, \dY)\in \sE_{W,\by}^{2\alpha}$ with the remainder $R^Y$. 
\end{proof}

In the next theorem, we define the nonlinear rough integral of a basic controlled rough paths against a nonlinear rough path.
	\begin{theorem}\label{nri}
	Suppose that $\bW=(W,\W)\in \C^{\alpha,\bbeta_2}([0,T]\times V;V)$. Let $(Y,\dY)\in \sE^{2\alpha}_W$. We define $\Xi\in \cC^{\alpha}_2([0,T] ;V)$ as follows:
	\[
	\Xi_{s,t}=W_{s,t}(Y_s)+\W_{s,t}(\dY_s,Y_s).
	\]
Then the following limit exists 
		\begin{align}\label{nrid}
	\cJ_{s,t}(\Xi):=\lim_{|\pi|\to 0}\sum_{k=1}^n\Xi_{t_{k-1},t_k}
	\end{align}
where $\pi=(s=t_0<t_1<\dots<t_n=t)$. Moreover,
	\begin{align}\label{nrisd}
	\big\|\cJ_{s,t}(\Xi)-\Xi_{s,t}\big\|_V\leq C_1|t-s|^{3\alpha},
	\end{align}
	where
	\[
	C_1=k_{\alpha}\|\bW\|_{\C_2}(1+2\|\dY\|_{\infty})^{\beta_0\vee\beta_1}(1+2\|Y\|_{\infty})^{\beta_1\vee\beta_2}\big(\|Y\|_{\alpha}+\|\dY\|_{\alpha}+\|R^Y\|_{2\alpha}\big),
	\]
	and $k_{\alpha}$ is defined in (\ref{kalpha}).
	\end{theorem}

	\begin{proof}
 For any $0\leq s\leq u\leq t\leq T$, we write
	\begin{align}\label{edxi}
	\delta\Xi_{s,u,t}=&\Xi_{s,t}-\Xi_{s,u}-\Xi_{u,t}\\
	=&-[W_{u,t}(Y_u)-W_{u,t}(Y_s)]+[\W_{s,t}(\dY_s,Y_s)-\W_{s,u}(\dY_s,Y_s)-\W_{u,t}(\dY_u,Y_u)].\nonumber
	\end{align}
According to Lemma \ref{sew}, it suffices to estimate $\|\delta \Xi_{s,u,t}\|_V$.  Recall the notations \eqref{crphi} and \eqref{sdpsi}. Since $Y$ is controlled by $W$, we can write 
\begin{align}\label{wytl}
	W_{u,t}(Y_u)&-W_{u,t}(Y_s)=DW_{u,t}(Y_s)(Y_{s,u})+\sR^{W}_{u,t}(Y_s,Y_u)\nonumber\\
	=&DW_{u,t}(Y_s)(W_{s,u}(\dY_s))+DW_{u,t}(Y_s)(R_{s,u}^{Y})+\sR^{W}_{u,t}(Y_s,Y_u).
\end{align}
On the other hand, by Chen's relation (\ref{chen}), we have
	\begin{align}\label{edxi2}
\W_{s,t}(\dY_s,Y_s)-&\W_{s,u}(\dY_s,Y_s)-\W_{u,t}(\dY_u,Y_u)\nonumber\\
&=DW_{u,t}(Y_s)(W_{s,u}(\dY_s))-\sD^{\W}_{u,t}\big((\dY_s, Y_s),(\dY_u, Y_u)\big).
	\end{align}
 Notice that, by definition, $\W\in \cC^{2\alpha,\bbeta^*_1,\bbeta^{**}_1}_2([0,T]^2\times V^2; V)$ where $\bbeta^*_1=(\beta_0,\beta_0\vee\beta_1)$ and $\bbeta^{**}_1=(\beta_1,\beta_1\vee\beta_2)$. Combining \eqref{edxi} - \eqref{edxi2}, with \eqref{crsd1} and \eqref{crsd2}, we obtain the following inequality
\begin{align}\label{dtdt}
\|\delta\Xi_{s,u,t}\|_V\leq& \|DW_{u,t}(Y_s)\|_{\kL_1(V;V)}\|R_{s,u}^{Y}\|_V+\frac{1}{2}\|W\|_{\alpha,\bbeta_2}(1+2\|Y\|_{\infty})^{\beta_2}\|Y\|_{\alpha}|t-s|^{3\alpha}\\
&+\|\W\|_{2\alpha,\bbeta^*_1,\bbeta^{**}_1}(1+2\|\dY\|_{\infty})^{\beta_1^*}(1+2\|Y\|_{\infty})^{\beta_1^{**}}(\|Y\|_{\alpha}+\|\dY\|_{\alpha})|t-s|^{3\alpha}\nonumber\\
\leq&\|\bW\|_{\C_2}(1+2\|\dY\|_{\infty})^{\beta_0\vee\beta_1}(1+2\|Y\|_{\infty})^{\beta_1\vee\beta_2}\big(\|Y\|_{\alpha}+\|\dY\|_{\alpha}+\|R^Y\|_{2\alpha}\big)|t-s|^{3\alpha}.\nonumber
	\end{align}
Thus we complete the proof by applying Lemma \ref{sew}.
	\end{proof}
	
 Due to the sewing lemma the functional $\cJ_{s,t}(\Xi)$ defined in Theorem \ref{nri} is additive. Therefore, we can define the nonlinear integral of $Y$ against $W$ on any time interval $[s,t]\subset [0,T]$ by $\cJ_{s,t}(\Xi)$, that is
\begin{align}\label{dnri}
\int_s^t W(dr,Y_r):=\cJ_{s,t}(\Xi).
\end{align}

By definition, we can easily verify that $\Xi$ in Theorem \ref{nri} is also $\alpha$-H\"{o}lder continuous. Recall that $\bbeta^*_1=(\beta_0, \beta_0 \vee \beta_1)$ and
$ \bbeta^{**}_1= (\beta_1, \beta_1 \vee \beta_2)$. Thus we have the following estimate,
	\begin{align}\label{icmxi}
	\|\Xi_{s,t}\|_V\leq& \|W_{s,t}(Y_s)\|_V+\|\W_{s,t}(\dY_s,Y_s)\|_V\nonumber\\
	\leq& \|W\|_{\alpha,\bbeta_2}(1+\|Y\|_{\infty})^{\beta_0}|t-s|^{\alpha}\nonumber\\
	&+\|\W\|_{2\alpha,\bbeta^*_1,\bbeta^{**}_1}(1+\|\dY\|_{\infty})^{\beta_0}(1+\|Y\|_{\infty})^{\beta_1}|t-s|^{2\alpha}.
	\end{align}
The following estimates follows from (\ref{nrisd}) and (\ref{icmxi}):
\begin{align}\label{hnnri}
\left\|\int_s^tW(dr,Y_r)\right\|_V\leq &\|\Xi_{s,t}\|_V+\|\mathcal{J}_{s,t}(\Xi)-\Xi_{s,t}\|_V\nonumber\\
\leq &C_2|t-s|^{\alpha},
\end{align}
where
\[
C_2=C_1T^{2\alpha}+\|W\|_{\alpha,\bbeta_2}(1+\|Y\|_{\infty})^{\beta_0}+T^{\alpha}\|\W\|_{2\alpha,\bbeta^*_1,\bbeta^{**}_1}(1+\|\dY\|_{\infty})^{\beta_0}(1+\|Y\|_{\infty})^{\beta_1}.
\]
\begin{remark}
To define a nonlinear rough integral, the growth condition of $(W,\W)$ is not necessary. In fact, let $\C^{\alpha,2}_{loc}([0,T]\times V; V)$ be the collection of pairs $(W,\W)$ such that $W:[0,T]\times V\to V$ is $\alpha$-H\"{o}lder in time, and twice differentiable in space with locally bounded spatial derivatives, $\W:[0,T]^2\times V\to V$ is $2\alpha$-H\"{o}lder continuous in time, and differentiable in space with locally bounded spatial derivative, and Chen's relation (\ref{chen}) holds. For any $\bW=(W,\W)\in\C^{\alpha,2}_{loc}([0,T]\times V; V)$, and $(Y, \dY)\in \sE_W^{2\alpha}$, the expression (\ref{nrid}) is still a well-defined nonlinear rough integral. However, the growth condition is really needed to consider the global existence of RDEs (see Section \ref{s.4.2}).
\end{remark}

\subsection{Properties of nonlinear rough integrals}\label{s.3.2}
In this section, we present some properties of nonlinear rough integrals. The next proposition shows that the nonlinear rough integral is a basic nonlinear controlled rough path (see Proposition \ref{lricrp} for the linear result).
	\begin{proposition}\label{ricrp}
		 Let $\bW=(W,\W)\in \C^{\alpha,\bbeta_2}([0,T]\times V;V)$. Suppose that $(Y,\dY)\in \sE^{2\alpha}_W$. Let $Z:[0,T]\to V$ be the nonlinear rough integral of $Y$ against $W$ in the sense of \eqref{dnri}:
		\begin{align}
		Z_t=\int_0^t W(dr,Y_r).
		\end{align}
	 Then, $Z$ is controlled by $W$: $(Z,\dZ)=(Z,Y)\in \sE^{2\alpha}_W$.
		\end{proposition}
	\begin{proof}
Let $R^Z_{s,t}:=Z_{s,t}-W_{s,t}(Y_s)$. Then by (\ref{nrisd}), we can write
\begin{align*}
\big\|R^Z_{s,t}&\big\|_V=\Big\|\int_s^tW(dr,Y_r)-W_{s,t}(Y_s)\Big\|_V\\
\leq &\|\cJ_{s,t}(\Xi)-\Xi_{s,t}\|_V+\|\W_{s,t}(\dY_s,Y_s)\|_V\\
\leq &k_{\alpha}\|\bW\|_{\C_2}(1+2\|\dY\|_{\infty})^{\beta_0\vee\beta_1}(1+2\|Y\|_{\infty})^{\beta_1\vee\beta_2}\big(\|Y\|_{\alpha}+\|\dY\|_{\alpha}+\|R^Y\|_{2\alpha}\big)|t-s|^{3\alpha}\\
&+\|\bW\|_{\C_2}(1+\|\dY\|_{\infty})^{\beta_0}(1+\|Y\|_{\infty})^{\beta_1}|t-s|^{2\alpha}.
\end{align*}
It follows that
\begin{align}\label{ermd}
\|R^Z\|_{2\alpha}\leq &k_{\alpha}\|\bW\|_{\C}(1+2\|\dY\|_{\infty})^{\beta_0\vee\beta_1}(1+2\|Y\|_{\infty})^{\beta_1\vee\beta_2}\nonumber\\
	&\times \big[1+T^{\alpha}\big(\|Y\|_{\alpha}+\|\dY\|_{\alpha}+\|R^Y\|_{2\alpha}\big)\big].
	\end{align}
	As a consequence, $Z$ is controlled by $W$ with the Gubinelli derivative $\dZ=Y$.
	\end{proof}

In the next proposition, we consider the stability of nonlinear rough integrals.
	\begin{proposition}\label{nic1}
		Let $\bW, \tbW\in \C^{\alpha,\bbeta_3}([0,T]\times V;V)$. Suppose that $(Y,\dY)\in \sE^{2\alpha}_{W}$ and $(\tY,\dtY)\in \sE^{2\alpha}_{\tW}$ respectively. Define
		\[
		Z_t=\int_0^tW(dr, Y_r)\ \mathrm{and}\ 	\tZ_t=\int_0^t\tW(dr, \tY_r).
		\]
		Then $(Z,Y)\in \sE_{W}^{2\alpha}$ and $(\tZ,\tY)\in\sE_{\tW}^{2\alpha}$ by Proposition \ref{ricrp}. In addition, the following inequality holds:
\begin{align}\label{dwyz}
d_{\alpha, W,\tW}\big((Z,Y),(\tZ,\tY)\big)\leq& C_3\vrho_{\alpha,\bbeta_3}(\bW, \tbW)+C_4\big(\|Y_0-\tY_0\|_V+\|\dY_0-\dtY_0\|_V\big)\nonumber\\
&+C_5d_{\alpha, W,\tW}((Y,\dY),(\tY,\dtY))\big].
\end{align}
where 
\begin{align*}
 C_3=&2k_{\alpha}(1+T^{\alpha})^2(1+\|\tbW\|_{\C_3})(1+2\|Y\|_{\infty}+2\|\tY\|_{\infty})^{\beta^{**}_2}(1+2\|\dY\|_{\infty}+2\|\dtY\|_{\infty})^{\beta^*_2+\beta_0\vee \beta_1}\nonumber\\
 &\quad \times \big[1+\|Y\|_{\alpha}+\|\tY\|_{\alpha}+(\|Y\|_{\alpha}+\|\tY\|_{\alpha})^2+\|R^{Y}\|_{2\alpha}\big],\\
C_4=&5k_{\alpha}(1+T^{\alpha})^2(\|\tbW\|_{\C_3}+\|\tbW\|_{\C_3}^2) (1+2\|Y\|_{\infty}+2\|\tY\|_{\infty})^{\beta^{**}_2}(1+2\|\dY\|_{\infty}+2\|\dtY\|_{\infty})^{\beta^*_2+\beta_1}\nonumber\\
&\quad \times \big[1+\|Y\|_{\alpha}+\|\dY\|_{\alpha}+\|\tY\|_{\alpha}+\|\dtY\|_{\alpha}+ (\|Y\|_{\alpha}+\|\tY\|_{\alpha})^2+\|R^{Y}\|_{2\alpha}\big]\nonumber\\
C_5=&6k_{\alpha}T^{\alpha}(1+T^{\alpha})(1+\|\tbW\|_{\C_3})^2 (1+2\|Y\|_{\infty}+2\|\tY\|_{\infty})^{\beta^{**}_2}(1+2\|\dY\|_{\infty}+2\|\dtY\|_{\infty})^{\beta^*_2+\beta_1}\\
&\quad \times \big[1+ T^{\alpha}(\|Y\|_{\alpha}+\|\dY\|_{\alpha}+\|\tY\|_{\alpha}+\|\dtY\|_{\alpha})+T^{2\alpha}(\|Y\|_{\alpha}+\|\tY\|_{\alpha})^2+T^{2\alpha}\|R^{Y}\|_{2\alpha}\big],
\end{align*}
$\beta^*_2=\max\{\beta_0,\beta_1,\beta_2\}$ and $\beta^{**}_2=\max\{\beta_1,\beta_2,\beta_3\}$.
\end{proposition}
	
	\begin{proof}
Due to Lemma \ref{tdtvphin} and Proposition \ref{ricrp}, it suffices to estimate $\|R^Z-R^{\tZ}\|_{2\alpha}$.

 Let $\Xi$ and $\tXi$ be the approximation of $Z$ and $\tZ$ respectively. That is,
\[
\Xi_{s,t}=W_{s,t}(Y_s)+\W_{s,t}(\dY_s,Y_s)\quad \mathrm{and}\quad \tXi_{s,t}=\tW_{s,t}(\tY_s)+\tW_{s,t}(\dtY_s,\tY_s).
\]
Set $\Delta=\Xi-\tXi$.  Taking into account   formulas \eqref{edxi} - \eqref{edxi2}, we can write
\begin{align}\label{b0}
\delta\Delta_{s,u,t} =&\big[DW_{u,t}(Y_s)(R^Y_{s,u})-D\tW_{u,t}(\tY_s)(R^{\tY}_{s,u})\big]+\big[\sR^{W}_{u,t}(Y_s,Y_u)-\sR^{\tW}_{u,t}(\tY_s,\tY_u)\big]\nonumber\\
&+\big[\sD^{\W}_{u,t}\big((\dY_s,Y_s),(\dY_u,Y_u)\big)-\sD^{\tdW}_{u,t}\big((\dtY_s,\tY_s),(\dtY_u,\tY_u)\big)\big]\nonumber\\
:=&J_1+J_2+J_3.
\end{align}
 For $J_1$, we have
\begin{align}\label{b1}
\|J_1\|_V\leq &\big\|DW_{u,t}(Y_s)(R_{s,u}^{Y})-D\tW_{u,t}(Y_s)(R_{s,u}^{Y})\big\|_V\nonumber\\
&+\big\|D\tW_{u,t}(Y_s)(R_{s,u}^{Y})-D\tW_{u,t}(\tY_s)(R_{s,u}^{Y})\big\|_V\nonumber\\
&+\big\|D\tW_{u,t}(\tY_s)(R_{s,u}^{Y})-D\tW_{u,t}(\tY_s)(R_{s,u}^{\tY})\big\|_V\nonumber\\
\leq &\|W-\tW\|_{\alpha,\bbeta_3}(1+\|Y\|_{\infty})^{\beta_1}\|R^{Y}\|_{2\alpha}|t-s|^{3\alpha}\nonumber\\
&+\|\tW\|_{\alpha,\bbeta_3}(1+\|Y\|_{\infty}+\|\tY\|_{\infty})^{\beta_2}\|Y_s-\tY_s\|_V\|R^{Y}\|_{2\alpha}|t-s|^{3\alpha}\nonumber\\
&+\|\tW\|_{\alpha,\bbeta_3}(1+\|\tY\|_{\infty})^{\beta_1}\|R^{Y}-R^{\tY}\|_{2\alpha}|t-s|^{3\alpha}.
\end{align}
Plugging (\ref{dtvphi1}) into (\ref{b1}), we have the following estimate for $J_1$:
\begin{align}\label{b11}
\|J_1\|_V	 \leq 	 &\Big\{(1+T^{\alpha})(1+\|\tW\|_{\alpha,\bbeta_3})(1+\|Y\|_{\infty}+\|\tY\|_{\infty})^{\beta_1\vee\beta_2}(1+\|\dY\|_{\infty})^{\beta_0}\nonumber\\
&\quad \times\|R^{Y}\|_{2\alpha}\|W-\tW\|_{\alpha,\bbeta_3}\nonumber\\
&+(1+T^{\alpha})(\|\tW\|_{\alpha,\bbeta_3}+\|\tW\|_{\alpha,\bbeta_3}^2)(1+\|Y\|_{\infty}+\|\tY\|_{\infty})^{\beta_1\vee\beta_2}\nonumber\\
&\quad \times(1+\|\dY\|_{\infty}+\|\dtY\|_{\infty})^{\beta_1}\|R^{Y}\|_{2\alpha}\big(\|Y_0-\tY_0\|_V+\|\dY_0-\dtY_0\|_V\big)\nonumber\\
 &+(\|\tW\|_{\alpha,\bbeta_3}+\|\tW\|_{\alpha,\bbeta_3}^2)(1+\|Y\|_{\infty}+\|\tY\|_{\infty})^{\beta_1\vee\beta_2}(1+\|\dY\|_{\infty}+\|\dtY\|_{\infty})^{\beta_1}\nonumber\\
 &\quad \times (1+T^{2\alpha}\|R^{Y}\|_{2\alpha})d_{\alpha, W,\tW}\big((Y,\dY), (\tY,\dtY)\big)\Big\}|t-s|^{3\alpha}.
\end{align}

 We apply Lemma \ref{crsd} to estimate $J_2$ and $J_3$ as follows. For $J_2$, by the mean value theorem, there exits $c\in[0,1]$, and $(\xi_1,\xi_2)=c(Y_s,Y_u)+(1-c)(\tY_s,\tY_u)$, such that
	 \begin{align*}
	 J_2= &\sR^{W-\tW}_{u,t}(Y_s,Y_u)+\big[D\sR^{\tW}_{u,t}(\xi_1,\xi_2)(Y_s-\tY_s,Y_u-\tY_u)\big]\\
:=&J_2^1+J_2^2.
	 \end{align*}
	 By (\ref{crsd1}), we can write
	\begin{align}\label{b21}
\|J_2^1\|_V\leq &\frac{1}{2}(1+2\|Y\|_{\infty})^{\beta_2}\|Y\|_{\alpha}^2\|W-\tW\|_{\alpha,\bbeta_3}|t-s|^{3\alpha}.
\end{align} 
On the other hand, using (\ref{crsd3}), \eqref{dtvphin} and (\ref{dtvphi1}), we have 
\begin{align}\label{b22}
\|J_2^2\|_V\leq &\|\tW\|_{\alpha,\bbeta_3}(1+2\|Y\|_{\infty}+2\|\tY\|_{\infty})^{\beta_2\vee\beta_3}\\
&\times \big[(\|Y\|_{\alpha}+\|\tY\|_{\alpha})^2\|Y_u-\tY_u\|_V+(\|Y\|_{\alpha}+\|\tY\|_{\alpha})\|Y-\tY\|_{\alpha}\big]|t-s|^{3\alpha}\nonumber\\
\leq &\Big[\|\tW\|_{\alpha,\bbeta_3}(1+2\|Y\|_{\infty}+2\|\tY\|_{\infty})^{\beta_2\vee\beta_3}(1+\|\dY\|_{\infty})^{\beta_0}\nonumber\\
&\times\big((\|Y\|_{\alpha}+\|\tY\|_{\alpha})+T^{\alpha}(\|Y\|_{\alpha}+\|\tY\|_{\alpha})^2\big)\|W-\tW\|_{\alpha,\bbeta_3}\nonumber\\
&+(1+T^{\alpha})(\|\tW\|_{\alpha,\bbeta_3}+\|\tW\|_{\alpha,\bbeta_3}^2)(1+2\|Y\|_{\infty}+2\|\tY\|_{\infty})^{\beta_2\vee\beta_3}(1+\|\dY\|_{\infty}+\|\dtY\|_{\infty})^{\beta_1}\nonumber\\
&\times\big((\|Y\|_{\alpha}+\|\tY\|_{\alpha})+(\|Y\|_{\alpha}+\|\tY\|_{\alpha})^2\big)\big(\|Y_0-\tY_0\|_V+\|\dY_0-\dtY_0\|_V\big)\nonumber\\
&+(\|\tW\|_{\alpha,\bbeta_3}+\|\tW\|_{\alpha,\bbeta_3}^2)(1+2\|Y\|_{\infty}+2\|\tY\|_{\infty})^{\beta_2\vee\beta_3}(1+\|\dY\|_{\infty}+\|\dtY\|_{\infty})^{\beta_1}\nonumber\\
&\times \big(T^{\alpha}(\|Y\|_{\alpha}+\|\tY\|_{\alpha})+T^{2\alpha}(\|Y\|_{\alpha}+\|\tY\|_{\alpha})^2\big)d_{\alpha, W,\tW}\big((Y,\dY), (\tY,\dtY)\big)\Big]|t-s|^{3\alpha}\nonumber.
\end{align}
For $J_3$, by using the mean value theorem again, we have
\begin{align*}
J_3=&\sD^{\W-\tdW}_{u,t}\big((\dY_s,Y_s), (\dY_u,Y_u)\big)+D\sD^{\tdW}_{u,t}(\bxi^1,\bxi^2)\big((\dY_s-\dtY_s,Y_s-\tY_s), (\dY_u-\dtY_u,Y_u-\tY_u)\big)\\
:=&J_3^1+J_3^2,
	 \end{align*}
	 where $\bxi^1=c'(\dY_s,Y_s)+(1-c')(\dtY_s,\tY_s)$ and $\bxi^2= c'(\dY_u,Y_u)+(1-c')(\dtY_u,\tY_u)$ for some $c'\in[0,1]$. Due to the inequalities (\ref{crsd2}), (\ref{crsd4}), (\ref{dtvphin}) and (\ref{dtvphi1}), we can show that
\begin{align}\label{b31}
\|J_3^1\|_V\leq &(1+2\|\dY\|_{\infty})^{\beta_0\vee\beta_1}(1+2\|Y\|_{\infty})^{\beta_1\vee\beta_2} (\|\dY\|_{\alpha}+\|Y\|_{\alpha})\nonumber\\
&\times \|\W-\tdW\|_{2\alpha,\bbeta^*_2, \bbeta^{**}_2}|t-s|^{3\alpha},
\end{align}
and
\begin{align}\label{b32}
\|J_3^2\|_V\leq &\Big[(1+T^{\alpha})\|\tdW\|_{2\alpha,\bbeta^*_2,\bbeta^{**}_2} (1+2\|Y\|_{\infty}+2\|\tY\|_{\infty})^{\beta^{**}_2}(1+2\|\dY\|_{\infty}+2\|\dtY\|_{\infty})^{\beta^*_2+\beta_0}\nonumber\\
&\quad \times (\|Y\|_{\alpha}+\|\dY\|_{\alpha}+\|\tY\|_{\alpha}+\|\dtY\|_{\alpha})\|W-\tW\|_{\alpha,\bbeta_3}\nonumber\\
&+3(1+T^{\alpha})(\|\tbW\|_{\C_3}+\|\tbW\|_{\C_3}^2)(\|Y\|_{\alpha}+\|\dY\|_{\alpha}+\|\tY\|_{\alpha}+\|\dtY\|_{\alpha})\nonumber\\
&\quad \times (1+2\|Y\|_{\infty}+2\|\tY\|_{\infty})^{\beta^{**}_2}(1+2\|\dY\|_{\infty}+2\|\dtY\|_{\infty})^{\beta^*_2+\beta_1}\nonumber\\
&\quad \times \big(\|Y_0-\tY_0\|_V+\|\dY_0-\dtY_0\|_V\big)\nonumber\\
	 &+2(1+T^{\alpha})(\|\tbW\|_{\C_3}+\|\tbW\|_{\C_3}^2) \big(1+ T^{\alpha}(\|Y\|_{\alpha}+\|\dY\|_{\alpha}+\|\tY\|_{\alpha}+\|\dtY\|_{\alpha})\big)\nonumber\\
&\quad \times (1+2\|Y\|_{\infty}+2\|\tY\|_{\infty})^{\beta^{**}_2}(1+2\|\dY\|_{\infty}+2\|\dtY\|_{\infty})^{\beta^*_2+\beta_1}\nonumber\\
&\quad \times d_{\alpha, W,\tW}\big((Y,\dY), (\tY,\dtY)\big)\Big]|t-s|^{3\alpha}.
\end{align}
Therefore, combining (\ref{b0}) and (\ref{b11}) - (\ref{b32}), we have

\begin{align}\label{dtdlt}
&\|\delta \Delta_{s,u,t}\|_V\nonumber\\
\leq	 &	 \Big\{(1+T^{\alpha})(1+\|\tbW\|_{\C_3})(1+2\|Y\|_{\infty}+2\|\tY\|_{\infty})^{\beta^{**}_2}(1+2\|\dY\|_{\infty}+2\|\dtY\|_{\infty})^{\beta^*_2+\beta_0}\nonumber\\
	 &\quad \times \big[\|Y\|_{\alpha}+\|\tY\|_{\alpha}+(\|Y\|_{\alpha}+\|\tY\|_{\alpha})^2+\|R^{Y}\|_{2\alpha}\big]\vrho_{\alpha,\bbeta_3}(\bW,\tbW)\nonumber\\
	 &+4(1+T^{\alpha})(\|\tbW\|_{\C_3}+\|\tbW\|_{\C_3}^2) (1+2\|Y\|_{\infty}+2\|\tY\|_{\infty})^{\beta^{**}_2}(1+2\|\dY\|_{\infty}+2\|\dtY\|_{\infty})^{\beta^*_2+\beta_1}\nonumber\\
	 &\quad \times \big[\|Y\|_{\alpha}+\|\dY\|_{\alpha}+\|\tY\|_{\alpha}+\|\dtY\|_{\alpha}+ (\|Y\|_{\alpha}+\|\tY\|_{\alpha})^2+\|R^{Y}\|_{2\alpha}\big]\nonumber\\
	 &\quad \times \big(\|Y_0-\tY_0\|_V+\|\dY_0-\dtY_0\|_V\big)\nonumber\\
	 &+4(1+T^{\alpha})(\|\tbW\|_{\C_3}+\|\tbW\|_{\C_3}^2) (1+2\|Y\|_{\infty}+2\|\tY\|_{\infty})^{\beta^{**}_2}(1+2\|\dY\|_{\infty}+2\|\dtY\|_{\infty})^{\beta^*_2+\beta_1}\nonumber\\
	 &\quad \times \big[1+ T^{\alpha}(\|Y\|_{\alpha}+\|\dY\|_{\alpha}+\|\tY\|_{\alpha}+\|\dtY\|_{\alpha})+T^{2\alpha}(\|Y\|_{\alpha}+\|\tY\|_{\alpha})^2+T^{2\alpha}\|R^{Y}\|_{2\alpha}\big]\nonumber\\
&\quad \times d_{\alpha, W,\tW}\big((Y,\dY), (\tY,\dtY)\big)\Big\}|t-s|^{3\alpha}.
\end{align}
	 On the other hand, by (\ref{crsd2}) and (\ref{dtvphi1}), we can show that
	 \begin{align}\label{dbw}
	 &\|\W_{s,t}(\dY_s, Y_s)-\tdW_{s,t}(\dtY_s, \tY_s)\|_V=(\W-\tdW)_{s,t}(\dY_s,Y_s)-\sD^{\tdW}_{s,t}\big((\dY_s,Y_s),(\dtY_s,\tY_s)\big)\nonumber\\
	 \leq &\Big\{(1+T^{\alpha})(1+\|\tbW\|_{\C_3})(1+\|\dY\|_{\infty}+\|\dtY\|_{\infty})^{2\beta_0\vee\beta_1}(1+\|Y\|_{\infty}+\|\tY\|_{\infty})^{\beta_1\vee\beta_2}\nonumber\\
	 &\quad \times\vrho_{\alpha,\bbeta_3}(\bW,\tbW)\nonumber\\
	 &+2(1+T^{\alpha})(\|\tbW\|_{\C_3}+\|\tbW\|_{\C_3}^2)(1+\|\dY\|_{\infty}+\|\dtY\|_{\infty})^{\beta_0\vee\beta_1+\beta_1}(1+\|Y\|_{\infty}+\|\tY\|_{\infty})^{\beta_1\vee\beta_2}\nonumber\\
	 &\quad \times \big(\|Y_0-\tY_0\|_V+\|\dY_0-\dtY_0\|_V\big)\nonumber\\
	 &+2T^{\alpha}(1+T^{\alpha})(\|\tbW\|_{\C_3}+\|\tbW\|_{\C_3}^2)(1+\|\dY\|_{\infty}+\|\dtY\|_{\infty})^{\beta_0\vee\beta_1+\beta_1}(1+\|Y\|_{\infty}+\|\tY\|_{\infty})^{\beta_1\vee\beta_2}\nonumber\\
	 &\quad \times d_{\alpha, W,\tW}\big((Y,\dY), (\tY,\dtY)\big)\Big\}|t-s|^{2\alpha}.
	 \end{align}

Notice that by Proposition \ref{ricrp}, we know that
\begin{align}\label{drze}
\|R^Z_{s,t}-R^{\tZ}_{s,t}\|_V\leq \|Z_{s,t}-\tZ_{s,t}-\Delta_{s,t}\|_V+\|\W_{s,t}(\dY_s, Y_s)-\tdW_{s,t}(\dtY_s, \tY_s)\|_V.
\end{align}
The inequality (\ref{dwyz}) follows from (\ref{dtvphin}), (\ref{dtdlt}) - (\ref{drze}) and the sewing lemma.
 \end{proof}

\section{Rough Differential Equations}\label{s.4}
  Let $\alpha\in (\frac{1}{3},\frac{1}{2}]$, $\bbeta_3=(\beta_0,\dots,\beta_3)$ where $\beta_k\geq 0$, $k=0,\dots,3$, and let $\bW=(W,\W)\in \C^{\alpha,\bbeta_3}([0,T]\times V; V)$. That is $W$ is $\alpha$-H\"{o}lder in time, and three times differentiable in space with growth multi-index $\bbeta_3$, $\W$ is $2\alpha$-H\"{o}lder in time and twice differentiable in space with growth multi-indexes $\bbeta_2^*=(\beta_0,\beta_0\vee \beta_1,\beta_0\vee\beta_1\vee\beta_2)$ and $\bbeta^{**}_2=(\beta_1,\beta_1\vee\beta_2,\beta_1\vee\beta_2\vee\beta_3)$, and $(W,\W)$ satisfies Chen's relation \eqref{chen}.
 	 Consider the following nonlinear RDE: 
	 \begin{align}\label{rde}
	 Y_t=\xi+\int_0^tW(dr,Y_r).
	 \end{align}
	 \begin{definition}\label{dsol}
An $\alpha$-H\"{o}lder continuous function $Y$ is said to be a solution to (\ref{rde}), if $(Y, Y)\in \sE_{W,(\xi,\xi)}^{2\alpha}$, and equality (\ref{rde}) holds for all $t\in [0,T]$ where the integral on the right-hand side is a nonlinear rough integral in the sense of Theorem \ref{nri}. 
\end{definition}
\subsection{Local existence}\label{s.4.1}
	 In this section, we establish the (local) existence of a solution for equation (\ref{rde}) using the Picard iteration method. To this end, we introduce the following notation. Let $\Phi\in \cC^{\alpha}([0,T];V)$, for any $0\leq s< t\leq T$, we make use of the notation
	 \[
	 \|\Phi\|_{\alpha,[s,t]}:=\sup_{u\neq v\in [s,t]}\frac{\|\Phi_{u,v}\|_V}{|v-u|^{\alpha}}.
	 \]
	 We also define $d_{W,\alpha, [s,t]}$ in a similar way. 
	 \begin{theorem}\label{rdelc}
For any $\xi\in V$, there exist a positive number $h$, such that the RDE \eqref{rde} has a solution $Y$ on $[0,h]$ with initial condition $Y_0=\xi$. In addition, the following inequality holds on $[0,h]$:
\begin{align}\label{maxyr0}
\|Y\|_{\alpha, [0,h]}\leq \min\Big\{&(1+({\color{blue}6}\gamma_1)^{-1})^{1+\beta_0}(1+\|\bW\|_{\C_3})(1+\|\xi\|_V)^{\beta_0},\nonumber\\
 &2k_{\alpha}\Big(\frac{1+\gamma_1}{\gamma_1}\Big)^{1+\gamma_1}\|\bW\|_{\C_3}(1+{\color{blue}2}\|\xi\|_V)^{\gamma_1}\Big\},
\end{align}
where $\gamma_1=\beta_0\vee\beta_1+\beta_1\vee\beta_2$.
		 \end{theorem}
\begin{proof}
Choose $h\in (0,1]$. Let 
\[
(Y^0_t,\dY^0_t):=(\xi+W_{0,t}(\xi), \xi), \quad t\in [0,h].
\]
Then $(Y^0,\dY^0)\in \sE_{W,(\xi,\xi)}^{2\alpha}$ with the remainder $R^{Y^0}_{s,t}\equiv 0$ for all $(s,t)\in [0,h]^2$. Due to Proposition \ref{ricrp}, for any $n\geq 1$, we can recursively define an element $(Y^n, \dot{Y}^n)\in \sE_{W,(\xi,\xi)}^{2\alpha}$ given by
\[
Y^{n+1}_t=\xi+\int_0^t W(dr, Y^n_t), \quad t\in [0,h].
\]
 By (\ref{ermd}), the following inequality holds for all $n\geq 1$
	\begin{align}\label{drrde}
	\|R^{Y^{n+1}}\|_{2\alpha, [0,h]}	\leq &k_{\alpha}\|\bW\|_{\C_3}(1+{\color{blue}2}\|\xi\|_V)^{\gamma_1}\\
	&\times \big[1+{\color{blue}2}h^{\alpha}\big(\|Y^n\|_{\alpha, [0,h]}+\|Y^{n-1}\|_{\alpha, [0,h]}+\|R^{Y^n}\|_{2\alpha, [0,h]}\big)\big]^{1+\gamma_1}.\nonumber
	\end{align}
 By iteration, we know that $(Y^{n+1},Y^n)\in \sE_W^{2\alpha,(\xi,\xi)}$, which implies that
\begin{align}\label{dyrde0}
\|Y^{n+1}\|_{\alpha, [0,h]}\leq \|W\|_{\alpha,\bbeta_3}(1+\|\xi\|_V+h^{\alpha}\|Y^n\|_{\alpha, [0,h]})^{\beta_0}+h^{\alpha}\|R^{Y^{n+1}}\|_{2\alpha, [0,h]}.
\end{align}
Substituting the inequality \eqref{drrde} into \eqref{dyrde0}, we have the following estimate
\begin{align}\label{dyrde}
\|Y^{n+1}&\|_{\alpha, [0,h]}\leq  2k_{\alpha}\|\bW\|_{\C_3}(1+{\color{blue}2}\|\xi\|_V)^{\gamma_1}\nonumber\\
&\times\big[1+{\color{blue}2}h^{\alpha}\big(\|Y^n\|_{\alpha, [0,h]}+\|Y^{n-1}\|_{\alpha, [0,h]}+\|R^{Y^n}\|_{2\alpha, [0,h]}\big)\big]^{1+\gamma_1}.
\end{align} 

Choose $h_1=\big[{\color{blue}12}k_{\alpha}\|\bW\|_{\C_3}(1+\gamma_1)^{1+\gamma_1}\gamma_1^{-\gamma_1}(1+{\color{blue}2}\|\xi\|_V)^{\gamma_1}\big]^{-\frac{1}{\alpha}}\wedge 1$ (by convention $0^0:=1$), and let $h\in (0,h_1]$. We claim that $\|Y^n\|_{\alpha, [o,h]}$ and $\|R^n\|_{2\alpha, [0,h]}$ are bounded uniformly in $n$ for $h\in (0,h_1]$. To this end, for any $h\in(0, h_1]$, let $f_h: \R_+\to\R_+ $ be given by
\[
f_h(x):=2k_{\alpha}\|\bW\|_{\C_3}(1+{\color{blue}2}\|\xi\|_V)^{\gamma_1}(1+{\color{blue}6}h^{\alpha}x)^{1+\gamma_1},
\]
and let $g_h(x)=f_h(x)-x$ for all $x\in\R_+$. It is easily to show that $g_{h_1}$ has a unique local minimum at
\begin{align*}
x_1=&({\color{blue}6}h_1^{\alpha})^{-1}\big[\big({\color{blue}12}k_{\alpha}\|\bW\|_{\C_3}(1+{\color{blue}2}\|\xi\|_V)^{\gamma_1}(1+\gamma_1)h_1^{\alpha}\big)^{-\frac{1}{\gamma_1}}-1\big]\\
=&2k_{\alpha}\|\bW\|_{\C_3}(1+{\color{blue}2}\|\xi\|_V)^{\gamma_1}\Big(\frac{1+\gamma_1}{\gamma_1}\Big)^{1+\gamma_1},
\end{align*}
and
\begin{align*}
g_{h_1}(x_1)=&2k_{\alpha}\|\bW\|_{\C_3}(1+{\color{blue}2}\|\xi\|_V)^{\gamma_1}\big({\color{blue}12}k_{\alpha}\|\bW\|_{\C_3}(1+{\color{blue}2}\|\xi\|_V)^{\gamma_1}(1+\gamma_1)h_1^{\alpha}\big)^{-\frac{1+\gamma_1}{\gamma_1}}\\
&-({\color{blue}6}h_1^{\alpha})^{-1}\big[\big({\color{blue}12}k_{\alpha}\|\bW\|_{\C_3}(1+{\color{blue}2}\|\xi\|_V)^{\gamma_1}(1+\gamma_1)h_1^{\alpha}\big)^{-\frac{1}{\gamma_1}}-1\big]= 0.
\end{align*}
Therefore, for any $x\in [0,x_1]$, the following inequality holds:
\[
f_h(x)\leq f_{h_1}(x)\leq f_{h_1}(x_1)=x_1.
\]
Notice that $\|\dY^0\|_{\alpha, [0,h]}=\|R^{Y^0}\|_{2\alpha, [0,h]}=0$. From the inequalities (\ref{drrde}) and (\ref{dyrde}) we can show by a recursive argument that
\begin{align}\label{maxyr}
\max_{n\geq 0}\{\|Y_n\|_{\alpha, [0,h]}, \|R^{Y_n}\|_{2\alpha, [0,h]}\}\leq x_1=2k_{\alpha}\|\bW\|_{\C_3}(1+{\color{blue}2}\|\xi\|_V)^{\gamma_1}\Big(\frac{1+\gamma_1}{\gamma_1}\Big)^{1+\gamma_1},
\end{align}
provided that $\|Y^0\|_{\alpha, [0,h]}\leq x_1$. Indeed, by definition, we know that $\|\dY^0\|_{\alpha}= \|R^{Y^0}\|_{2\alpha}=0$, and the following estimate holds:
\[
\|Y^0\|_{\alpha, [0,h]}\leq \|W\|_{\alpha,\bbeta_3}(1+\|\xi\|_V)^{\gamma_1}\leq 2k_{\alpha}\|\bW\|_{\C_3}(1+{\color{blue}2}\|\xi\|_V)^{\gamma_1}=f_{h_1}(0)\leq x_1.
\]
 As a consequence, we conclude that $\|Y^n\|_{\alpha, [0,h]}$ and $\|R^n\|_{2\alpha, [0,h]}$ are bounded uniformly in $n$ for $h\in (0,h_1]$. This  also yields that
\begin{align}\label{maxyr1}
\max_{n\geq 0}\{\|Y_n\|_{\infty,[0,h]}\}\leq \|\xi\|_V+(6\gamma_1)^{-1}.
\end{align}
By \eqref{maxyr}, \eqref{maxyr1}, Proposition \ref{nic1} and the fact that $0<h\leq h_1=\big[{\color{blue}12}k_{\alpha}\|\bW\|_{\C_3}(1+\gamma_1)^{1+\gamma_1}\gamma_1^{-\gamma_1}(1+{\color{blue}2}\|\xi\|_V)^{\gamma_1}\big]^{-\frac{1}{\alpha}}\wedge 1$, we get the follow estimate
\[
d_{\alpha, W, [0,h]}\big((Y^{n+1},Y^n), (Y^n,Y^{n-1})\big)\leq  C_5 d_{\alpha,W,[0,h]}\big((Y^n,Y^{n-1}), (Y^{n-1},Y^{n-2})\big)
\]
and
\begin{align*}
C_5=&6k_{\alpha}h^{\alpha}(1+h^{\alpha})(1+\|\bW\|_{\C_3})^2(1+2\|Y^{n}\|_{\infty}+2\|Y^{n-1}\|_{\infty})^{\beta^{**}_2}\\
&\times (1+2\|Y^{n-1}\|_{\infty}+2\|Y^{n-2}\|_{\infty})^{\beta_2^*+\beta_1}\\
&\times \Big[1+h^{\alpha}(2\|Y^{n-1}\|_{\alpha}+\|Y^n\|_{\alpha}+\|Y^{n-2}\|_{\alpha})+h^{2\alpha}(\|Y^{n-1}\|_{\alpha}+\|Y^n\|_{\alpha})^2+h^{2\alpha}\|R^{Y^n}\|_{2\alpha}\Big]\\
\leq& 2(18\gamma_1^2+15\gamma_1+2)(3\gamma_1^2)^{-1}k_{\alpha}(1+\|\bW\|_{\C_3})^2\big[(3\gamma_1+2)(3\gamma_1)^{-1}+4\|\xi\|_V\big]^{\beta^{**}_2+\beta^*_2+\beta_1}h^{\alpha}\\
\leq& 2(18\gamma_1^2+15\gamma_1+2)(3\gamma_1^2)^{-1}k_{\alpha}\big[(3\gamma_1+2)(3\gamma_1)^{-1}+4\big]^{\beta^{**}_2+\beta^*_2+\beta_1}(1+\|\bW\|_{\C_3})^2\\
&\times (1+\|\xi\|_V)^{\beta^{**}_2+\beta^*_2+\beta_1}h^{\alpha}.
\end{align*}
Let $\gamma_2={\beta^{**}_2+\beta^*_2+\beta_1}=\max\{\beta_0,\beta_1,\beta_2\}+\max\{\beta_1,\beta_2,\beta_3\}+\beta_1$, and let
\begin{align}\label{c6}
C_6=\max\big\{&2(18\gamma_1^2+15\gamma_1+2)(3\gamma_1^2)^{-1}k_{\alpha}\big[(3\gamma_1+2)(3\gamma_1)^{-1}+4\big]^{\gamma_2},\nonumber\\
&12k_{\alpha}(1+\gamma_1)^{1+\gamma_1}\gamma_1^{-\gamma_1}\big\}.
\end{align}
 Then, we have
\begin{align}\label{cbbeta}
&d_{\alpha,W,[0,h]}\big((Y^{n+1},Y^n),(Y^n,Y^{n-1})\big)\nonumber\\
\leq&C_6(1+\|\bW\|_{\C_3})^2(1+\|\xi\|_V)^{\gamma_2}h^{\alpha}d_{\alpha, W, [0,h]}\big((Y^n,Y^{n-1}),(Y^{n-1},Y^{n-2})\big).
\end{align} 
 Choose $h_2=[2C_6(1+\|\bW\|_{\C_3})^2(1+\|\xi\|_V)^{\gamma_2}]^{-\frac{1}{\alpha}}\leq h_1$, and let $h\in (0, h_2]$. Then by \eqref{cbbeta}, we have the following inequality
\[
d_{\alpha, W}\big((Y^{n+1},Y^n),(Y^n,Y^{n-1})\big)\leq \frac 12 d_{\alpha, W}\big((Y^n,Y^{n-1}),(Y^{n-1},Y^{n-2})\big).
\]
 This yields that
\[
\sum_{n=1}^{\infty}d_{\alpha, W, [0,h]}\big((Y^{n+1},Y^n),(Y^n,Y^{n-1})\big)<\infty.
\]
 Due to Lemma \ref{cmplt}, we can conclude that $(Y^n, Y^{n-1})\to (Y, Y)\in \sE_{W,(\xi,\xi)}^{2\alpha}$ as $n\to \infty$. By Proposition \ref{nic1} again, we have for any $0\leq s\leq t\leq h$,
 \begin{align*}
\Big\|Y^{n+1}_{s,t}-\int_s^tW(dr, Y_r)\Big\|_V=&\Big\|\int_s^tW(dr,Y^n_r)-\int_s^tW(dr, Y_r)\Big\|_V\\
\leq &C d_{\alpha, W, [0,h]}\big((Y^n,Y^{n-1}), (Y,Y)\big)|t-s|^{\alpha},
 \end{align*}
for some constant $C>0$ uniformly in $n$. This implies that equation (\ref{rde}) holds for all $t\in [0,h]$. Finally, the inequality (\ref{maxyr0}) follows from (\ref{dyrde0}) and (\ref{maxyr}) and the fact that $(Y,Y)$ is the limit of $(Y^n,Y^{n-1})$ in $\sE_{W,(\xi,\xi)}^{2\alpha}$.
\end{proof}

\subsection{Uniqueness and global existence}\label{s.4.2}
In this section, we prove the uniqueness of a solution for equation (\ref{rde}). We also present some hypotheses that imply the global existence of a solution for this equation.

\begin{theorem}\label{unique}
For any time interval $[0,T]$ and initial value $\xi\in V$. There exists at most one solution to equation (\ref{rde}).
\end{theorem}
\begin{proof}
Suppose that $Y$ and $\tY$ are two solutions to (\ref{rde}) with initial condition $\xi$ on $[0,T]$. By Proposition \ref{nic1}, the following inequality holds on $[0,h]\subset [0,T]$, assuming $h\le 1$.
\begin{align}\label{dyty}
d_{\alpha,W,[0,h]}\big((Y,Y),(\tY,\tY)\big)\leq C_5d_{\alpha, W,[0,h]}\big((Y,Y),(\tY,\tY)\big),
\end{align}
where
\begin{align*}
C_5=&12k_{\alpha}h^{\alpha}(1+\|\bW\|_{\C_3})^2 (1+2\|Y\|_{\infty}+2\|\tY\|_{\infty})^{\beta^{**}_2}(1+2\|\dY\|_{\infty}+2\|\dtY\|_{\infty})^{\beta^*_2+\beta_1}\\
	 &\quad \times \big[1+ (\|Y\|_{\alpha}+\|\dY\|_{\alpha}+\|\tY\|_{\alpha}+\|\dtY\|_{\alpha})+(\|Y\|_{\alpha}+\|\tY\|_{\alpha})^2+\|R^{Y}\|_{2\alpha}\big].
\end{align*}
Choosing $h$ small enough, (\ref{dyty}) yields that $Y\equiv \tY$ on $[0,h]$. Notice that the choice of $h$ doesn't dependent on the initial value. Therefore, by iteration, we can extend the uniqueness to any time interval $[0,T]$.
\end{proof}

As stated in Section \ref{s.2}, the linear growth of the vector field cannot guarantee the global existence of a RDE driven by a linear rough path. This is also true in the case of nonlinear rough paths. In order to obtain the global existence, we introduce the following growth condition of $W$. Let $\bW=(W,\W)\in \C^{\alpha,\bbeta_3}([0,T]\times V;V)$ and let $\gamma_2=\max\{\beta_0,\beta_1,\beta_2\}+\max\{\beta_1,\beta_2,\beta_3\}+\beta_1$.

\begin{hyp}{(H)}\label{hyp1} $\frac{\gamma_2}{\alpha}-\gamma_2+\beta_0\leq 1$. 
\end{hyp}
A similar condition in the linear situation can be seen, e.g., in \cite{sd-11-besalu-nualart,ejp-09-lejay}.

\begin{theorem}\label{globn}
Under Hypothesis \ref{hyp1}, the RDE \eqref{rde} has a solution on any time interval $[0,T]$. By Theorem \ref{unique}, this solution is unique.
\end{theorem}
\begin{proof}
Let $\epsilon_1=[2C_6(1+\|\bW\|_{\C_3})^2(1+\|\xi\|_V)^{\gamma_2}]^{-\frac{1}{\alpha}}$ where $C_6$ is the constant appearing in \eqref{c6}. Then, by Theorem \ref{rdelc}, the RDE has a solution $ (Y^{(1)},Y^{(1)})$ on $[0,\epsilon_1]$ with initial condition $Y^{(1)}_0=\xi$. We denote by $\xi_1=Y^{(1)}_{\epsilon_1}$ the terminal value of $Y$. Consider the following RDE
\begin{align}\label{rdest}
Y_t=Y_s+\int_s^tW(dr, Y_r)+\int_s^t (dr, Y_r).
\end{align}
By Theorem \ref{rdelc} again, the equation (\ref{rdest}) has a solution $(Y^{(2)}, Y^{(2)})$ on $[\epsilon_1,\epsilon_1+\epsilon_2]$ with initial condition $\xi_1$, where $\epsilon_2= [2C_6(1+\|\bW\|_{\C_3})^2(1+\|\xi_1\|_V)^{\gamma_2}]^{-\frac{1}{\alpha}}$. By iteration, we have a sequence $\{\epsilon_n\}_{n\geq 1}$ with values in $(0,1)$, such that the equation (\ref{rdest}) has a solution $Y^{(n+1)}$ on $[\eta_n,\eta_{n+1}]:=[\sum_{k=1}^n \epsilon_k, \sum_{k=1}^{n+1} \epsilon_{n+1}]$ with initial condition $Y^{(n+1)}_{\eta_n}=\xi_n:=Y^{(n)}_{\eta_n}$ and $\epsilon_{n+1}= [2C_6(1+\|\bW\|_{\C_3})^2(1+\|\xi_n\|_V)^{\gamma_2}]^{-\frac{1}{\alpha}}$. By (\ref{maxyr0}) we have the following inequality
\begin{align*}
\|\xi_{n+1}\|_V\leq \|Y^{(n+1)}\|_{\infty}\leq &\|\xi_n\|_V+\epsilon_{n+1}^{\alpha}\|Y^{(n+1)}\|_{\alpha}\nonumber\\
\leq& \|\xi_n\|_V+C_7(1+\|\bW\|_{\C_3})^{-1}(1+\|\xi_n\|_V)^{\beta_0-\gamma_2},
\end{align*}
where 
 \[
 C_7=(2C_6)^{-1}\big(1+(6\gamma_1)^{-1}\big)^{1+\beta_0}
\]
depends only on $\alpha,\beta_0,\dots,\beta_3$. Recall the assumption $\frac{\gamma_2}{\alpha}-\gamma_2+\beta_0\leq 1$. By the mean value theorem, there exist $\tau \in [0,1]$, such that
\begin{align*}
(1+\|\xi_{n+1}&\|_V)^{\frac{\gamma_2}{\alpha}}\leq \big[1+ \|\xi_n\|_V+C_7(1+\|\bW\|_{\C_3})^{-1}(1+\|\xi_n\|_V)^{\beta_0-\gamma_2}\big]^{\frac{\gamma_2}{\alpha}}\nonumber\\
=&(1+\|\xi_n\|_V)^{\frac{\gamma_2}{\alpha}}+\big[C_7(1+\|\bW\|_{\C_3})^{-1}(1+\|\xi_n\|_V)^{\beta_0-\gamma_2}\big]\nonumber\\
&\times \frac{\gamma_2}{\alpha}\big[1+\|\xi_n\|_V+\tau C_7(1+\|\bW\|_{\C_3})^{-1}(1+\|\xi_n\|_V)^{\beta_0-\gamma_2}\big]^{\frac{\gamma_2}{\alpha}-1}
\end{align*}
By definition we know that $\beta_0\leq \gamma_2$. This implies 
\begin{align*}
\big[1+\|\xi_n\|_V&+\tau C_7(1+\|\bW\|_{\C_3})^{-1}(1+\|\xi_n\|_V)^{\beta_0-\gamma_2}\big]^{\frac{\gamma_2}{\alpha}-1}\\
\leq &(1+\|\xi_n\|_V)^{\frac{\gamma_2}{\alpha}-1}\times \max\{1,(1+C_7(1+\|\bW\|_{\C_3})^{-1})^{\frac{\gamma_2}{\alpha}-1}\}\\
\leq &(1+C_7)^{\frac{\gamma_2}{\alpha}}(1+\|\xi_n\|_V)^{\frac{\gamma_2}{\alpha}-1}.
\end{align*}
As a consequence of above two inequalities, under the assumption \ref{hyp1}, we can write
\begin{align}\label{xin1}
(1+\|\xi_{n+1}\|_V)^{\frac{\gamma_2}{\alpha}}\leq &(1+\|\xi_n\|_V)^{\frac{\gamma_2}{\alpha}}+ \frac{\gamma_2}{\alpha}C_7(1+C_7)^{\frac{\gamma_2}{\alpha}}.
\end{align}
It follows that
\begin{align*}
\epsilon_{n+1}\geq & [2C_6(1+\|\bW\|_{\C_3})^2]^{-\frac{1}{\alpha}}\Big[(1+\|\xi_n\|_V)^{\frac{\gamma_2}{\alpha}}+ \frac{\gamma_2}{\alpha}C_7(1+C_7)^{\frac{\gamma_2}{\alpha}}\Big]^{-1}\\
= &\Big[\epsilon_n^{-1}+\big(2C_6(1+\|\bW\|_{\C_3})^2\big)^{\frac{1}{\alpha}}\frac{\gamma_2}{\alpha}C_7(1+C_7)^{\frac{\gamma_2}{\alpha}}\Big]^{-1}:=(\epsilon_n^{-1}+K_0)^{-1}.
\end{align*}
Observe that the constant $K_0$ is independent of $n$. Thus by iteration, the following inequality holds
\begin{align}\label{sumep}
\sum_{n=1}^{\infty}\epsilon_n\geq \sum_{n=0}^{\infty}\frac{1}{\epsilon_1^{-1}+nK_0}=\infty.
\end{align}
In other words, we can extend the solution to any time interval $[0,T]$. 
\end{proof}

Assume that the derivatives of $W$ are all bounded, that is $\bbeta_3=(\beta_0,0,0,0)$. Then, Hypothesis \ref{hyp1} is equivalent to $\beta_0\leq \alpha$ and it coincides with Besal\'{u} and Nualart's condition for global existence (see Theorem 4.1 of \cite{sd-11-besalu-nualart}).

\subsection{Properties of the solutions}

Assume Hypothesis \ref{hyp1}. In this section, we prove some properties of the solution to the RDE (\ref{rde}). The first proposition below provides an estimate for the H\"{o}lder norm of the solution to (\ref{rde}).
\begin{proposition}\label{trdehn}
Assume that $\bW=(W,\W)$ satisfies the conditions in Theorem \ref{globn}. Let $Y$ be the solution to the RDE (\ref{rde}) with initial condition $\xi\in V$. Then the following estimate holds:
\begin{align}\label{rdehn}
\|Y\|_{\alpha}\leq c\|\bW\|_{\C_3}(1+\|\bW\|_{\C_3})^{\frac{2\gamma_1}{\gamma_2}}(1+\|\xi\|_V)^{\gamma_1}e^{\frac{\alpha\gamma_1 K_0T}{\gamma_2}}
\end{align}
for some $c$ depending on $\alpha$ and $\bbeta_3$.
\end{proposition}
\begin{proof}
Let $\epsilon_1=[2C_6(1+\|\bW\|_{\C_3})^2(1+\|\xi\|_V)^{\gamma_2}]^{-\frac{1}{\alpha}}$ where $C_6$ is the same as in \eqref{cbbeta}. Theorems \ref{rdelc} and \ref{unique} implies that there exists a unique solution to (\ref{rde}) with initial condition $Y_0=\xi$ on $[0,\epsilon_1]$. Denote the solution by $Y^{(1)}$. Then, by proceeding a similar argument in Theorem \ref{globn}, we obtain $\{Y^{(n+1)}\}_{n\geq 1}$, where $Y^{(n+1)}$ is the unique solution to RDE (\ref{rdest}) on $[\eta_n,\eta_{n+1}]=[\sum_{k=1}^n\epsilon_k,\sum_{k=1}^{n+1}\epsilon_k]$ with initial condition $Y^{(n+1)}_{\eta_n}:=\xi_n=Y^{(n)}_{\eta_n}$ and $\epsilon_{n+1}= [2C_6(1+\|\bW\|_{\C_3})^2(1+\|\xi_n\|_V)^{\gamma_2}]^{-\frac{1}{\alpha}}$. By (\ref{maxyr0}) and (\ref{xin1}), we have the following estimate:
\begin{align}\label{yn+1}
\|Y^{(n+1)}\|_{\alpha}\leq &2^{1+\gamma_1}\Big(\frac{1+\gamma_1}{\gamma_1}\Big)^{1+\gamma_1}k_{\alpha}\|\bW\|_{\C_3} \Big[(1+\|\xi\|_V)^{\frac{\gamma_2}{\alpha}}+ n\frac{\gamma_2}{\alpha}C_7(1+C_7)^{\frac{\gamma_2}{\alpha}}\Big]^{{\color{blue}\frac{\alpha\gamma_1}{\gamma_2}}}.
\end{align}
On the other hand, for any $T>0$, there exists $N\in\N$, such that $\eta_N\leq T\leq \eta_{N+1}$. Notice that by (\ref{sumep}), we have
\begin{align*}
T\geq &\sum_{n=1}^N\epsilon_n\geq \sum_{n=1}^N(\epsilon_1^{-1}+K_0n)^{-1}\geq \frac{1}{K_0}\big(\log(\epsilon_1^{-1}+K_0N)-\log(\epsilon_1^{-1})\big).
\end{align*}
In other words, 
\begin{align}\label{N}
N\leq \frac{1}{K_0}\big(e^{K_0T+\log(\epsilon_1^{-1})}-\epsilon_1^{-1}\big)=K_0^{-1}(2C_6(1+\|\bW\|_{\C_3})^2)^{\frac{1}{\alpha}}(1+\|\xi\|_V)^{\frac{\gamma_2}{\alpha}}(e^{K_0T}-1).
\end{align}
Let $Y$ be the solution to (\ref{rde}) on $[0,T]$ with initial condition $\xi$. Then, combining (\ref{yn+1}) and (\ref{N}), we have
\begin{align*}
\|Y\|_{\alpha}\leq \max_{1\leq n\leq N+1}\|Y^{(n)}\|_{\alpha}\leq c\|\bW\|_{\C_3}(1+\|\bW\|_{\C_3})^{\frac{2\gamma_1}{\gamma_2}}(1+\|\xi\|_V)^{\gamma_1}e^{\frac{\alpha\gamma_1 K_0T}{\gamma_2}}
\end{align*}
for some $c$ depending on $\alpha$ and $\bbeta_3$.
\end{proof}

The next proposition provides the dependency of the solution to (\ref{rde}) on the initial condition under Hypothesis \ref{hyp1}.
\begin{proposition}\label{dpic}
Assume that $\bW=(W,\W)$ satisfies the conditions in Theorem \ref{globn}. Let $Y$ and $\tY$ be the solutions to the RDE (\ref{rde}) with initial conditions $\xi$ and $\txi$, respectively. Then the following estimate holds
\begin{align}\label{dic}
d_{\alpha,W}\big((Y,Y),(\tY,\tY)\big)\leq c^T\|\xi-\txi\|_V, 
\end{align}
where $c$ is a constant depending on $\alpha$, $\bbeta_3$, $T$, $\|\bW\|_{\C}$, $\xi$, and $\txi$.
\end{proposition}
\begin{proof}
By Propositions \ref{nic1}, \ref{trdehn}, and the fact that $Y$ and $\tY$ are solutions to \eqref{rde}, we can write
\begin{align}
d_{\alpha, W, [0,h]}\big((Y,Y),(\tY,\tY)\big)\leq& c_1\|\xi-\txi\|_V+c_2h^{\alpha}d_{\alpha, W, [0,h]}((Y,Y),(\tY,\tY))\big],
\end{align}
on $[0,h]\subset [0,T]$, where $c_1,c_2$ are constants depending on $\|\bW\|_{\C_3}$, $\alpha$, $\bbeta_3$, $T$, $\xi$ and $\txi$. Let $\epsilon =( 2c_2)^{-\frac{1}{\alpha}}\wedge 1$. It follows that 
\begin{align*}
d_{\alpha, W, [0,\epsilon]}\big((Y,Y),(\tY,\tY)\big) \leq 2c_1\|\xi-\txi\|_V
\end{align*}
on $[0,\epsilon]$. By iteration, we have that for any $n\geq 1$,
\begin{align*}
d_{\alpha, W, [n\epsilon,(n+1)\epsilon]}\big((Y,Y),(\tY,\tY)\big) \leq 2c_1\|Y_{n\epsilon}-\tY_{n\epsilon}\|_V,
\end{align*}
 and
\begin{align*}
\|Y_{n\epsilon}-\tY_{n\epsilon}\|_V\leq &\|Y_{(n-1)\epsilon}-\tY_{(n-1)\epsilon}\|_V+\epsilon^{\alpha}\|Y-\tY\|_{\alpha,[(n-1)\epsilon,n\epsilon]}\\
\leq &\|Y_{(n-1)\epsilon}-\tY_{(n-1)\epsilon}\|_V+\epsilon^{\alpha}d_{\alpha, W, [(n-1)\epsilon,n\epsilon]}\big((Y,Y),(\tY,\tY)\big)\\
\leq &2\|Y_{(n-1)\epsilon}-\tY_{(n-1)\epsilon}\|_V.
\end{align*}
Thus we can write
\begin{align*}
d_{\alpha, W, [n\epsilon,(n+1)\epsilon]}\big((Y,Y),(\tY,\tY)\big) \leq 2^{n+1}c_1\|\xi-\txi\|_V.
\end{align*}

Let $N$ be the integer such that $N\epsilon\leq T\leq (N+1)\epsilon$, then it follows that
\[
d_{\alpha, W}\big((Y,Y),(\tY,\tY)\big)\leq\max_{1\leq n\leq N}\big\{d_{\alpha, W, [n\epsilon,(n+1)\epsilon]}\big((Y,Y),(\tY,\tY)\big)\big\}\leq c^T\|\xi-\txi\|_V,
\]
for some $c>0$ depending on $\alpha$, $\bbeta_3$, $T$, $\|\bW\|_{\C}$, $\xi$, and $\txi$.
\end{proof}
Due to Propositions \ref{trdehn} and \ref{dpic}, we can deduce the following corollary.
 \begin{corollary}\label{corde}
Suppose that $\bW=(W,\W)$ satisfies the conditions in Theorem \ref{globn}. 
\begin{enumerate}[(i)]
\item Write $Y(\xi)$ for the solution to the RDE \eqref{rde} with initial condition $\xi\in V$. Then,  $\|Y(\xi)\|_{\alpha}$ is bounded uniformly in the space $\{\xi,\|\xi\|_V\leq K\}$ for any positive constant $K$.
\item The constant $c^T$ in \eqref{dic} is fixed in the space $\{(\xi,\txi),\|\xi\|_V+\|\txi\|_V\leq K\}$ for any positive constant $K$.
\end{enumerate}
\end{corollary}

\begin{remark}
As a consequence of Proposition \ref{dpic}, we have the following estimates
\[
\|Y-\tY\|_{\alpha}\leq c^T\|\xi-\txi\|_V,
\]
and
\[
\sup_{t\in[0,T]}\|Y_t-\tY_t\|_V\leq \big(1+c^TT^{\alpha}\big)\|\xi-\txi\|_V.
\]
\end{remark}

\section{Comparison of linear and nonlinear rough paths}\label{s.5}

\subsection{Nonlinear rough paths constructed by compositions}\label{s.5.1}
In this section, we consider a special class of nonlinear rough paths that are constructed by compositions of some nonlinear functions and linear rough paths.

\begin{definition}
Let $m$ be a positive integer. The space $\cC^{m,\bbeta_n}_{loc}(V^2;V)$ is the collection of function $f$ such that, for any compact set $K\subset V$
\begin{align}
\|f\|_{K,m,\bbeta_n}:=\sum_{j=0}^m\sum_{k=0}^n\sup_{\substack{x\in K\\y\in V}}\frac{\|D_2^{k}D_1^j f(x,y)\|_{\mathfrak{B}_{k+j}}}{(1+\|y\|_V)^{\beta_k}}<\infty
\end{align}
 where $D_1$ and $D_2$ are the partial derivatives of the first and second argument, respectively, and $\mathfrak{B}_{k+j}$ is the corresponding linear space of derivatives.
\end{definition}

Let $f\in \cC^{m,\bbeta_n}_{loc}(V^2;V)$, and let $\bX=(X,\X)\in \C^{\alpha}([0,T]; V)$ be a $V$-valued linear rough path. We aim to interpret $W(t,x)=f(X_t, x)$ as a nonlinear rough path with suitable parameters $m, n\in \N$. Due to Definition \ref{nrp}, an $\alpha$-H\"{o}lder nonlinear rough path contains a $\alpha$-H\"{o}lder continuous function $W$ and a $2\alpha$-H\"{o}lder continuous function $\W$ that defines a version of following double integral:
\[
\int_s^tDW(dr, y)W_{s,r}(x):=\W_{s,t}.
\]
As $W(t, x)=f(X_t, x)$, we expect that $\W$ is defined via the theory of linear rough paths by the following expression
\begin{align}\label{dw}
\W_{s,t}(x,y):=\int_s^t g(dr, y)(f(X_r, x))-g_{s,t}(y)(f(X_s, x)),
\end{align}
 where $g(t,y)=D_2f(X_t, y)$ and $g_{s,t}(y)=g(t,y)-g(s,y)$. Applying It\^{o}'s formula (Lemma \ref{itolrp}), the integral on the right-hand side of (\ref{dw}) can be defined as follows
 \begin{align}\label{dw1}
 \int_s^t g(dr, y)(f(X_r, x)):=&\int_s^t D_{21}f(X_r, y)f(X_r, x)d\bX_r\nonumber\\
 &+\frac{1}{2}\int_s^t D_{211}f(X_r,y)f(X_r,x)d\lgl X\rgl_r.
 \end{align}
In the next proposition, we will show that $(W,\W)$ is a nonlinear rough path where $W(t,x)=f(X_t,x)$ and $\W_{s,t}(x,y)$ is defined in \eqref{dw}.
\begin{proposition}\label{bwbx}
Assume that $n\geq 1$, and $f\in \cC^{3,\bbeta_n}_{loc}(V^2;V)$. Suppose that $\bX=(X,\X)\in \C^{\alpha}([0,T];V)$. Let $W(t,x)=f(X_t,x)$, and let $\W$ be defined by (\ref{dw}) and (\ref{dw1}). Then $\bW:=(W,\W)\in \C^{\alpha,\bbeta_n}([0,T]\times V;V)$.
\end{proposition}
\begin{proof}

 We prove this proposition by checking the properties in Definition \ref{nrp}. Let $K$ be the closed convex hull of the set $\{X_t,t\in[0,T]\}$. Then $K$ is a compact subset in $V$. 

(i) For any $k\in\{0,\dots, n\}$ and $\bz_k=(z_1,\dots,z_k)\in V^k$, by the mean value theorem, there exists $\xi$ between $X_s$ and $X_t$, such that
\begin{align*}
\|D^kW_{s,t}(x)(\bz_k)\|_V=&\|D^k_2f(X_t,x)(\bz_k)-D^k_2f(X_s,x)(\bz_k)\|_V\\
\leq & \|D_1D^k_2f(\xi,x)(\bz_k)(X_{s,t})\|_V\\
\leq & \|f\|_{K,3,\bbeta_n} \Big(\prod_{i=1}^k\|z_i\|_V\Big)\|X\|_{\alpha}(1+\|x\|_V)^{\beta_k}|t-s|^{\alpha}.
\end{align*}
This implies that $W\in \cC^{\alpha,\bbeta_n}([0,T]\times V;V)$.

(ii) a) Fix $(x,y)\in V^2$. Set $h(z)=h^{x,y}(z):=D_{21}f(z,y)(f(z, x))$ for all $z\in V$. Then, $h$ is an $\cL(V;V)$-valued function on $V$. It is easy to verify that $h\in \cC^2_{loc}(V;V)$.  Let $Y_t=h(X_t)$, and let 
\[
Y'_t=Dh(X_t)=D_{211}f(X_t,y)(f(X_t,x))+D_{21}f(X_t,y)D_1f(X_t,y),
\]
for all $t\in [0,T]$, where $D_{21}f(X_t,y)D_1f(X_t,y)$ is considered as an operator on $V\times V$ with values in $V$, that is
\[
D_{21}f(X_t,y)D_1f(X_t,y)(x_1,x_2)=D_{21}f(X_t,y)(D_1f(X_t,y)(x_2),x_1).
\]
 By Lemma \ref{cmpcr}, $(Y,Y')\in \sD_X^{2\alpha}(\cL(V;V))$. In addition,  by the mean value theorem, we can easily show that
\begin{align}\label{bwbx1}
\|Y'\|_{\alpha}\leq 2\|f\|_{K,3,\bbeta_n}^2(1+\|x\|_V)^{\beta_0}(1+\|y\|_V)^{\beta_1}\|X\|_{\alpha}
\end{align}
and
\begin{align}\label{bwbx2}
\|R^Y\|_{2\alpha}\leq 2\|f\|_{K,3,\bbeta_n}^2(1+\|x\|_V)^{\beta_0}(1+\|y\|_V)^{\beta_1}\|X\|_{\alpha}^2.
\end{align}
Let $\Xi_{s,t}:=Y_sX_{s,t}+Y'_s\X_{s,t}$ for any $(s,t)\in[0,T]^2$. The following estimate follows from (\ref{bwbx1}), (\ref{bwbx2}) and Theorem \ref{tlri}:
\begin{align}\label{xzxi1}
\Big\|\int_s^tY_r d\bX_r-\Xi_{s,t}\Big\|_V\leq &k_{\alpha}(\|X\|_{\alpha}\|R^Y\|_{2\alpha}+\|\X\|_{2\alpha}\|Y'\|_{\alpha})|t-s|^{3\alpha}\nonumber\\
\leq &2k_{\alpha}\big[ \|f\|^2_{K,3,\bbeta_n}(1+\|x\|_V)^{\beta_0}(1+\|y\|_V)^{\beta_1}\|X\|_{\alpha}^3\nonumber\\
&+\|f\|^2_{K,3,\bbeta_n}(1+\|x\|_V)^{\beta_0}(1+\|y\|_V)^{\beta_1}\|X\|_{\alpha}\|\X\|_{2\alpha}\big]|t-s|^{3\alpha}.
\end{align}
On the other hand, by Taylor's theorem, there exists $\xi=cX_s+(1-c)X_t$ for some $c\in [0,1]$ such that 
\begin{align*}
&\Xi_{s,t}-g_{s,t}(y)(f(X_s, x))\\
=&D_{21}f(X_s,y)(f(X_s,y), X_{s,t})+D_{211}f(X_s,y)(f(X_s,x),\X_{s,t})\\
&+D_{21}f(X_s,y)D_1f(X_s,y)\X_{s,t}-\big(D_2f(X_t,y)(f(X_s,x))-D_2f(X_s,y)(f(X_s,x))\big)\\
=&D_{211}f(X_s,y)(f(X_s,x),\X_{s,t})+D_{21}f(X_s,y)D_1f(X_s,y)\X_{s,t}\\
&-\frac{1}{2}D_{211}f(\xi,y)(f(X_s,y),X_{s,t},X_{s,t}).
\end{align*}
It follows that
\begin{align}\label{xzxi2}
\|\Xi_{s,t}-g_{s,t}(y)(f(X_s, x))\|_V\leq &\Big[2\|f\|_{K,3,\bbeta_n}^2(1+\|x\|_V)^{\beta_0}(1+\|y\|_V)^{\beta_1}\|\X\|_{2\alpha}\\
&+\frac{1}{2}\|f\|_{K,3,\bbeta_n}^2(1+\|x\|_V)^{\beta_0}(1+\|y\|_V)^{\beta_1}\|X\|_{\alpha}^2\Big]|t-s|^{2\alpha}.\nonumber
\end{align}

  Finally, by definition
\[
\lgl X\rgl_{s,t}=X_{s,t}\otimes X_{s,t}-2\X_{s,t},
\] 
which implies that
\[
\|\lgl X\rgl\|_{2\alpha}\leq \|X\|_{\alpha}^2+2\|\X\|_{\alpha}.
\]
Therefore, Young's integral term can be estimated as follows
\begin{align}\label{xzxi3}
\Big\|\int_s^tD_{211}f(X_r,y)f(X_r,x)&d\lgl X\rgl_r\Big\|_V\leq \sup_{z\in K}\big\|D_{211}f(z,y)f(z,x)\lgl X\rgl_{s,t}\big\|_V\\
\leq &\|f\|_{K,3,\bbeta_n}^2(1+\|x\|_V)^{\beta_0}(1+\|y\|_V)^{\beta_1}(\|X\|_{\alpha}^2+2\|\X\|_{2\alpha})\nonumber.
\end{align}
 Recall that 
\begin{align*}
\W_{s,t}(x,y)=&\int_s^t g(dr, y)(f(X_r, x))-g_{s,t}(y)(f(X_s, x))\\
=&\int_s^tY_rd\bX_r-g_{s,t}(f(X_s,x))+\frac{1}{2}\int_s^tD_{211}f(X_r,y)f(X_r,x)d\lgl X\rgl_r.
\end{align*}
Thus by combining (\ref{xzxi1}) - (\ref{xzxi3}), we have
\begin{align*}
\|\W(x,y)\|_{2\alpha}\leq C (1+\|x\|_V)^{\beta_0}(1+\|y\|_V)^{\beta_1},
\end{align*}
where the constant $C$ depends on $\alpha$, $\|f\|_{K,3,\bbeta_n}$, $\|X\|_{\alpha}$ and $\|\X\|_{2\alpha}$.

(ii) b) The next step is to estimate the spatial derivatives of $\W$. Observe that $\W$ consists of three terms: the rough integral, Young's integral, and $g_{s,t}(y)(f(X_s,x))$. Consider $g_{s,t}(y)(f(X_s,x))$ as a function of $(x,y)\in V^2$. Then, for any $(z_1, z_2)\in V^2$,
\begin{align*}
Dg_{s,t}(y)(f(X_s,x))(z_1,z_2)=&[D_2f(X_t,y)-D_2f(X_s,y)](D_2f(X_s,y)(z_1))\\
&+[D_{22}f(X_t,y)-D_{22}f(X_s,y)](f(X_s,y),z_2).
\end{align*}
For the rough integral term, we compute the derivative of its approximation. That is, for all $(z_1,z_2)\in V^2$,
\begin{align*}
D\Xi_{s,t}(z_1, &z_2)=D_{21}f(X_s, y)(D_2f(X_s, x)(z_1),X_{s,t})+D_{212}f(X_s, y)(f(X_s, x),X_{s,t},z_2)\\
&+D_{211}f(X_s, y)(D_2f(X_s,x)(z_1),\X_{s,t})D_{2112}f(X_s, y)(f(X_s,x),\X_{s,t}, z_2)\\
&+D_{21}f(X_s, y)D_{12}f(X_s, x)(\X_{s,t},z_1)+D_{212}f(X_s, y)D_1f(X_s, x)(\X_{s,t},z_2),
\end{align*}
where 
\[
D_{21}f(X_s, y)D_{12}f(X_s, x)(z_1,z_2,z_3)=D_{21}f(X_s, y)(z_1, D_{12}f(X_s, x)(z_2,z_3))
\]
and
\[
D_{212}f(X_s, y)D_1f(X_s, x)(z_1,z_2,z_3)=D_{212}f(X_s, y)(D_1f(X_s, x)(z_2),z_1,z_3).
\]
By the sewing lemma, we can show that for all $0\leq s\leq t\leq T$,
\[
\sum_{|\pi|\to 0}D\Xi_{s,t}\to \cJ_{s,t}(D\Xi),
\]
in $\cL(V^2; V)$ uniformly on compact sets in $(x,y)\in V^2$. Therefore,
\[
\cJ_{s,t}(D\Xi)=D\cJ_{s,t}(\Xi)=D\Big[\int_s^tD_{21}f(X_r, y)f(X_r, x)d\bX_r\Big].
\]
By a similar argument in (ii) a), we can show that
\begin{align*}
D\Big[\int_s^tD_{21}f(X_r, y)&f(X_r, x)d\bX_r\Big]-Dg_{s,t}(y)(f(X_s,x))
\end{align*}
is $2\alpha$-H\"{o}lder continuous in time. Moreover, the growth is of order $\beta_0\vee\beta_1$ in $x$, and $\beta_1\vee\beta_2$ in $y$. Young's integral term can be also estimated by using the sewing lemma and get the same result. Finally, by iteration, we conclude that $\W\in \cC^{2\alpha, \bbeta^*_{n-1},\bbeta^{**}_{n-1}}_2([0,T]^2\times V^2;V)$.

(iii) Notice that the linear rough integral on the right hand side of (\ref{dw}) is additive, Chen's relation follows immediately.
\end{proof}
Let $W(t,x)=f(X_t,x)$ for all $(t,x)\in [0,T]\times V$. In the next lemma, we show that a rough function controlled by $W$ is also controlled by $X$.
\begin{lemma}\label{evcrp}
Suppose that $f\in \cC^{3,\bbeta_1}_{loc}(V^2;V)$ and $X\in \cC^{\alpha}([0,T]; V)$. Let $W(t,x)=f(X_t,x)$, and let $(Y,\dY)\in \sE^{2\alpha}_W$ in the sense of Definition \ref{crp}. Then $(Y,Y')\in \sD_W^{2\alpha}(V)$ in the sense of Definition \ref{lcrp} for some $Y'\in \cC^{\alpha}(V;\cL(V;V))$.
\end{lemma}
\begin{proof}
Let $R^Y:[0,T]^2\to V$ be given by 
\[
R^Y_{s,t}=Y_{s,t}-W_{s,t}(\dY_s)=Y_{s,t}-[f(X_t,\dY_s)-f(X_s,\dY_s)]
\]
for all $0\leq s\leq t\leq T$. Then, $R^Y\in \cC^{2\alpha}([0,T];V)$. By Taylor's theorem, there exists $\xi$ between $X_s$ and $X_t$ such that
\begin{align*}
Y_{s,t}=D_1f(X_s,\dY_s)X_{s,t}+\frac{1}{2}D_{11}f(\xi,\dY_s)X_{s,t}\otimes X_{s,t}+R^{Y}_{s,t}.
\end{align*}
Let $Y':[0,T]\to \cL(V;V)$ be given by $ Y'_t:=D_1f(X_t, \dY_t)$ for any $t\in [0,T]$. Then it follows that
\begin{align}\label{dxdw1}
Y_{s,t}-Y'_sX_{s,t}=\frac{1}{2}D_1f(\xi,\dY_s)X_{s,t}\otimes X_{s,t}+R^{Y}_{s,t}.
\end{align}
On the other hand, the mean value theorem implies that
\begin{align}\label{dxdw2}
Y'_{s,t}=&D_1f(X_t,\dY_t)-D_1f(X_s,\dY_s)\nonumber\\
=&D_{11}f(\xi^1,\dY_s)X_{s,t}+D_{12}f(X_s,\xi^2)\dY_{s,t}.
\end{align}
for some $(\xi^1,\xi^2)$ between $(X_s, \dY_s)$ and $(X_t,\dY_t)$. Similarly as in Proposition \ref{bwbx}, let $K$ be the closed convex hull of $\{X_t,0\leq t\leq T\}$. The equalities \eqref{dxdw1} and \eqref{dxdw2} yield that
\[
\|Y'\|_{\alpha}\leq \|f\|_{K,2,\bbeta_1}\big[(1+\|\dY\|_{\infty})^{\beta_1}\|\dY\|_{\alpha}+(1+\|\dY\|_{\infty})^{\beta_0}\|X\|_{\alpha}\big]
\]
and
\[
\|\widetilde{R}^{Y}\|_{2\alpha}\leq \frac{1}{2}\|f\|_{K,3,\bbeta_n}(1+\|\dY\|_{\infty})^{\beta_0}\|X\|_{\alpha}^2+\|R^{Y}\|_{2\alpha},
\]
where $\widetilde{R}^Y_{s,t}:=Y_{s,t}-Y'_sX_{s,t}$ for all $0\leq s\leq t\leq T$. This completes the proof.
\end{proof}
In the next theorem, we prove the equivalence of linear and nonlinear rough integrals, provided that $(W,\W)$ is given in Proposition \ref{bwbx}.
\begin{theorem}\label{eqlnri}
Suppose that $f\in \cC^{3,\bbeta_2}_{loc}(V^2;V)$ and $\bX=(X,\X)\in \C^{\alpha}([0,T]; V )$. Let $(W,\W)$ be defined in Proposition \ref{bwbx}, and let $(Y,\dY)\in \sE_{W}^{2\alpha}$. Then by Lemma \ref{evcrp}, there exits $Y'=D_1f(X,\dY)$, such that $(Y, Y')\in \sD_X^{2\alpha}(V)$. In addition, the following equality holds for all $0\leq s\leq t\leq T$,
\begin{align}\label{eqlni}
\int_s^tW(dr,Y_r)=\int_s^tD_1f(X_r,Y_r)d\bX_r+\frac{1}{2}\int_s^tD_{11}f(X_r,Y_r)d\lgl X\rgl_r,
\end{align}
where the integral on the left hand side is the nonlinear rough integral in the sense of Theorem \ref{nri}, the first integral on the right hand side is the linear rough integral in the sense of Theorem \ref{tlri}, and the last integral is Young's integral.
\end{theorem}
\begin{proof}
Let $\Xi$ and $\tXi$ be the approximation of left hand and right hand sides of (\ref{eqlnri}) respectively. That is
\begin{align}\label{xifw}
\Xi_{s,t}=&W_{s,t}(Y_s)+D_{21}f(X_s,Y_s)f(X_s,\dY_s)X_{s,t}+D_{211}f(X_s,Y_s)f(X_s,\dY_s)\X_{s,t}\nonumber\\
&+D_{21}f(X_s,Y_s)D_1f(X_s,\dY_s)\X_{s,t}+\frac{1}{2}D_{211}f(X_s,Y_s)f(X_s,\dY_s)\lgl X\rgl_{s,t}\nonumber\\
&-[D_2f(X_t,Y_s)f(X_s,\dY_s)-D_2f(X_s,Y_s)f(X_s,\dY_s)].
\end{align}
and
\begin{align*}
\tXi_{s,t}=&D_1f(X_s,Y_s)X_{s,t}+D_{11}f(X_s,Y_s)\X_{s,t}+D_{12}f(X_s,Y_s)D_1f(X_s,\dY_s)\X_{s,t}\nonumber\\
&+\frac{1}{2}D_{11}f(X_s,Y_s)\lgl X \rgl_{s,t},
\end{align*}
where
\[
D_{21}f(X_s,Y_s)D_1f(X_s,\dY_s)(z_1,z_2)=D_{21}f(X_s,Y_s)(D_1f(X_s,\dY_s)(z_2),z_1),
\]
and
\[
D_{12}f(X_s,\dY_s)D_1f(X_s,\dY_s)(z_1,z_2)=D_{12}f(X_s,Y_s)(z_1,D_1f(X_s,\dY_s)(z_2)).
\]
By Theorem \ref{tlri}, \ref{nri} and Proposition \ref{bwbx}, it is not hard to verify that 
\[
\|Z_{s,t}-\Xi_{s,t}\|_V+\|\tZ_{s,t}-\tXi_{s,t}\|_V=O(|t-s|^{3\alpha}),
\]
where $Z_{s,t}$ and $\tZ_{s,t}$ denotes the left and right hand side of (\ref{eqlnri}). On the other hand, note that by definition $\lgl X\rgl_{s,t}=X_{s,t}\otimes X_{s,t}-2\X_{s,t}$. Thus by Taylor's theorem, we can show that
\begin{align*}
\Xi_{s,t}=&D_1f(X_s,Y_s)X_{s,t}+\frac{1}{2}D_{11}f(X_s,Y_s)X_{s,t}\otimes X_{s,t}\\
&+D_{21}f(X_s,Y_s)D_1f(X_s,\dY_s)\X_{s,t}+O(|t-s|^{3\alpha}),
\end{align*}
and
\begin{align*}
\tXi_{s,t}=&D_1f(X_s,Y_s)X_{s,t}+\frac{1}{2}D_{11}f(X_s,Y_s)X_{s,t}\otimes X_{s,t}\\
&+D_{12}f(X_s,\dY_s)D_1f(X_s,\dY_s)\X_{s,t}+O(|t-s|^{3\alpha}).
\end{align*}
This yields that $Z_{s,t}=\tZ_{s,t}$ for all $0\leq s\leq t\leq T$.
\end{proof}

\subsection{Nonlinear rough paths as \texorpdfstring{$\cC^{{\bbeta}_n}(V;V)$-}\ valued (linear) rough paths}\label{s.5.2}
Let $V$ be a separable Banach space. In this section, we consider a nonlinear rough path defined in Section \ref{s.3} as a $\cC^{{\bbeta}_n}(V;V)$-valued rough path. Then, we reintroduce the controlled rough paths and nonlinear rough integral in the new sense. Finally, these two approaches to the nonlinear rough paths are proved to be equivalent.

We start this section by defining the space $\cC^{{\bbeta}_n}(V;V)$:
\begin{definition}
Let $\bbeta_n=(\beta_0,\dots,\beta_n)$ be a multi-index, where $\beta_k\geq 0$ for all $k\in\{0,1,\dots,n\}$. The space $\cC^{\bbeta_n}(V;V)$ is the collection of continuously differentiable functions on $V$ with values in $V$, equipped with the norm:
\[
\|\phi\|_{\bbeta_n}=\sum_{k=0}^n\sup_{x\in V}\frac{\|D^k\phi(x)\|_{\kL_k(V;V)}}{(1+\|x\|_V)^{\beta_k}}<\infty.
\]
Then $(\cC^{\bbeta_n}(V;V), \|\cdot\|_{\bbeta_n})$ is a separable Banach space.
\end{definition}

In the follow lemma, we show the equivalence of the space $\cC^{\alpha}([0,T]; \cC^{\bbeta_n}(V;V))$ and $\cC^{\alpha,\bbeta_n}([0,T]\times V;V)$ defined in Definition \ref{dnrf}.
\begin{lemma}\label{erfs}
		\begin{enumerate}[(i)]
			\item Let $\Phi\in \cC^{\alpha,\bbeta_n}([0,T]\times V;V)$ defined by (\ref{ccv}) with $\Phi_0\in \cC^{\bbeta_n}(V;V)$. Then, $\Phi\in \cC^{\alpha}([0,T];\cC^{\bbeta_n}(V;V))$.
			
			\item Conversely, if $\Phi\in \cC^{\alpha}([0,T];\cC^{\bbeta_n}(V;V))$, then $\Phi\in \cC^{\alpha,\bbeta_n}([0,T]\times V;V)$.
		\end{enumerate}
	\end{lemma}
	\begin{proof}
(i) Fixed $t\in[0,T]$, we can show that
\begin{align*}
\|\Phi_t\|_{\bbeta_n}\leq \|\Phi_0\|_{\bbeta_n}+\|\Phi_{0,t}\|_{\alpha,\bbeta_n}\leq \|\Phi_0\|_{\bbeta_n}+T^{\alpha}\|\Phi\|_{\alpha,\bbeta_n}<\infty.
\end{align*}
Similarly for any $0\leq s\leq t\leq T$, we have
\begin{align*}
\|\Phi_{s,t}\|_{\bbeta_n}\leq \|\Phi\|_{\alpha,\bbeta_n}|t-s|^{\alpha}.
\end{align*}
It follows that as a $\cC^{\bbeta_n}(V;V)$-valued function, $\|\Phi\|_{\alpha}\leq \|\Phi\|_{\alpha,\bbeta_n}<\infty$.

(ii) We estimate $\|\Phi\|_{\alpha,\bbeta_n}$ as follows:
\begin{align*}
\|\Phi\|_{\alpha,\bbeta_n}=\sum_{k=0}^n\sup_{\substack{s\neq t\in [0,T]\\x\in V}}\frac{\|D^k\Phi_{s,t}(x)\|_V}{|t-s|^{\alpha}(1+\|x\|_V)^{\beta_k}}= \sup_{s\neq t\in [0,T]}\frac{\|\Phi_{s,t}\|_{\bbeta_n}}{|t-s|^{\alpha}}\leq \|\Phi\|_{\alpha}.
\end{align*}
As a consequence, $\Phi\in \cC^{\alpha\bbeta_n}([0,T]\times V;V)$.
\end{proof}

Let $n\geq 1$, and let $(W,\cW)\in \C^{\alpha}([0,T]; \cC^{\bbeta_n}(V;V))$ be a $\cC^{\bbeta_n}(V;V)$ linear rough path in the sense of Definition \ref{dlrp}. We define $\W:[0,T]^2\times V^2\to V$ as follows:
\begin{align}\label{ddw}
\W_{s,t}(x,y):=\cD^{(2)}\cW_{s,t}(x,y),
\end{align}
 where $\cD^{(2)}:\cC^{\bbeta_n}(V;V)^{\otimes 2}\to \cC^{\bbeta_{n-1}^*,\bbeta_{n-1}^{**}}(V\times V;V)$ where $\bbeta_{n-1}^*$ and $\bbeta_{n-1}^{**}$ are defined in \eqref{beta12}, is given by
\[
\cD^{(2)}(\phi^1, \phi^2)(x,y):=D\phi^2(y)(\phi^1(x)),
\]
for all $(\phi^1,\phi^2)\in \cC^{\bbeta_n}(V;V)^2$ and $(x,y)\in V^2$. In the next proposition, we show that $(W,\W)\in \C^{\alpha,\bbeta_n}([0,T]\times V;V)$.

\begin{proposition}\label{eqb}
 	Let $\bW=(W,\cW)\in \C^{\alpha}([0,T];\cC^{\bbeta_n}(V;V))$, and let $\W:[0,T]^2\times V^2\to V$ by given by (\ref{ddw}). Then $(W,\W)\in \C^{\alpha,\bbeta_n}([0,T]\times V;V)$.
 \end{proposition}
\begin{proof}
According to Lemma \ref{erfs}, we know that $W\in \cC^{\alpha,\bbeta_n}([0,T]\times V;V)$ and thus $\W\in \cC^{\alpha,\bbeta^*_{n-1}, \bbeta^{**}_{n-1}}_2([0,T]^2\times V^2;V)$. It suffices to verify Chen's relation (\ref{chen}). Recall that $(W,\cW)\in \C^{\alpha}([0,T];\cC^{\bbeta_n}(V;V))$ satisfies Chen's relation (\ref{lchen}). It follows that
	\begin{align*}
	\W_{s,t}(x,y)-\W_{s,u}(x,y)&-\W_{u,t}(x,y)=\cD^{(2)}(\cW_{s,t}-\cW_{s,u}-\cW_{u,t})(x,y)\\
	=&\cD^{(2)}(W_{s,u}\otimes W_{u,t})(x,y)=DW_{u,t}(y)(W_{s,u}(x)).
	\end{align*}
As a consequence, $(W,\W)\in \C^{\alpha,\bbeta_n}([0,T]\times V;V)$.
\end{proof}
	
Let $\bW=(W,\cW)\in \C^{\alpha}([0,T];\cC^{\bbeta_n}(V;V))$. In the theory of linear rough paths, under the assumption that $Y\in \sD_W^{2\alpha}(\cL(\cC^{\bbeta_n}(V;V);V))$, the rough integral of $Y$ against $W$ is well-defined. The nonlinear rough integral defined in Section \ref{s.3} can be also interpreted as this type of linear rough integral. In this case, the controlled rough path $Y$ belongs to a proper subset of $\sD_W^{2\alpha}(\cL(\cC^{\bbeta_n}(V;V);V))$, that is equivalent to $\sE_W^{2\alpha}$ in the sense of Definition \ref{crpd}. To describe this subset, we introduce the following special class of operators in $\cL(\cC^{\bbeta_n}(V;V);V)$. For any $x\in V$, let $\hx:\cC^{\bbeta_n}(V;V)\to V$ given by 
	\begin{align}\label{hx}
	\hx(\phi):=\phi(x).
	\end{align}
Then $\hx\in \cL(\cC^{\bbeta_n}(V;V); V)$ with operator norm bounded by $(1+\|x\|_V)^{\beta_0}$. Let $n\geq 1$ and let $W\in \cC^{\alpha}([0,T];\cC^{\bbeta_n}(V;V))$. We introduce the following space of basic controlled rough paths of a $\cC^{\bbeta_n}(V;V)$-valued rough path.
	\begin{definition}\label{lbcrp}
 A pair of functions $(\cY,\cY')\in\sD_W^{2\alpha}(\cL(\cC^{\bbeta_n}(V;V);V))$ is called a basic controlled rough path of $W$, if there exists a pair of functions $(Y, \dY)\in \cC^{\alpha}(V;V)\times \cC^{\alpha}(V;V)$, such that for all $t\in[0,T]$, $\cY_t=\hY_t$ and 
\begin{align}\label{gnd}
\cY_t'(\phi_1,\phi_2)=\hY_t'(\phi_1,\phi_2):=D\phi_2(Y_s)(\phi_1(\dY_t)).
\end{align}
We write $\tsE_W^{2\alpha}$ for the collection of such pairs.
	\end{definition}
The next proposition provides the equivalence of the space $\tsE_W^{2\alpha}$ and $\sE_W^{2\alpha}$.
\begin{proposition}\label{eqvcrp}
Let $n\geq 1$ and $W\in \cC^{\alpha}([0,T];\cC^{\bbeta_n}(V;V))$. Then by Lemma \ref{erfs}, $W\in \cC^{\alpha,\bbeta_n}([0,T]\times V;V)$ as well. In addition, the following properties hold:
\begin{enumerate}[(i)]
\item Let $(Y, \dY)\in \sE_W^{2\alpha}$ in the sense of Definition \ref{crp}. Then, $(\hY,\hY')\in \tsE_W^{2\alpha}$ in the sense of Definition \ref{lbcrp}, where $\hY_t$ and $\hY'_t$ are given by \eqref{hx} and (\ref{gnd}) respectively for all $t\in [0,T]$.

\item Conversely, let $(\hY,\hY')\in \tsE_W^{2\alpha}$ with associated pair $(Y,\dY)\in \cC^{\alpha}(V;V)^2$. Then $(Y,\dY)\in \sE_W^{2\alpha}$.
	\end{enumerate}
\end{proposition}
\begin{proof}
(i) By assumption $Y\in \cC^{\alpha}([0,T];V)$. It follows that
	\begin{align*}
\|\hY\|_{\alpha}=&\sup_{s\neq t\in [0,T]}\frac{\|\hY_{s,t}\|_{\cL(\cC^{\bbeta_n}(V;V); V)}}{|t-s|^{\alpha}}=\sup_{s\neq t\in [0,T]}\sup_{0\neq \phi\in \cC^{\bbeta_n}(V;V)}\frac{\|\phi(Y_t)-\phi(Y_s)\|_V}{|t-s|^{\alpha}\|\phi\|_{\bbeta_n}}\\
\leq & (1+\|Y\|_{\infty})^{\beta_1}\|Y\|_{\alpha}.
	\end{align*}
This implies that $\hY\in \cC^{\alpha}([0,T]; \cL(\cC^{\bbeta_n}(V;V); V))$. Similarly, since $\dY\in \cC^{\alpha}([0,T];V)$, we can deduce the following inequality:
	\begin{align*}
\|\hY'\|_{\alpha}\leq (1+\|Y\|_{\infty})^{\beta_2}(1+\|\dY\|_{\infty})^{\beta_0}\|Y\|_{\alpha}+(1+\|Y\|_{\infty})^{\beta_1}(1+\|\dY\|_{\infty})^{\beta_1}\|\dY\|_{\alpha}.
	\end{align*}
It suffices to estimate the reminder term. For any $\phi\in \cC^{\bbeta_n}(V;V)$, the remainder $R^{\hY}_{s,t}(\phi)$ is given by
\[
R^{\hY}_{s,t}(\phi)=\phi(Y_t)-\phi(Y_s)-D\phi(Y_s)W_{s,t}(\dY_s).
\]
Due to the fact that $(Y, \dY)\in \sE_W^{2\alpha}$, we have
\begin{align*}
\|R^{\hY}_{s,t}(\phi)\|_{\alpha}\leq \|\phi\|_{\bbeta_n}\big[(1+\|Y\|_{\infty})^{\beta_2}\|Y\|_{\alpha}^2+(1+\|Y\|_{\infty})^{\beta_1}\|R^Y\|_{2\alpha}\big]|t-s|^{2\alpha}.
\end{align*}
This implies $R^{\hY}\in \cC^{2\alpha}_2([0,T];\cL(\cC^{\bbeta_n}(V;V);V))$. As a consequence, we conclude that $(\hY, \hY')\in \tsE_W^{2\alpha}$.

(ii) To prove the converse result, it suffice the show that $R^Y\in \cC^{2\alpha}_2([0,T];V)$, where
\[
R^Y_{s,t}:=Y_{s,t}-W_{s,t}(\dY_s).
\]
Let $K$ be the closed convex hull of the set $\{Y_t, t\in[0,T]\}$, and let $\widetilde{K}$ is a compact set in $V$ whose interior contains $K$. Choose a function $\phi: V\to V$ that is infinitely differentiable and satisfies the following properties:
\begin{enumerate}[a)]
\item $\phi(x)=x$ for all $x\in K$, that implies $D\phi(x)=I$ and $D^2\phi(x)=0$ for all $x\in K$, where $I$ denotes the identity operator in $\cL(V;V)$.
\item $\phi(x)\equiv x_0\in V$ for all $x\notin \widetilde{K}$.
\item $\phi$ itself and all the derivatives of $\phi$ are bounded.
\end{enumerate}
Then, it is easy to check that $\phi\in \cC^{\bbeta_n}(V;V)$ for any multi-index $\bbeta_n$. In addition, we can show that
\begin{align*}
\|R^Y_{s,t}\|_V=\|\phi(Y_t)-\phi(Y_s)-D\phi (Y_s)[W_{s,t}(\phi_s)]\|_V=\|R^{\hY}_{s,t}(\phi)\|_V\leq \|R^{\hY}\|_{2\alpha}\|\phi\|_{\beta_n}|t-s|^{2\alpha}.
\end{align*}
In other words, $R^Y\in \cC^{2\alpha}([0,T]; V)$, and thus $(Y,\dY)\in \sE_W^{2\alpha}$.
\end{proof}

In the next theorem, we will show the equivalence of two rough integrals. 
 \begin{theorem}
 	Let $\bW=(W,\cW)\in \C^{\alpha}([0,T]; \cC^{\bbeta_2}(V;V))$. Due to Proposition \ref{eqb}, we can construct $(W,\W)\in \C^{\alpha,\bbeta_2}([0,T]\times V;V)$. Assume that $(\hY,\hY')\in \tsE_W^{2\alpha}$ with associated pair $(Y,\dY)\in \sE_W^{2\alpha}$ by Proposition \ref{eqvcrp}. Then, the following two rough integrals coincide, 
 		\begin{align}\label{equ}
 		\int_s^tW(dr,Y_r)=\int_s^t \hY_rd\bW_r,
 		\end{align}
 		where the integral on the left hand side is in the sense of \eqref{dnri}, and the integral on the right side is in the sense of Theorem \ref{tlri}.
 \end{theorem}
\begin{proof}
Let $\Xi_{s,t}$, $\tXi_{s,t}$ be the approximation of the integral on the left and right hand side respectively. That is,
\begin{align*}
\Xi_{s,t}=W_{s,t}(Y_s)+\W_{s,t}(\dY_s,Y_s),\ \mathrm{and}\ \tXi_{s,t}=\hY_sW_{s,t}+\hY_s'\cW_{s,t}.
\end{align*}
By definition of $\W$ and $(\hY,\hY')$, we have
\[
\hY_sW_{s,t}+\hY_s'\cW_{s,t}=W_{s,t}(Y_s)-\cD^{(2)}\cW_{s,t}(\dY_s, Y_s)=W_{s,t}(Y_s)+\W_{s,t}(\dY_s,Y_s).
\]
This implies the equality (\ref{equ}).
\end{proof}

At the end of this section, we provide an alternative approach to study the nonlinear RDE introduced in Section \ref{s.4}. Let $\bW=(W, \cW)\in \C^{\alpha}([0,T];\cC^{\bbeta_3}(V;V))$. Then, the RDE (\ref{rde}) can be also understood as the following equation:
\begin{align}\label{rdealt}
Y_t=\int_0^t\delta(Y_r)d\bW_r,
\end{align}
where $\delta$ denotes the Dirac delta operator, that is $\delta: V\to \cL(\cC^{\bbeta_3}(V;V);V)$ be given by $\delta(x)=\hx$. A function $Y\in \cC^{\alpha}([0,T]; V)$ is said to be a solution to (\ref{rdealt}), if $(Y,\delta(Y))\in \sD_W^{2\alpha}(V)$ and the equality holds. On the other hand, suppose that $Y$ is a solution to (\ref{rdealt}). Then, $(\hY,\hY')\in \tsE_W^{2\alpha}$ with associated pair $(Y,Y)\in \sE_W^{2\alpha}$. Therefore, $Y$ is a solution to the equation (\ref{rde}) in the sense of Definition \ref{dsol}.

On the other hand, notice that as an $\cL(\cC^{\bbeta_3}(V;V);V)$-valued operator, $\delta$ is third times differentiable. More precisely, the derivatives of $\delta$ can be written as follows $D^k\delta(x)(\phi)=D^k\phi(x)$ for $k=1,2,3$. Thus $\|D^{k}\delta(x)\|\leq (1+\|x\|_V)^{\beta_k}$ for all $k=0,1,2,3$. Then Theorem \ref{globn} is a simple variant of Theorem 4.1 of Besal\'{u} and Nualart \cite{sd-11-besalu-nualart}. For other conditions for global existence, we refer the reader to the papers of Lejay \cite{ejp-09-lejay,sdp-12-lejay}.

\subsection{An It\^{o} type formula for controlled rough paths}\label{s.5.3}
In this section, we follow the idea of Section \ref{s.5.2} to consider the nonlinear rough path as a $\cC^{\bbeta_n}(V;V)$-valued rough path. Then, we aim to generalize an It\^{o} type formula (see (3.12) in Hu and L\^{e} \cite{tams-17-hu-le} for the nonlinear Young's case).
\begin{theorem}\label{tito}
	Let $\bW=(W,\cW)\in \cC^{\alpha}([0,T]; \cC^{\bbeta_3}(V;V))$. Assume that $(Y,Y')\in \sD_W^{2\alpha}(V)$ and $(Z,Z')\in\sD_W^{2\alpha}(\cL(V;K))$. Then, the following It\^{o} type formula holds:
	\begin{align}\label{nito}
	\int_s^t Z_r &dW(r, Y_r)=\int_s^t Z_rW(dr, Y_r)+\int_s^tZ_r DW(r,Y_r)dY_r\nonumber\\
	&+\frac{1}{2}\Big[\int_s^t Z_r D^2W(r,Y_r)d\lgl Y\rgl_r +\int_s^r Z_r d \llg X, Y\rrg_r+\int_s^t Z_rd \llg Y,X\rrg_r\Big],
	\end{align}
	where 
	\begin{align}\label{dwnri}
	X_t:=\int_0^tDW(dr, Y_r)=\lim_{|\pi|\to 0} \big[DW_{t_{k-1},t_k}(Y_{t_{k-1}})+(\cD^{(2)})^2Y_{t_{k-1}}'\cW_{t_{k-1},t_k}\big],
	\end{align}
	\[
	(\cD^{(2)})^2Y_t'(\phi_1,\phi_2):=D^2\phi_2(Y_t'\phi_1)\in \cL(V;V),
	\]
	$\lgl Y\rgl$, $\llg X, Y\rrg$ and $\llg Y,X \rrg$ are $2\alpha$-continuous functions defined in Definition \ref{qdcp} and Remark \ref{rqdcp}, the integrals on the first line are rough integrals in the sense of Proposition \ref{lritcrp} (ii), and the integrals on the second line are Young's integral.
	\end{theorem}
The formula (\ref{nito}) provides the differential of $W(t,Y_t)$. Comparing with the classic It\^{o} lemma, the function $W$ is only $\alpha$-H\"{o}lder in the time argument. In this case, the assumption that $Y$ is controlled by $W$ makes sure that $W(dt, Y_t)$ is well-defined as the differential of a controlled rough path of $W$.

In order to prove Theorem \ref{tito}, we should make each integral in (\ref{nito}) to be well-defined. The first lemma below shows that $F_t=W(t, Y_t)$ is controlled by $W$.
\begin{lemma}\label{lito1}
	Let $W\in \cC^{\alpha}([0,T];\cC^{\bbeta_2}(V;V))$, and let $(Y,Y')\in \sD_W^{2\alpha}(V)$. Denote by $F_t=W(t,Y_t)$. Then, $F\in \sD_W^{2\alpha}(V)$.
	\end{lemma}
\begin{proof}
	By Taylor's theorem and the fact that $(Y,Y')\in \sD_W^{2\alpha}(V)$, there exist $\xi=cY_s+(1-c)Y_t$ and $\txi=\tc Y_s+(1-\tc)Y_t$ for some $c,\tc\in[0,1]$, such that
	\begin{align*}
	F_{s,t}=&F_t-F_s=W_{s,t}(Y_s)+[W_{s,t}(Y_t)-W_{s,t}(Y_s)]+W_s(Y_t)-W_s(Y_s)\\
	=&\hY_sW_{s,t}+DW_{s,t}(\xi)Y_{s,t}+DW_s(Y_s)[Y'_sW_{s,t}+R^Y_{s,t}]+\frac{1}{2}D^2W_s(\txi)Y_{s,t}^{\otimes 2}.
	\end{align*}
	This yields that $(F, F')\in \sD_W^{2\alpha}(V)$, where $F':=\hY+DW(Y)Y'\in \cL(\cC^{\bbeta_2};V)$.
	\end{proof}
Suppose that $\bW=(W,\cW)\in \C^{\alpha}([0,T]; \cC^{\bbeta_3}(V;V))$. As a consequence of Proposition \ref{lritcrp} (ii) and Lemma \ref{lito1}, the integral $\int_s^t Z_rd W(r,Y_r)=\int_s^t Z_rd F_r$ is well-defined as the integral of two controlled rough paths. Additionally, by Taylor's theorem, we can approximate this integral in the following way:
\begin{align}\label{itotp1}
\int_s^t Z_r&d W(r,Y_r)=Z_sF_{s,t}+Z_s'F_s'\cW_{s,t}+O(|t-s|^{3\alpha})\nonumber\\
=&Z_sW_{s,t}(Y_s)+Z_sDW_{s,t}(Y_s)Y_{s,t}+Z_sDW_s(Y_s)Y_{s,t}+\frac{1}{2}Z_sD^2W_s(Y_s)(Y_{s,t},Y_{s,t})\nonumber\\
&+Z_s'\hY_s\cW_{s,t}+Z_s'DW(s,Y_s)Y_s'\cW_{s,t}+O(|t-s|^{3\alpha}),
\end{align}
where
\[
Z_s'\hY_s(\phi_1,\phi_2)=Z_s'(\phi_1)[\phi_2(Y_s)],
\]
and
\[
Z_s'DW(s,Y_s)Y_s'(\phi_1,\phi_2)=Z_s'(\phi_1)[DW(s,Y_s)(Y_s'(\phi_2))],
\]
for all $(\phi_1,\phi_2)\in \cC^{\bbeta_3}(V;V)^2$.

The next lemma provides a generalized version of Theorem \ref{nri} and Proposition \ref{ricrp}. The proof is similar, we omit it here.
\begin{lemma}
Let $(W,\cW)\in \cC^{\alpha}([0,T];\cC^{\bbeta_2}(V;V))$, and let $(Y,Y')\in \sD_W^{2\alpha}(V)$. Then, the following limit exists and defines a version of integral:
\begin{align*}
\int_s^tW(dr, Y_r):=\lim_{|\pi|\to 0}\sum_{k=1}^n\Big[W_{t_{k-1},t_k}(Y_s)+Y_{t_{k-1}}'\hY_{t_{k-1}}\cD^{(2)}\cW_{t_{k-1},t_k}\Big],
\end{align*}
where $Y_t'\hY_t\cD^{(2)}(\phi_1,\phi_2):=D\phi_2(Y_t)[Y_t'(\phi_1)]$ for any $(\phi_1,\phi_2)\in\cC^{\bbeta_2}(V;V)$. Moreover, $(G, Y)\in \sE_W^{2\alpha}$, where $G_t:=\int_0^t W(dr, Y_r)$ for all $t\in[0,T]$.
\end{lemma}
Therefore, the integral $\int_s^t Z_rW(dr, Y_r)=\int_s^tZ_rdG_r$ can be approximated as follows:
\begin{align}\label{itotp2}
\int_s^tZ_rW(dr,Y_r)=&Z_sG_{s,t}+Z_s'\hY_s\cD^{(2)}\cW_{s,t}+O(|t-s|^{3\alpha})\nonumber\\
=&Z_sW_{s,t}(Y_s)+Z_sY_s'\hY_s\cD^{(2)}\cW_{s,t}+Z_s'\hY_s\cW_{s,t}+O(|t-s|^{3\alpha}).
\end{align}

Assume that $(W,\cW)\in \C^{\alpha}([0,T];\cC^{\bbeta_3}(V;V))$. Let $H_t=Z_tDW(t, Y_t)\in \cL(V;V)$ for all $t\in [0,T]$. By a similar argument as Lemma \ref{lito1}, we can show that
\begin{align*}
H_{s,t}=Z_{s,t}DW(s,Y_s)+Z_s\hY_sDW_{s,t}+Z_sD^2W(s,Y_s)Y_s'W_{s,t}+O(|t-s|^{2\alpha}).
\end{align*}
It follows that
\begin{align}\label{itotp3}
\int_s^tZ_rDW(r,Y_r)dY_r=&Z_sDW(s,Y_s)Y_{s,t}+Z_s'DW(s,Y_s)Y_s'\cW_{s,t}\\
&+Z_s\hY_s\cD^{(1)}Y_s'\cW_{s,t}+Z_sD^2W(s,Y_s)Y_s'Y_s'\cW_{s,t}+O(|t-s|^{3\alpha}),\nonumber
\end{align}
where
\[
Z_s'DW(s,Y_s)Y_s'(\phi_1,\phi_2)=Z_s'(\phi_1)[DW(s,Y_s)(Y_s'(\phi_2))],
\]
\[
Z_s\hY_s\cD^{(1)}Y_s'(\phi_1,\phi_2)=Z_s[D\phi_1(Y_s)Y_s'(\phi_2)],
\]
and
\[
Z_sD^2W(s,Y_s)Y_s'Y_s'(\phi_1,\phi_2)=Z_s[D^2W(s,Y_s)(Y_s'(\phi_1),Y_s'(\phi_2))],
\]
for all $(\phi_1,\phi_2)\in \cC^{\bbeta_3}(V;V)$.

By a similar argument in Theorem \ref{nri} and the sewing lemma, we can show that the limit in \ref{dwnri} uniquely exists. It allows us to define $X_t$ to be the limit. In addition, we can verify that $X\in \sD_W^{2\alpha}(\cL(V;V))$. Thus three quadratic compensator terms on the second line of (\ref{nito}) are all well-defined, and according to Remark \ref{rqdcp} (iii), $\lgl Y \rgl\in \cC^{2\alpha}_2([0,T];V\otimes V)$ and $\llg X,Y\rrg,\llg Y,X\rrg\in \cC^{2\alpha}_2([0,T];V)$. Therefore, the integrals on the second line of (\ref{nito}) can be interpreted as Young's integrals. We can approximate them as follows:
\begin{align}\label{itotp4}
\int_s^tZ_r&D^2W(r,Y_r)\lgl Y\rgl_r=Z_sD^2W(s,Y_s)\lgl Y \rgl_{s,t}+O(|t-s|^{3\alpha})\nonumber\\
=&Z_sD^2W(s,Y_s)\big[Y_{s,t}\otimes Y_{s,t}-2Y_s'Y_s'\cW_{s,t}\big]+O(|t-s|^{3\alpha}),
\end{align}
\begin{align}
\int_s^tZ_r&\llg X,Y\rrg_r=Z_sDW_{s,t}(Y_s)Y_{s,t}-2Z_s\hY_s\cD^{(1)}Y_s'\cW_{s,t}+O(|t-s|^{3\alpha}),
\end{align}
and
\begin{align}\label{itotp5}
\int_s^tZ_r&\llg Y,X\rrg_r=Z_sDW_{s,t}(Y_s)Y_{s,t}-2Z_sY_s'\hY_s\cD^{(2)}\cW_{s,t}+O(|t-s|^{3\alpha}).
\end{align}
As we approximated all the integrals in (\ref{nito}), the proof of Theorem \ref{tito} is straightforward.
\begin{proof}[Proof of Theorem \ref{tito}]
Denote by $LHS$ and $RHS$ the left and right hand side of equation (\ref{nito}) respectively. Recall the equality (\ref{itotp1}), that is,
\begin{align*}
LHS=&Z_sW_{s,t}(Y_s)+Z_sDW_{s,t}(Y_s)Y_{s,t}+Z_sDW_s(Y_s)Y_{s,t}+\frac{1}{2}Z_sD^2W_s(Y_s)(Y_{s,t},Y_{s,t})\nonumber\\
&+Z_s'\hY_s\cW_{s,t}+Z_sDW_s(Y_s)Y_s'\cW_{s,t}+O(|t-s|^{3\alpha}).
\end{align*}
On the other hand, combining (\ref{itotp2}) - (\ref{itotp5}), we have
\begin{align*}
RHS=&Z_sW_{s,t}(Y_s)+Z_sDW_{s,t}(Y_s)Y_{s,t}+Z_sDW_s(Y_s)Y_{s,t}+\frac{1}{2}Z_sD^2W_s(Y_s)(Y_{s,t},Y_{s,t})\nonumber\\
&+Z_s'\hY_s\cW_{s,t}+Z_sDW_s(Y_s)Y_s'\cW_{s,t}+O(|t-s|^{3\alpha}),
\end{align*}
as well. Since $\alpha\in (\frac{1}{3},\frac{1}{2}]$, it follows that equality (\ref{nito}) holds for all $0\leq s\leq t\leq T$.
\end{proof}

\begin{remark}
We present another approach to Theorem \ref{tito}. Notice that $W(t,Y_t)=\hY_tW_t$. Under the assumption that $W\in \cC^{\alpha}([0,T];\cC^{\bbeta_3}(V;V))$, we can show that $Y\in \sD_W^{2\alpha}(V)$ implies $\hY\in \sD_W^{2\alpha}(\cL(\cC^{\bbeta_2};V))$. It follows from Lemma \ref{itolrp} that
\[
d (\hY_t W_t)=(d\hY_t)W_t+\hY_t d\bW_t+\frac{1}{2}[d\llg \hY, W\rrg_t+d\llg W,\hY\rrg_t],
\]
which is equivalent to (\ref{nito}). In addition, from this point of view, $W$ can be replaced by any rough function $U\in \sD_W^{2\alpha}(\cC^{\bbeta_2}(V;V))$.
\end{remark}

\section{Rough partial differential equations}\label{s.6}
In this section, we apply the theory of nonlinear rough paths to a class of partial differential equations in a H\"{o}lder media. 
\subsection{RDEs with spatial parameters}\label{s.6.1}
Let $(W,\cW)\in \C^{\alpha}([0,T];\cC^{\bbeta_3}(\R^d;\R^d))$, and let $\W$ be given by \eqref{ddw}. Assume Hypothesis \ref{hyp1}. Then, due to Theorem \ref{globn}, for any fixed $x\in\R^d$, the following equation
\begin{align}\label{rdef}
Y_t(x)=x+\int_0^t W(dr, Y_r(x)),
\end{align}
has a unique solution $Y(x)$ on $[0,T]$. In this section, by studying the gradient in $x$ of $Y_t(x)$, we will show that $Y_t(x)$ is invertible in $x$, and the inverse is controlled by $W$ as well.

 In the next theorem, we follow the idea of Hu and L\^{e} \cite{tams-17-hu-le} to show that $Y_t(x)$ is differentiable in $x$.
\begin{theorem}\label{trdef}
Let $(W,\cW)\in \C^{\alpha}([0,T];\cC^{\bbeta_3}(\R^d;\R^d))$. Assume Hypothesis \ref{hyp1}. Let $Y=\{Y_t(x),t\in[0,T], x\in \R^d\}$ be the unique solution to (\ref{rdef}). Then for any $t\in [0,T]$, $Y_t$ is differentiable, and the gradient $DY_t$ satisfies the following equation:
\begin{align}\label{rdef1}
D Y_t(x)= I+\int_0^tdF_r(x) D Y_r(x),
\end{align}
where $I$ denotes the $d\times d$ identity matrix and $F(x)$ is a $d\times d$ matrix-valued function given by
\[
F_t(x):=\int_0^t D W(dr, Y_r(x))
\]
that is defined in the sense of (\ref{dwnri}). Moreover, for every $t\in [0,T]$ and $x\in \R^d$, $D Y_t(x)$ is invertible, and its inverse $(D Y_t(x))^{-1}=:M_t(x)$ satisfies the following equation:
\begin{align}\label{ivrdef}
M_t(x)=I-\int_0^tM_r(x)dF_r(x)+\int_0^t[M_r(x)]^L d\lgl F(x) \rgl_r.
\end{align}
where $\lgl F(x)\rgl_r$ is the quadratic compensator of $F(x)$, which is an $(\R^d\otimes \R^d)^{\otimes 2}$-valued $2\alpha$-H\"{o}lder continuous function on $[0,t]$, and $[M_r(x)]^L:(\R^d\otimes \R^d)^{\otimes 2} \to \R^d\otimes \R^d$ is given by 
\[
[M_r(x)]^L(A\otimes B)=M_r(x)\cdot A\cdot B,
\]
 for any $d\times d$ matrices $A$ and $B$.
\end{theorem}
\begin{proof}
Fix $x\in \R^d$. Let $e$ be a unit vector in $\R^d$. For any $h\in(0,1)$, we write
\[
\eta^h_t:=\frac{1}h[Y_t(x+he)-Y_t(x)].
\]
We claim that as $h\downarrow 0$, $\eta^h_t$ converges to the solution to the following equation
\begin{align}\label{eeta}
\eta_t=e+\int_0^tdF_r(x)\eta_r=e+\int_0^t D W(dr, Y_r(x))\eta_r.
\end{align}
Firstly, we show that \eqref{eeta} has a unique solution. Notice that $F(x)$ is defined as a nonlinear rough integral. Then, by Proposition \ref{ricrp}, $F(x)$ is controlled by $D W$ and thus by $W$. That is, 
\[
F_{s,t}(x)=D W_{s,t}(Y_s(x))+O(|t-s|^{2\alpha}):=(\hY_s(x) D) W_{s,t}+O(|t-s|^{2\alpha}),
\]
where $\hY_{\cdot}(x) D$ is considered as an $\alpha$-H\"{o}lder continuous function on $[0,T]$ taking values in $\cL(\cC^{\bbeta_3}(\R^d;\R^d);\cL(\R^d;\R^d))$.  Here $\hY$ is defined in \eqref{hx}. We can also directly define the operator $Y_s(x)D$ by the former expression. $DW_{s,t}(Y_s(x))$ is just an approximation of the integral without the double integral term, thus the error is $O(|t-s|^{2\alpha})$.
By Proposition \ref{lritcrp} (ii), $F(x)$ can be interpreted as a linear rough path. Thus, equation (\ref{eeta}) is a linear RDE, and it follows from Theorem \ref{tlrde} that this equation has a unique solution. 

On the other hand, by Corollary \ref{corde}, $\|\eta^h\|_{\alpha}$ is uniformly bounded in $h\in (0,1)$. As a consequence of the Arzel\`{a}-Ascoli theorem, there exists a sequence $\{h_n\}_{n\geq 1}$, such that, as $n\to\infty$, $h_n\downarrow 0$, and $\eta^{h_n}_t$ converges to some function $\eta_t$ in $\cC^{\alpha'}([0,T];\R^d)$ for any $\alpha'\in(0,\alpha)$. In addition, by the sewing lemma, $\eta^{h_n}$ satisfies the following estimate
\begin{align}\label{eetah}
\eta^{h_n}_{s,t}=D W_{s,t}(Y_s(x))\eta^{h_n}_s+D \cW_{s,t}(Y_s(x),Y_s(x))(\eta^{h_n}_s, \eta^{h_n}_s)+O(|t-s|^{3\alpha})+O(h_n),
\end{align}
for all $0\leq s< t\leq T$. Let $n\to \infty$. The estimate (\ref{eetah}) implies that $\eta_t$ satisfies the RDE (\ref{eeta}). Therefore, $D Y_t(x)$ exists and is the unique solution to (\ref{rdef}).

To prove the invertibility of $D Y_t(x)$, we follow Stroock's idea (see Chapter 8 of \cite{lnm-83-stroock}). Let $M_t(x)$ be the unique solution to the linear RDE (\ref{ivrdef}). By (\ref{itoiso}) and (\ref{ito}), we can deduce the following equation:
\begin{align*}
D Y_t(x)M_t(x)=&I+\int_0^tdF_r(x)D Y_r(x)M_r(x)-\int_0^tD Y_r(x)M_r(x)dF_r(x)\\
&+\int_0^t [D Y_r(x)M_r(x)]^Ld\lgl F(x)\rgl_r-\int_0^t [D Y_r(x)M_r(x)]^Md\lgl F(x)\rgl_r,
\end{align*}
where $[D Y_r(x)M_r(x)]^M$ is considered as a linear operator on $(\R^d\otimes \R^d)^{\otimes 2}$ given by 
\[
[D Y_r(x)M_r(x)]^M(A\otimes B)=A \cdot D Y_r(x) M_r(x)\cdot B,
\]
for any $d\times d$ matrices $A$ and $B$. Notice that $D Y_t(x)M_t(x)\equiv I$ solves this equation. Thus the uniqueness of linear RDEs implies that $M_t= (D Y_t)^{-1}$.
\end{proof}
\begin{remark}
By taking further spatial derivatives on both sides of \eqref{rdef1} and \eqref{ivrdef}, we can show that $DY_t$ and $M_t$ are both twice spatial differentiable with locally bounded derivatives. On the other hand, since Theorem \ref{trdef} shows that $D Y_t(x)$ is invertible in $x$ for all $(t,x)\in [0,T]\times \R^d$, by the implicit function theorem, we deduce that for any fixed $t\in [0,T]$, $Y_t$ has an inverse $Z_t$ such that $Z_t(Y_t(x))=Y_t(Z_t(x))=x$.
\end{remark}
In the next lemma, we prove that fix $x\in \R^d$, $Z(x)$ is controlled by $W$. 
\begin{lemma}\label{zcr}
Let $Y(x)=\{Y_t(x),t\in [0,T]\}$ be the solution to the RDE (\ref{rdef}), and let $Z_t=Y^{-1}_t$ be the inverse of $Y_t$. Fix $x\in \R^d$. Then $Z(x)$ is controlled by $W$.
\end{lemma}
\begin{proof}
For any $t\in[0,T]$, since $Z_t$ is the inverse of $Y_t$, and $DY_tM_t=I$ for all $t$, we deduce that $Z_t$ is differentiable and its derivative is given by the following formula
\begin{align}\label{dzt}
DZ_t(x)=M_t(Z_t(x)).
\end{align}
Fix $(t,x)\in (0,T]\times \R^d$. Let $y=Z_t(x)$. Then $x=Y_t(y)$. Notice that a similar argument as in Theorem \ref{trdef} implies that $M_t(x)$ is differentiable in $x$ and the derivative is locally bounded. Thus by Taylor's theorem, the following equality holds for all $s\in [0,t)$
\begin{align}\label{zst1}
Z_{s,t}(x)=&Z_s(Y_s(y))-Z_s(Y_t(y))\nonumber\\
=&-DZ_s(Y_s(y))Y_{s,t}(y)+O(|t-s|^{2\alpha})\nonumber\\
=&-M_s(Z_s(x))Y_{s,t}(Z_s(x))+O(|t-s|^{2\alpha}).
\end{align}
On the other hand, by Proposition \ref{ricrp}, we have
\[
Y_{s,t}(x)=W_{s,t}(Y_s(x))+O(|t-s|^2).
\]
Combining above two inequalities, we can write
\begin{align*}
Z_{s,t}(x)=-M_s(Z_s(x))W_{s,t}(x)+O(|t-s|^{2\alpha}).
\end{align*}
Let $Z'(x)=\{Z_t'(x),t\in[0,T]\}$ where $Z_t'(x):\cC^{\bbeta_3}(\R^d;\R^d)\to \R^d$ is given by
\[
Z_t'(x)\Phi:=M_t(Z_t(x))\Phi(x).
\]
Then it is easy to check that $Z_s'(x)\in\cL(\cC^{\bbeta_3}(\R^d;\R^d);\R^d)$, and thus $(Z(x),Z'(x))\in \sD_W^{2\alpha}(\R^d)$.
\end{proof}

\begin{remark}
By taking derivative on both sides of \eqref{dzt}, we have
\[
D^2Z_t(x)=DM_t(Z_t(x))M_t(Z_t(x)).
\]
Furthermore, we can deduce a more delicate estimate than \eqref{zst1} as follows,
\begin{align}\label{zst2}
Z_{s,t}&(x)=-M_s(Z_t(x))Y_{s,t}(Z_t(x))-\frac{1}{2}DM_s(y)M_s(y)Y_{s,t}(y)^{\otimes 2}+O(|t-s|^{3\alpha})\nonumber\\
=&-M_s(Z_s(x))W_{s,t}(x)-M_s(Z_s(x))\W(x,x)+\frac{1}{2}DM_s(Z_s(x))M_s(Z_s(x))W_{s,t}(x)^{\otimes 2}\nonumber\\
&+M_s(Z_s(x))DW_{s,t}(x)W_{s,t}(x)+O(|t-s|^{3\alpha}),
\end{align}
where for all $i=1,2,\dots,d$,
\begin{align*}
&\big[DM_s(Z_s(x))M_s(Z_s(x))W_{s,t}(x)^{\otimes 2}\big]^i\\
=&\sum_{k_1\dots k_3=1}^d\frac{\partial M^{ik_2}(x)}{\partial x_{k_1}}(Z_s(x))M_s^{k_1k_3}(Z_s(x))W_{s,t}^{k_2}(x)W^{k_3}_{s,t}
\end{align*}
and
\begin{align*}
\big[DM_s(y)W_{s,t}(x)^{\otimes 2}\big]^i=\frac{\partial M^{ik_2}}{\partial k_1}(Z_s(x)) W^{k_1}_{s,t}(x)W^{k_2}_{s,t}(x).
\end{align*}
This estimate will be used in Section \ref{s.6.2} below.
\end{remark}

\subsection{Rough partial differential equations}\label{s.6.2}
Let $h\in \cC_{loc}^3(\R^d;\R^d)$ the space of functions that are locally bounded and have locally bounded first, second and third derivatives. In this section, we will show that $u=\{u(t,x)=h(Z_t(x)), (t,x)\in[0,T]\times \R^d\}$, where $Z_t(x)$ is defined in Section \ref{s.6.1}, is the unique solution to equation \eqref{rpde}.

\begin{definition}\label{soltspt}
Let $(W,\cW)\in \C^{\alpha}([0,T];\cC^{\bbeta_3}(\R^d;\R^d))$, let $\W:\R^d\times \R^d\to \R^d$ be given by \eqref{ddw}, and let $h$ be a function on $\R^d$ with values in $\R^d$. A function $u=\{u(t,x),(t,x)\in [0,T]\times \R^d\}$ is called a solution to equation \eqref{rpde} with initial condition $h$, if the following properties are satisfied:
\begin{enumerate}[(i)]
\item $u(0,x)=h(x)$ for all $x\in \R^d$.

\item $u$ is twice spatial differential everywhere, and $D u(\cdot, x)$ is controlled by $W$ for all $x\in\R^d$.
\item The following equality is true for all $(t,x)\in [0,T]\times \R^d$
\begin{align}\label{solrpde}
u(t,x)=&h(x)-\int_0^tD u(r,x)W(dr,x)+\frac{1}{2}\int_0^tD u(r,x) d\llg DW(x), W(x) \rrg_r\nonumber\\
&+\frac{1}{2}\int_0^t Du(r,x)d\llg W(x), DW(x) \rrg_r+\frac{1}{2} \int_0^t D^2 u(r,x) d\lgl W(x)\rgl_r,
\end{align}
where the first integral is defined as follows,
\[
\int_0^tD u(r,x)W(dr,x):=\int_0^tD u(r,x)dW_r(\xi)\Big|_{\xi=x},
\]
the quadratic compensators 
\[
\llg DW(x), W(x) \rrg_{s,t}:=\llg DW, W\rrg_{s,t}(\xi_1,\xi_2)\big|_{(\xi_1,\xi_2)=(x,x)},
\] 
\[
\llg W(x), DW(x) \rrg_{s,t}:=\llg W, DW\rrg_{s,t}(\xi_1,\xi_2)\big|_{(\xi_1,\xi_2)=(x,x)},
\] 
and
 \[
\lgl W(x) \rgl_{s,t}:=\lgl W\rgl_{s,t}(\xi_1,\xi_2)\big|_{(\xi_1,\xi_2)=(x,x)}
\] 
are defined by \eqref{qdcp1}, \eqref{qdcp2} and \eqref{qdcp3}, $D^2u(r,x)$ is considered as a linear operator from $\R^d\otimes \R^d\to \R$, that is 
\[
D^2u(t,x)M=\sum_{i,j=1}^d\frac{\partial^2 u(t,x)}{\partial x_i\partial x_j}M^{ij},
\]
for any $d\times d$ matrix $M=(M^{ij})_{i,j=1}^d$, and the last three integrals are in the sense of Young's integral.
\end{enumerate}
\end{definition}

In the next theorem, we will show that $h(Z_t)$, where $Z_t$ is defined in Section \ref{s.6.1}, is a solution to equation \eqref{rpde}.
\begin{theorem}\label{thmrpde}
Let $(W,\cW)\in \C^{\alpha}([0,T];\cC^{\bbeta_3}(\R^d;\R^d))$, and let $\W$ be given by \eqref{ddw}. Assume Hypothesis \ref{hyp1}. Let $Y$ be the solution to the equation (\ref{rdef}), and let $Z_t=Y^{-1}_t$ for all $t\in [0,T]$. Suppose that $h\in \cC^3_{loc}(\R^d;\R^d)$. Then, $u(t,x)=h(Z_t(x))$ is a solution to \eqref{rpde} in the sense of Definition \ref{soltspt}.
\end{theorem}
\begin{proof}
We prove this theorem by checking every property in Definition \ref{soltspt}.
 By assumption, we know that $u(0,x)=h(Z_0(x))=h(x)$. In addition, since $h\in \cC^3_{loc}(\R^d;\R^d)$ and $Z_t(x)$ is twice spatial differentiable, we can show that
 \[
 D[h(Z_t(x))]=(Dh)(Z_t(x))M_t(x)
 \]
 and
 \[
 D^2[h(Z_t(x))]=(D^2h)(Z_t(x))M_t(x)^2+(Dh)(Z_t(x))DM_t(x),
 \]
where $(Dh)(Z_t(x))DM_t(x)$ is a $d\times d$ matrix with component 
\[
[(Dh)(Z_t(x))DM_t(x)]^{ij}=\sum_{k=1}^d\frac{\partial}{\partial x_k}h(Z_t(x))\frac{\partial}{\partial x_j}M^{ki}(x).
\]
Recall that $M_t(x)$ is the solution to the linear RDE \eqref{ivrdef}. Then we can write
\[
M_t(x)=-M_s(x)F_{s,t}(x)+O(|t-s|^{2\alpha})=-M_s(x)D W_{s,t}(Y_s(x))+O(|t-s|^{2\alpha}).
\]
This implies that
\begin{align*}
M_t(Z_t(x))&-M_s(Z_s(x))=M_{s,t}(Z_s(x))+DM_s(Z_s(x))Z_{s,t}(x)+O(|t-s|^{2\alpha})\\
=&-M_s(x)D W_{s,t}(x)-DM_s(Z_s(x))M_s(Z_s(x))W_{s,t}(x)+O(|t-s|^{2\alpha})
\end{align*}
Let $M_t(Z_t(x)):\cC^{\bbeta_3}(\R^d;\R^d)\to \R^d\otimes \R^d$ be given by
\[
M_t(Z_t(x))'\Phi:=-M_t(x)D \Phi(x)-DM_t(Z_t(x))M_t(Z_t(x))\Phi(x),
\]
where
\[
[DM_t(Z_t(x))M_t(Z_t(x))\Phi(x)]^{ij}=\sum_{k_1,k_2}\frac{\partial}{\partial x_{k_1}}M_t^{k_2,i}(Z_t(x))\Phi^{j}(x)
\]
for any $(t,x)\in [0,T]\times \R^d$. We can show that 
\[
M(Z(x))'\in \cC^{\alpha}([0,T];\cL(\cC^{\bbeta_3}(\R^d;\R^d);\R^d\otimes \R^d)).
\]
Thus $(M(Z(x)),M'(Z(x)))\in \sD_W^{2\alpha}(\cL(\cC^{\bbeta_3}(\R^d;\R^d);\R^d\otimes \R^d))$. On the other hand, we know that $Z(x)$ is controlled by $W$ due to Lemma \ref{zcr}. Note that $h\in \cC^3_{loc}(\R^d;\R^d)$, by Lemma \ref{cmpcr}, we deduce that $(D(h(Z(x))),D(h(Z(x)))')\in \sD_W^{2\alpha}(\R^d\otimes \R^d)$, where the Gubinelli derivative $[Dh(Z(x))]':\cC^{\bbeta_3}(\R^d;\R^d)\to \R^d$ is given by
\begin{align*}
[D(h(Z(x)))]'\Phi=&-(D^2h)(Z_t(x))M_t(Z_t(x))\Phi(x)M_t(Z_t(x))\\
&-(Dh)(Z_t(x))M_t(Z_t(x))D \Phi(x)\\
&-(Dh)(Z_t(x))DM_t(Z_t(x))M_t(Z_t(x))\Phi(x).
\end{align*}
As a consequence the property (i) and (ii) of Definition \ref{soltspt} are satisfied.

In the next step, we will prove equality \eqref{solrpde} by a similar argument as in Theorem \ref{tito}. For any $0\leq s\leq t\leq T$, as a consequence Taylor's theorem, we can write
\begin{align}\label{hzst1}
h(Z_t(x))-h(Z_s(x))=&(Dh)(Z_s(x))Z_{s,t}(x)+\frac{1}{2}(D^2h)(Z_s(x))Z_{s,t}(x)^{\otimes 2}+O(|t-s|^{3\alpha})\nonumber\\
:=&I_1+I_2+O(|t-s|^{3\alpha}).
\end{align}
By \eqref{zst2}, we have
\begin{align}
I_1=&-(Dh)(Z_s(x))M_s(Z_s(x))W_{s,t}(x)-(Dh)(Z_s(x))M_s(Z_s(x))\W(x,x)\\
&+\frac{1}{2}(Dh)(Z_s(x))DM_s(Z_s(x))M_s(Z_s(x))W_{s,t}(x)^{\otimes 2}\nonumber\\
&+(Dh)(Z_s(x))M_s(Z_s(x))DW_{s,t}(x)W_{s,t}(x)+O(|t-s|^{3\alpha}), \nonumber
\end{align}
and
\begin{align}
I_2=&\frac{1}{2}\big[(D^2h)(Z_s(x))M_s(Z_s(x))W_{s,t}(x)\big]\cdot \big[M_s(Z_s(x)) W_{s,t}(x)\big]+O(|t-s|^{3\alpha}),
\end{align}
where
\begin{align*}
(Dh)&(Z_s(x))DM_s(Z_s(x))M_s(Z_s(x))W_{s,t}(x)^{\otimes 2}\\
=&\sum_{k_1,\dots,k_4=1}^d\frac{\partial h}{\partial x_{k_1}}(Z_s(x))\frac{M_s^{k_1k_3}}{\partial k_2}(Z_s(x))W^{k_3}_{s,t}(x)M_s^{k_2k_4}(Z_s)W^{k_4}_{s,t}(x).
\end{align*}
Due to Theorem \ref{tlri}, we can write
\begin{align}\label{hzst2}
\int_s^tDh&[Z_r(x)]W(dr,x)=(Dh)(Z_s(x))M_s(Z_s(x))W_{s,t}(x)\nonumber\\
&-[(D^2h)(Z_s(x))M_s(x)]^LM_s(Z_s(x))^R\cW_{s,t}(x,x)\nonumber\\
&-(Dh)(Z_s(x))M_s(Z_s(x)) \W_{s,t}^*(x,x)\nonumber\\
&-(Dh)(Z_s(x))DM_s(Z_s(x))M_s(Z_s(x))\cW_{s,t}(x,x),
\end{align} 
where
\begin{align*}
[(D^2h)&(Z_s(x))M_s(x)]^LM_s(Z_s(x))^R\cW_{s,t}(x,x)\\
=&\sum_{k_1,\dots,k_4=1}^d\frac{\partial^2 h}{\partial x_{k_1}\partial x_{k_2}}(Z_s(x))M_s^{k_2k_3}(Z_s(x))\cW_{s,t}^{k_3k_4}M_s^{k_1k_4}(Z_s(x))
\end{align*}
and
\begin{align*}
(Dh)&(Z_s(x))DM_s(Z_s(x))M_s(Z_s(x))\cW_{s,t}(x,x)\\
=&\sum_{k_1,\dots,k_4=1}^d\frac{\partial h}{\partial x_{k_1}}(Z_s(x))\frac{\partial M^{k_1k_3}_s}{\partial x_{k_2}}(Z_s(x))M_s^{k_2k_4}(Z_s(x))\cW^{k_3k_4}_{s,t}(x,x)
\end{align*}
Combining \eqref{hzst1} - \eqref{hzst2}, we have
\begin{align*}
h(Z_t(x))&-h(Z_s(x))+\int_s^t D[h(Z_r(x))]W(dr,x)\\
=&\frac{1}{2}(Dh)(Z_s(x))M_s(Z_s(x))\big[\llg DW(x), W(x)\rrg_{s,t}+\llg W(x), DW(x)\rrg_{s,t}\big]\\
&+\frac{1}{2}[(D^2h)(Z_s(x))M_s(x)]^LM_s(Z_s(x))^R\lgl W(x)\rgl_{s,t}\nonumber\\
&+\frac{1}{2}(Dh)(Z_s(x))DM_s(Z_s(x))M_s(Z_s(x))\lgl W(x)\rgl_{s,t}+O(|t-s|^{3\alpha}).
\end{align*}
On the other hand, by the theory of Young's integral, we can show that
\begin{align*}
&\int_0^tD [h(Z_r(x))] d\llg DW(x), W(x) \rrg_r+\int_0^t D[h(Z_r(x))]d\llg W(x), DW(x) \rrg_r\nonumber\\
&+ \int_0^t D^2 [h(Z_r(x))] d\lgl W(x)\rgl_r\\
=&(Dh)(Z_s(x))M_s(Z_s(x))\big[\llg DW(x), W(x)\rrg_{s,t}+\llg W(x), DW(x)\rrg_{s,t}\big]\\
&+[(D^2h)(Z_s(x))M_s(x)]^LM_s(Z_s(x))^R\lgl W(x)\rgl_{s,t}\nonumber\\
&+(Dh)(Z_s(x))DM_s(Z_s(x))M_s(Z_s(x))\lgl W(x)\rgl_{s,t}+O(|t-s|^{3\alpha}).
\end{align*}
It follows that (\ref{solrpde}) holds if $u(t,x)=h(Z_t(x))$ for all $(t,x)\in[0,T]\times \R^d$.
\end{proof}

In the next theorem, we will show that the solution is unique in the space $\cC^{\alpha, 3}_{loc}([0,T]\times \R^d)$ provided that $(W,\cW)\in \cC^{\alpha}([0,T];\cC^{\bbeta_4}(\R^d;\R^d))$ and $h\in \cC^4_{loc}(\R^d;\R)$.
\begin{theorem}
Let $(W,\cW)\in \C^{\alpha}([0,T];\cC^{\bbeta_3}(\R^d;\R^d))$, and let $\W$ be given by \eqref{ddw}. Assume Hypothesis \ref{hyp1}. Let $h\in \cC^4_{loc}(\R^d;\R)$. The solution to the RPDE \eqref{rpde} exists and is unique in the space $\cC^{\alpha, 3}_{loc}([0,T]\times \R^d;\R)$.
\end{theorem}
\begin{proof}
Firstly, we show the existence of the equation \eqref{rpde} in the space $\cC^{\alpha,3}_{loc}([0,T]\times \R^d;\R^d)$. Due to Theorem \ref{thmrpde}, it suffice to show that $h(Z)\in\cC^{\alpha,3}_{loc}([0,T]\times \R^d;\R^d)$.

Notice that $DZ_t(x)=M_t(Z_t(x))$, $D^2Z_t(x)=DM_t(Z_t(x))M_t(Z_t(x))$, and 
\begin{align*}
D^3 Z_t(Z_t(x))=&D^2M_t(Z_t(x))M_t(Z_t(x))M_t(Z_t(x))\\
&+DM_t(Z_t(x))DM_t(Z_t(x))M_t(Z_t(x))
\end{align*}
for all $(t,x)\in [0,T]\times \R^d$. Fix $x\in \R^d$, the functions $M_t(x)$, $DM_t(x)$, $D^2M_t(x)$ and $D^3M_t(x)$ are all solutions to corresponding linear RDEs driven by $\alpha$-H\"{o}lder linear rough paths. Thus $M_t(x)$, $DM_t(x)$, $D^2M_t(x)$ and $D^3M_T(x)$ are all $\alpha$-H\"{o}lder in time and locally bounded in space. Recall that $h\in \cC^{4}_{loc}(\R^d;\R)$. As a consequence $h(Z_t(x))$, $D[h(Z_t(x))]$, $D^2[h(Z_t(x))]$ and $D^3[h(Z_t(x))]$ are all $\alpha$-H\"{o}lder in time and locally bounded in space. In other words, we can conclude that $h(Z)\in\cC^{\alpha,3}_{loc}([0,T]\times \R^d;\R^d)$.

In the next step, we will prove the uniqueness of the RPDE \eqref{rpde}. Suppose that $u\in \cC^{\alpha, 3}([0,T]\times \R^d; \R^d)$ is a solution to \eqref{rpde}. Let $Y$ be the solution to the RDE \eqref{rdef}. Then, by Taylor's theorem, we can write
\begin{align}\label{uyst}
u(t,Y_t(x))-u(s,Y_s(x))=&u_{s,t}(Y_s(x))+Du_{s,t}(Y_s(x))Y_{s,t}+Du_s(Y_s(x))Y_{s,t}(x)\nonumber\\
&+\frac{1}{2}D^2u_s(Y_s(x))Y_{s,t}(x)^{\otimes 2}+O(|t-s|^{3\alpha}).
\end{align}
Notice that as a solution to \eqref{rpde}, $u$ satisfies the following equality for all $x\in \R^d$,
\[
u_{s,t}(x)=-Du_s(x) W_{s,t}(x)+O(|t-s|^{2\alpha}).
\]
It follows that fix $x\in\R^d$, $u(x)$ is controlled by $W(x)$. As a consequence, $Du(x)$ is also controlled by $W(x)$ with the Gubinelli derivative $-D^2u_s(x)$. Therefore, the following estimate holds
\begin{align}\label{ustys}
u_{s,t}&(Y_s(x))=-Du(s,Y_s(x))W_{s,t}(Y_s(x))+D^2u(s,Y_s(x))\cW_{s,t}(Y_s(x),Y_s(x))\nonumber\\
&+\frac{1}{2}Du(s,x)\big[\llg DW(Y_s(x)),W(Y_s(x))\rrg_{s,t}+\llg W(Y_s(x)),DW(Y_s(x))\rrg_{s,t}\big]\nonumber\\
&+\frac{1}{2}D^2u(s,x)\lgl DW(Y_s(x))\rgl_{s,t}+O(|t-s|^{3\alpha}).
\end{align}
In addition, recall that $Y$ is the solution to \eqref{rdef}. Then, \eqref{ustys} implies that
\begin{align}
Du_{s,t}(Y_s(x))Y_{s,t}(x)=&-D^2u(s,Y_s(x))W_{s,t}(Y_s(x))W_{s,t}(Y_s(x))\nonumber\\
&-Du(s,Y_s(x))DW_{s,t}(Y_s(x))W_{s,t}(Y_s(x))O(|t-s|^{3\alpha}).
\end{align}
Also, we have the following estimates
\begin{align}
Du_s(Y_s(x))Y_{s,t}(x)=&Du_s(Y_s(x))W_{s,t}(Y_s(x))\nonumber\\
&+Du_s(Y_s(x))\W_{s,t}(Y_s(x),Y_s(x))+O(|t-s|^{3\alpha}),
\end{align}
and
\begin{align}\label{d2usyst}
D^2u_s(Y_s(x))Y_{s,t}(x)^2=D^2u_s(Y_s(x))W_{s,t}(Y_s(x))^{\otimes 2}+O(|t-s|^{3\alpha}).
\end{align}
Combining \eqref{uyst} - \eqref{d2usyst}, we have
\[
u(t,Y_t(x))-u(s,Y_s(x))=O(|t-s|^{3\alpha}).
\]
Because $\alpha\in (\frac{1}{3},\frac{1}{2}]$, it follows that $u(t,Y_t(x))\equiv u(0, Y_0(x))=h(x)$. In other words, $u(t,x)=u(t,Y_t(Z_t(x)))=h(Z_t(x))$ for all $(t,x)\in [0,T]\times \R^d$.
\end{proof}

\end{document}